\documentclass{amsart}

\usepackage{amsmath,amssymb}
\usepackage{amsthm}
\usepackage{amsrefs}
\usepackage{indentfirst}
\usepackage{mathrsfs}
\usepackage{hyperref}
\usepackage{xcolor}
\usepackage[margin=1.4in]{geometry}
\usepackage{nicefrac}
\usepackage{subfigure}
\usepackage{graphicx}
\usepackage{tikz}
\usepackage{wrapfig}

\newcommand{\bx}{\boldsymbol{x}}
\newcommand{\calM}{\mathcal{M}}

\newcommand{\bR}{\mathbb{R}}
\newcommand{\calP}{\mathcal{P}}
\newcommand{\calR}{\mathcal{R}}
\newcommand{\calT}{\mathcal{T}}

\newcommand{\norm}[1]{\left\| #1\right\|}
\DeclareMathOperator{\diag}{diag}

\newlength{\leftstackrelawd}
\newlength{\leftstackrelbwd}
\def\leftstackrel#1#2{\settowidth{\leftstackrelawd}%
{${{}^{#1}}$}\settowidth{\leftstackrelbwd}{$#2$}%
\addtolength{\leftstackrelawd}{-\leftstackrelbwd}%
\leavevmode\ifthenelse{\lengthtest{\leftstackrelawd>0pt}}%
{\kern-.5\leftstackrelawd}{}\mathrel{\mathop{#2}\limits^{#1}}}

\numberwithin{equation}{section}
\newtheorem{theorem}{Theorem}[section]
\newtheorem{lemma}[theorem]{Lemma}
\newtheorem{remark}[theorem]{Remark}

\newtheorem{assumption}[theorem]{Assumption}
\newtheorem{definition}[theorem]{Definition}
\newtheorem{corollary}[theorem]{Corollary}
\allowdisplaybreaks[4]

\title[The Gross-Pitaevskii eigenvalue problem]{Fully discretized Sobolev gradient flow for the Gross-Pitaevskii eigenvalue problem}
\author{Ziang Chen}
\address{(ZC) Department of Mathematics, Massachusetts Institute of Technology, 77 Massachusetts Avenue, Cambridge, MA 02139.}
\email{ziang@mit.edu}
\author{Jianfeng Lu}
\address{(JL) Departments of Mathematics, Physics, and Chemistry, Duke University, Box 90320, Durham, NC 27708.}
\email{jianfeng@math.duke.edu}
\author{Yulong Lu}
\address{(YL) School of Mathematics, University of Minnesota, 206 Church Street SE, Minneapolis, MN 55455.}
\email{yulonglu@umn.edu}
\author{Xiangxiong Zhang}
\address{(XZ) Department of Mathematics, Purdue University, 150 N. University Street, West Lafayette, IN 47907.}
\email{zhan1966@purdue.edu}
\date{\today}

\thanks{The work of ZC and JL is supported in part by National Science Foundation via awards DMS-2012286 and DMS-2309378. YL thanks the support from the National Science Foundation through the award DMS-2343135 and the support from the Data Science Initiative at University of Minnesota through a MnDRIVE DSI Seed Grant.  XZ is supported by NSF  DMS-2208518. XZ is grateful to Prof. Weizhu Bao and Prof. Yongyong Cai for discussions of discrete gradient flows.}

\begin{document}
\begin{abstract}
    This paper studies the numerical approximation of the ground state of the Gross-Pitaevskii (GP) eigenvalue problem with a fully discretized Sobolev gradient flow induced by the $H^1$ norm. For the spatial discretization, we consider the finite element method with quadrature using $P^k$ basis on a simplicial mesh and $Q^k$ basis on a rectangular mesh. We  prove the global convergence to a critical point of the discrete GP energy, and establish a local exponential convergence to the ground state under the assumption that the linearized discrete Schr\"odinger operator has a positive spectral gap. We also show that for the $P^1$ finite element discretization with quadrature on an unstructured shape regular simplicial mesh, the eigengap satisfies a mesh-independent lower bound, which implies a mesh-independent local convergence rate for the proposed discrete gradient flow. Numerical experiments with discretization by high order $Q^k$ spectral element methods in two and three dimensions are provided to validate the efficiency of the proposed method.
\end{abstract}

\maketitle

\section{Introduction} 
\subsection{The Gross-Pitaevskii eigenvalue problem} 

A standard mathematical model of the equilibrium states in Bose–Einstein condensation (BEC) \cites{bose1924plancks,einstein1925quantentheorie,dalfovo1999theory,pitaevskii2003bose} is through the minimization of the Gross-Pitaevskii energy. For $N$ identical bosons, with an scattering length $a$ and an external potential $V(\bx)$, the Gross-Pitaevskii (GP) energy functional is defined as 
\[\mathcal E^{\text{GP}}(\phi)= \int_{\mathbb R^3}\left(|\nabla \phi(\bx)|^2 +V(\bx) |\phi(\bx)|^2  +4\pi a | \phi(\bx) |^4 \right)\mathrm{d}\bx,  \]
and the GP energy, denoted by $E^{\text{GP}}(N, a)$,  is defined as the infimum of $\mathcal E^{\text{GP}}$ under normalization $\int_{\mathbb R^3}  |\phi(\bx)|^2 \mathrm{d}\bx=N$. It has been used for finding the ground state energy per unit volume of a dilute, thermodynamically infinite, homogeneous gas. In some typical experiments the value of $a$ is about $10^{-3}$, while $N$ varies from $10^3$ to $10^7$.

Computations are usually done on a finite domain and it is usually acceptable to make the assumption that we can approximate wave function of interest by compact support due to the fast decay rate at infinity \cites{gravejat2004decay, laire2022existence}. 
Let $\mathcal E_L^{\text{GP}}$ denote the energy functional defined  on a box domain $\Omega=[-L,L]^d$ and 
$E^{GP}_L(N, a)$ denote the corresponding GP energy, then $\lim\limits_{L\to \infty} E^{GP}_L(N, a)=E^{GP}(N, a)$, see \cite{lieb2001bosons}.
Since the GP energy satisfies the scaling relation $E^{GP}(N, a)=NE^{GP}(1, Na)$,   we consider the following simplified rescaled model for finding ground state: minimizing the energy functional 
\[ E(\phi)=\frac12 \int_\Omega \left(|\nabla \phi(\bx)|^2 +V(\bx) |\phi(\bx)|^2 \right) \mathrm{d}\bx+\frac{\beta}{4} \int_\Omega | \phi(\bx) |^4 \mathrm{d}\bx,\quad \Omega=[-L,L]^d, \]
over the constraint set $\left\{\phi\in H_0^1(\Omega): \int_{ \Omega}  |\phi(\bx)|^2 \mathrm{d}\bx=1\right\}$, where  $d=1,2,3$,  $V(\bx)\geq 0$ and $\beta>0$. While it is also of interest to extend the problem to $\beta < 0$, see e.g., \cite{antoine2013computational}, we restrict to the case  $\beta>0$ in this work.
	
The existence, uniqueness, and regularity of the GP ground states are well understood, see e.g., \cite{lieb2001bosons}. For $\beta>0$, $E(\phi)$ has the unique positive ground state $\phi(\bx)>0$, which is also the eigenfunction to the nonlinear eigenvalue problem 
\begin{equation}
	-\Delta u(\bx)+V(\bx)u(\bx)+\beta|u(\bx)|^2 u(\bx)=\lambda u(\bx), \quad  \int_{ \Omega}  |u(\bx)|^2 \mathrm{d}\bx=1, u(\bx)|_{\partial \Omega}=0.
	\label{continuum}
\end{equation}  
Notice that \eqref{continuum} should be understood in the sense of distribution, i.e., the variational form of \eqref{continuum} is to seek $\lambda \in \mathbb R$ and $u\in H_0^1(\Omega)$ satisfying
\begin{equation}
	( \nabla u,\nabla  v)+(V u, v)+\beta(|u|^2 u, v)=\lambda (u, v),\quad\forall ~v\in H_0^1(\Omega),
	\label{variational}
\end{equation}
where $(u, v)=\int_\Omega u(\bx) v(\bx) \mathrm{d}\bx$. Let $u^*$ be the ground state to $E(\cdot)$, then by setting $u=v=u^*$ in \eqref{variational}, the corresponding eigenvalue should satisfy
\[\lambda^*=2E(u^*)+\frac{\beta}{2}\bar \rho,\quad \bar\rho =\int_{ \Omega}  |\phi(\bx)|^4 \mathrm{d}\bx.\]
Since the ground state $u^*$ remains unchanged under a constant shift of the potential, without loss of generality, we may assume $V(\bx)\geq c>0$ for some $c>0$.  
 
\subsection{Related work}

The study of numerical solutions to the Gross-Pitaevskii problem \eqref{variational} has a long history. Self-consistent field iteration (SCF) \cites{defranceschi2000scf,cances2000convergence,cances2000can,upadhyaya2018density} is one of the most popular iterative techniques for a nonlinear eigenvalue problem, which involves a linearized eigenvalue problem during each iteration. For the problem \eqref{continuum}, SCF may diverge unless a good initial guess is provided. 
 
Another category of popular methods takes an optimization perspective of the energy functional. They can be viewed as discrete-in-time gradient flows (i.e., gradient descent) of the energy functional linked to \eqref{variational}. Earlier works in this category are based on an implicit Euler discretization of the $L^2$-gradient flow \cites{bao2004computing,bao2003numerical,bao2003ground}. More recently, several alternative gradient flows have been proposed by modifying the underlying metric, including the projected Sobolev gradient flow \cites{danaila2010new, kazemi2010minimizing, danaila2017computation, zhang2019exponential, henning2020sobolev, heid2021gradient,chen2023convergence} and the J-method \cites{jarlebring2014inverse, altmann2021j}. Projected Sobolev gradient flow is based on first computing the Sobolev gradients, which are the Riesz representation of the Fr\'echet derivative of the GP energy functional within an appropriate Hilbert space (e.g., $H^1(\Omega)$), and then projecting the gradients to the tangent space of the Riemannian manifold defined by the normalization constraint. Despite the empirical success of projected Sobolev gradient flows for solving the GP eigenvalue problem, their convergence analysis is still underdeveloped. Our work fits in this line of research. 

For the existing convergence results for gradient-flow-based methods: The work \cite{kazemi2010minimizing} established the global exponential convergence of the continuous-in-time projected $H^1$-gradient flow to a critical point of $E$. The work \cite{henning2020sobolev} obtained a global exponential convergence of a continuous projected Sobolev flow with an alternative metric to the ground state and also proved the global convergence (without a rate) of its forward Euler discretization. A more recent work \cite{zhang2019exponential} established a local exponential convergence of the discrete-in-time flow of \cite{henning2020sobolev} under the assumption that the discrete iterates are uniformly bounded. A more explicit local convergence rate depending on the eigengap of the linearized problem at the groundstate is obtained in \cite{henning2023dependency}. In our previous work \cite{chen2023convergence}, we improved the analysis of the global convergence and local rate of convergence of discrete-in-time projected Sobolev gradient flows with several common choices of inner products in $H^1$-space. 
 
In addition to the time-discretization of the projected Sobolev gradient flows, their spatial discretization \cite{bao2004computing} is of course necessary for the practical implementation of the schemes. However, most of the prior theoretical work on projected Sobolev flows for the GP eigenvalue problem do not consider spatial discretization and it remains open how to extend the convergence analysis to the fully discretized setting. 
The convergence of the numerical solution using finite element method has been first established in \cite{cances2010numerical},
and there is also some recent progress on estimating the discretization error for energy, eigenvalue, and eigenfunction in the setting of mixed finite element method  \cite{gallistl2024mixed}, but the convergence of the fully discretized gradient flows has not been analyzed before the initial submission of this work. 
After the initial submission of this work, some very recent progress was reported in \cite{hauck2024positivity}, which extended the results in \cite{henning2020sobolev} to  the fully discretized $A_u$ Sobolev gradient descent with the monotone $P^1$ finite element method. Compared to \cite{hauck2024positivity}, we consider the fully discretized $H^1$ Sobolev gradient descent and we also prove locally exponential convergence rate that is not covered in \cite{hauck2024positivity}.

Let us also mention some works on numerical analysis for general nonlinear eigenvalue problems, where $\frac{\beta}{4} \int_\Omega | \phi(\bx) |^4 \mathrm{d}\bx$ in the energy functional of the GP eigenvalue problem is generalized to $\frac{1}{2}\int_\Omega F(|\phi(\bx)|^2)\mathrm{d}\bx$; we refer interested readers to \cites{cances2010numerical,cances2018two,dusson2023overview}.

\subsection{Contribution of the present work}
We summarize our major contribution as follows. 
\begin{itemize}
    \item We propose a fully discretized Sobolev gradient descent for approximating the ground state of the GP energy, which can be viewed as a Riemannian gradient descent method on the sphere under a metric induced by a modified $H^1$-norm. 
    \item We prove the global convergence of the fully discretized Sobolev gradient descent with respect to the modified $H^1$ metric to a critical point of the discrete GP energy and a local convergence to the ground state with an exponential rate. See Corollary~\ref{cor:global_converge} and  Theorem~\ref{thm:local_converge}. We   prove the convergence of the ground eigenpair of the discrete GP energy to those of the continuous counterpart as well as a positive discrete eigengap as the mesh size diminishes for $P^1$ finite element method with quadrature on unstructured shape regular simplicial meshes under the classical mesh constraint for monotonicity, which includes the  second order finite difference scheme; see Theorem~\ref{thm:consist} and Theorem~\ref{thm:positive_eigengap}. 
    \item We provide numerical experiments with $Q^k$ spectral element method as spatial discretization to verify the accuracy and efficiency of the proposed approach for solving  GP problems in both two and three dimensions. Due to the fact that only Laplacian needs to be inverted in the algorithm, the scheme on a structured mesh can be easily implemented and efficiently accelerated on modern GPUs.
\end{itemize}

\subsection{Organization of the rest of the paper} 
As preliminaries, we first discuss the spatial discretization of the GP eigenvalue problem based on finite element method in   Section \ref{sec-fem}, then review some useful properties of the discrete energy in Section \ref{sec:discretep}. In particular, the second order finite element method on a uniform gives the most popular second-order finite difference scheme. The fully discretized Sobolev gradient descent methods are given in Section \ref{sec:discretescheme}. We present the two main convergence results in Section \ref{sec:convergence}, with numerical experiments given in Section \ref{sec-tests}. Further preliminary results and proof details can be found in the Appendix. Concluding remarks are given in Section \ref{sec:remarks}.

\section{The classical finite element method with quadrature}
\label{sec-fem}
We consider the classical continuous finite element method with quadrature using $P^k$ basis on a simplicial mesh or $Q^k$ basis on a rectangular mesh. This section briefly reviews its definition. If the fnite element method is defined on a uniform structured mesh,  it is well known that both the $P^1$ scheme and the $Q^1$ scheme are equivalent to the second-order finite difference scheme.  
	
\subsection{Finite element Galerkin method}
\label{sec-derivation}
	
We first consider a uniform rectangular mesh $\Omega_h$ for the rectangular domain $\Omega$. For any rectangle $e$ in the mesh $\Omega_h$, let $Q^k$ be the space of tensor product polynomials of degree $k$:
$$Q^k(e)=\Bigg\{ p(\bx)=\sum\limits_{i_1,i_2,\dots,i_d=0}^k p_{i_1 i_2\cdots i_d}x_1^{i_1} x_2^{i_2}\cdots x_d^{i_d},\ \bx = (x_1,x_2,\dots, x_d)\in e \Bigg\}.$$ 
Let $V^h_0\subset H_0^1(\Omega)$ be  continuous piecewise $Q^k$ polynomial space  with zero boundary:
\[V^h_0=\{ v_h(\bx) \in C(\Omega): v_h(\bx)|_{\partial \Omega}=0, \ v_h \big|_e\in Q^k(e),\ \forall~e \in \Omega_h\}\subset H_0^1(\Omega).\] We also consider an unstructured simplicial mesh  $\Omega_h$ with $e$ denoting a simplex in $\Omega_h$, e.g., a triangular mesh in two dimensions with $e$ denoting a triangle. Let $P^k$ be the space of polynomials of degree $k$:
\[P^k(e)=\Bigg\{ p(\bx)=\sum\limits_{  i_1+i_2+\dots+i_d\leq k} p_{i_1 i_2\cdots i_d}x_1^{i_1} x_2^{i_2}\cdots x_d^{i_d},\ \bx = (x_1,x_2,\dots, x_d)\in e \Bigg\},\]
and $V^h_0\subset H_0^1(\Omega)$ be  continuous piecewise $P^k$ polynomial space  with zero boundary:
\[V^h_0=\{ v_h(\bx) \in C(\Omega): v_h(\bx)|_{\partial \Omega}=0, \ v_h \big|_e\in P^k(e),\ \forall~e \in \Omega_h\}\subset H_0^1(\Omega).\]

The finite element Galerkin method for \eqref{continuum} is to seek $\lambda_h \in\mathbb R$ and $u_h\in V_0^h$ satisfying
\begin{equation}
	( \nabla u_h,\nabla  v_h)+(V u_h, v_h)+\beta(|u_h|^2 u_h, v_h)=\lambda_h (u_h, v_h),\quad\forall~v_h\in V^h_0.
	\label{fem}
\end{equation} 
The corresponding discrete energy can be given as
\begin{equation*}
    E(u_h)  =\frac12(\nabla u_h, \nabla u_h)+\frac12(V u_h, u_h)+\frac{\beta}{4}(u_h^2, u_h^2).
\end{equation*}

The convergence of the finite element method for the nonlinear eigenvalue problem \eqref{continuum} was discussed in \cite{cances2010numerical}.
The standard {\it a priori} error estimates for a linear eigenvalue problem, e.g., $\beta=0$ in \eqref{fem} is $2k$-order for eigenvalues and $k$-order for eigenvector in $H^1$-norm, under suitable regularity assumptions. See \cites{ciarlet1968numerical, babuvska1987estimates, babuvska1989finite, BABUSKA1991641, knyazev2006new} for discussions on the rate of convergence of numerical schemes for eigenvalue problems.

\subsection{Finite element method with quadrature}
In practice, one often uses quadrature for integrals to implement the finite element method. The $Q^k$ spectral element method is to replace all integrals in \eqref{fem} by $(k+1)$-point Gauss-Lobatto quadrature in each dimension. Standard {\it a priori} finite element method error estimates still hold, see \cite{ciarletbook} and references therein. 
For the $P^k$ finite element method, suitable quadrature on a simplex can be used, e.g., the simplest quadrature using the average of values at vertices can be used for the $P^1$ finite element method.
Let $\langle\cdot,\cdot\rangle$ denote that integrals are replaced by   quadrature, then the method is to find $u_h\in V^h$ satisfying
\begin{equation}
	\langle \nabla u_h,\nabla  v_h\rangle+\langle V u_h, v_h\rangle+\beta \langle |u_h|^2 u_h, v_h\rangle=\lambda_h \langle u_h, v_h\rangle,\quad\forall~v_h\in V^h_0.
	\label{fem2}
\end{equation}
The corresponding discrete energy is given as
\begin{equation} 
    E_h(u_h)=\frac12 \langle \nabla u_h, \nabla u_h\rangle+\frac12 \langle V u_h, u_h\rangle+\frac{\beta}{4}\langle u_h^2, u_h^2\rangle. \label{fem-energy}
\end{equation}

\subsection{The matrix-vector form}

For either $Q^k$ or $P^k$ finite element method, we will use a quadrature rule such that a matrix vector form of the scheme can be easily written.

We first describe the $Q^k$ finite element method on a uniform rectangular mesh.
Assume that $\Omega_h$ consists of uniform $N_c^d$ cubic cells for the cubic domain $\Omega=[-L, L]^d$. Then there are in total $(N_ck+1)^d$  Gauss-Lobatto points. Any $Q^k$ polynomial on a cubic element $e$ can be represented as a Lagrangian interpolation polynomial at $(k+1)^d$ Gauss-Lobatto points, thus the $Q^k$ spectral element method \eqref{fem2} also becomes a finite difference scheme on all Gauss-Lobatto nodes. For $Q^1$ and $Q^2$ bases, all the Gauss-Lobatto points form a uniform grid.  For $k\geq 3$, the Gauss-Lobatto points are not uniform in each element. 	
For homogeneous Dirichlet boundary condition, the boundary points are not unknowns. Thus the total number of unknowns is the interior grid points with the number $N=n^d$ where $n=N_ck-1$. To derive an equivalent matrix form of the scheme \eqref{fem2}, let $\phi_i(\bx)\in V_0^h$ ($i=1,\cdots,N$) be the continuous piecewise $Q^k$ Lagrangian basis at all Gauss-Lobatto points $\bx_i$ ($i=1,\cdots,N$) in the interior of $\Omega_h$. For any piecewise polynomial $u_h(\bx)\in V^h_0$, let $u_i=u_h(\bx_i)$. Then $u_h(\bx)=\sum_{i=1}^N u_i\phi_i(\bx)$. Let $\mathbf u=\begin{bmatrix} u_1& \cdots &u_N \end{bmatrix}^\top$ and $w_i$ be the quadrature weight at $\bx_i$. 

Next we consider the $P^k$ finite element method on a simplicial mesh. The number of degree of freedoms of $P^k$ polynomial on a $d$-dimensional simplex is $\frac{(k+d)!}{k!d!}$.  Consider a $(k+1)$-th order accurate quadrature rule on a simplex using $\frac{(k+d)!}{k!d!}$ quadrature points, e.g., the quadrature rule using $3$ vertices for $P^1$ on a triangle, the quadrature rule using $6$ vertices for $P^1$ on a tetrahedron and the quadrature rule using $3$ vertices and $3$ edge centers for $P^2$ on a triangle. Assume there are in total $N$ quadrature points which lie in the interior of the domain $\Omega_h$. Let $\bx_i$ ($i=1,\cdots,N$) be all interior quadrature points with $w_i$ being the quadrature weight at $\bx_i$.
Let $\phi_i(\bx)\in V_0^h$ ($i=1,\cdots,N$) be the continuous piecewise $P^k$ Lagrangian basis at all interior quadrature points $\bx_i$ ($i=1,\cdots,N$). For any piecewise polynomial $u_h(\bx)\in V^h_0$, let $u_i=u_h(\bx_i)$. Then $u_h(\bx)=\sum_{i=1}^N u_i\phi_i(\bx)$. Let $\mathbf u=\begin{bmatrix} u_1& \cdots &u_N \end{bmatrix}^\top$.

With the notation above, we have
\begin{equation} \label{part1}
	\langle V u_h, v_h\rangle =\sum_{i=1}^N w_i V_i u_i v_i =\mathbf v^\top \mathbb M   \mathbb V \mathbf u,
\end{equation}
where $\mathbb M=\text{diag}\{w_1, \cdots ,w_N\}$ and $\mathbb  V=\text{diag}\{V_1,\cdots,   V_N\}$ are diagonal matrices and $V_i=V(\bx_i)$. We also have
\begin{equation} \label{part2}
	\langle  \nabla  u_h,\nabla  v_h\rangle= \mathbf v^\top \mathbb S \mathbf u,
\end{equation}
where $\mathbb S$ is the stiffness matrix given by $\mathbb S_{ij}=\langle \nabla \phi_i, \nabla \phi_j \rangle$.

 Using \eqref{part1} and \eqref{part2}, the matrix form of (\ref{fem2}) is to find $\mathbf u \in \mathbb R^N$ satisfying
\begin{equation*}
	\mathbf v^\top \mathbb S \mathbf u+\mathbf v^\top \mathbb M   \mathbb V \mathbf u+\beta \mathbf v^\top \mathbb M  \mathbf u^3=\lambda_h \mathbf v^\top \mathbb M   \mathbf u,\quad\forall~\mathbf v\in \mathbb R^N,
\end{equation*}
or equivalently
\begin{equation}
	\mathbb S \mathbf u+ \mathbb M   \mathbb V \mathbf u+\beta  \mathbb M  \mathbf u^3=\lambda_h   \mathbb M   \mathbf u,
	\label{fd3}
\end{equation} 
where $\mathbf u^3=\begin{bmatrix} u_1^3 & \cdots & u_N^3 \end{bmatrix}^\top$. Let $\Delta_h=-\mathbb M^{-1}\mathbb S$, then  \eqref{fd3} can also be written as
\begin{equation}
    -\Delta_h \mathbf u+    \mathbb V \mathbf u+\beta  \mathbf u^3=\lambda_h   \mathbf u.
    \label{fd2}
\end{equation}
In the formulation \eqref{fd3}, all the matrices are symmetric positive definite, but the discrete Laplacian $\Delta_h=-\mathbb M^{-1}\mathbb S$ is in general not symmetric in \eqref{fd2} except the special case of second order finite difference, which however does not affect numerical implementations since the symmetric form \eqref{fd3} should be implemented instead of the form \eqref{fd2}.
	
\subsection{The discrete energy and discrete $L^2$ norm}
	
Using the same notation, the discrete energy \eqref{fem-energy} can be written as
\begin{equation}
	E_h(u_h)=\frac12 \mathbf u^\top \mathbb S \mathbf u+\frac12 \mathbf u^\top \mathbb M \mathbb V \mathbf u+\frac{\beta}{4} (\mathbf u^2)^\top \mathbb M  \mathbf u^2.
	\label{fem_matrix-energy}
\end{equation}
Introduce the interpolation operator 
\begin{equation}
	\label{interpolation}
	\Pi: \mathbb R^N\longrightarrow V_0^h,\quad\mathbf v \longmapsto \sum_{i=1}^N v_i \phi_i(\bx).
\end{equation}
The discrete integration by parts is ensured in the following sense:
\begin{equation}
	\langle \nabla u_h, \nabla v_h\rangle=\mathbf v^\top \mathbb S \mathbf u=\mathbf v^\top \mathbb M (-\Delta_h) \mathbf u=\langle  -\Pi[\Delta_h \mathbf u],  v_h\rangle. 
    \label{integrationbyparts}
\end{equation} 
For two vectors $\mathbf u, \mathbf v\in\mathbb R^N$, we define the discrete $L^2$ inner product $\langle \mathbf u, \mathbf v\rangle_h$ by setting  
\begin{equation}
	\label{discreteL2norm}
	\langle \mathbf u, \mathbf v\rangle_h :=\mathbf u^\top \mathbb M\mathbf v.
\end{equation}
Thus the discrete energy \eqref{fem_matrix-energy} and \eqref{fem-energy} can also be written in matrix form 
\begin{equation*}
	E_h(u_h) = \mathbf E_h(\mathbf u) = \frac12 \langle-\Delta_h \mathbf u,  \mathbf u\rangle_h+\frac12  \langle\mathbb V \mathbf u,  \mathbf u\rangle_h+\frac{\beta}{4}  \langle \mathbf u^2,  \mathbf u^2\rangle_h , \label{fd-energy} 
\end{equation*}
with normalization constraint $\langle \mathbf u, \mathbf u\rangle_h=1$.

\section{Properties of the discrete energy}
\label{sec:discretep}

We discuss some properties of the discrete energy $E_h(u_h) = \mathbf E_h(\mathbf u)$ in this section.
	
\subsection{The monotonicity of the discrete Laplacian}
\label{sec:monotonicity}	
A matrix $A\in \mathbb R^{n\times n}$ is called {\it monotone} if its inverse has nonnegative entries $A^{-1}\geq 0$. At a fixed vector $\mathbf u$, the linearized operator for \eqref{fd2} is given by 
\begin{equation}\label{eq:A_u}
	A_{\mathbf u}=-\Delta_h+\mathbb V+\beta \diag(\mathbf u^2).
\end{equation}
The matrix $A_{\mathbf u}$ is {\it irreducible}, which can be easily verified by the graph that the discrete Laplacian represents, see \cite{li2019monotonicity}. The definition of irreducible matrices is given in Appendix \ref{sec-PF}. By the Perron Frobenius Theorem in Appendix \ref{sec-PF}, if $A_{\mathbf u}$ is also monotone, then its smallest eigenvalue has multiplicity one, with a unique unit positive eigenvector. For the $Q^1$ finite element scheme on a uniform rectangular mesh (or equivalently the second order finite difference scheme) or the $P^1$ finite element method on a simplical mesh under suitable constraints, is used, it is straightforward to verify that $A_{\mathbf u}$  satisfies Theorem \ref{rowsumcondition-thm}, implying that  $A_{\mathbf u}$ is an M-matrix thus monotone, which is a well-known result in the literature; see Appendix \ref{sec-Mmatrix}. 
The explicit expressions of the second-order finite difference and $P^1$ finite element method on an unstructured mesh are given in Appendix \ref{apx:2nd_finite_diff_explicit}.

For a simplex $T$ in a simplicial mesh $\Omega_h\subset\mathbb R^d$ of dimension $d$, let $\kappa^T_E$ be the $(d-2)$-dimensional simplex opposite to the edge $E$ in the simplex $T$ and $\theta_E^T$ be the angle between the two faces containing the edge $E$ in the simplex $T$. By {xu1999monotone}*{Lemma 2.1}, the simplicial mesh constraint for monotonicity is 
\begin{equation}
    \label{simpicialmesh}
    \sum_{T\supset E} \frac{1}{d(d-1)}|\kappa_E^T|\cot \theta_E^T\geq 0,
\end{equation}
 where $T\supset E$ means summation over all simplexes $T$ containing the edge $E$. Such a constraint reduces to a Delaunay triangular mesh in two dimensions (see Appendix \ref{appendix-P1FEM}), which is more general and more practical than a non-obtuse
triangulation.
We summarize the results as the following theorem:
\begin{theorem}\label{thm:2nd-monotone}
    For the $P^1$ finite element scheme with quadrature on a simplicial mesh satisfying \eqref{simpicialmesh},  which includes the classical second order finite difference scheme, 
 $A_{\mathbf u}=-\Delta_h+\mathbb V+\beta \diag(\mathbf u^2)$ is an M-matrix thus monotone. As a result, it has a unique positive unit eigenvector and the corresponding eigenvalue is simple and the smallest eigenvalue of $A_{\mathbf u}$.
\end{theorem}
	
For the high-order accurate discrete Laplacian, the matrix $-\Delta_h+\mathbb V$ is no longer an M-matrix. It is proven in \cite{li2019monotonicity} that the fourth-order accurate Laplacian of $Q^2$ scheme in two dimensions are products of M-matrices thus still monotone under certain mesh size constraints. It is possible to prove similar results for the three-dimensional case following the same arguments in \cite{li2019monotonicity}. Extensions to quasi-uniform meshes are given in \cite{cross-Q2}. It is also possible to extend the monotonicity to $Q^3$ element \cite{cross2020monotonicity}. All these monotonicity results for high-order schemes hold under mesh size constraints, which makes further discussion of global convergence much more complicated. Thus we only discuss the global convergence for the second-order scheme.

\subsection{Ground state of the discrete energy}
	
In this subsection, we only focus on $P^1$ finite element scheme on a simplicial mesh satisfying  \eqref{simpicialmesh}, which includes the second order finite difference scheme. In general, it is difficult to extend all results about the discrete energy \eqref{fd-energy} to high-order schemes.

\begin{theorem}
    \label{theorem-discrete-energy}
    For the $P^1$ finite element method with quadrature weights $w_i$ on a simplicial mesh satisfying \eqref{simpicialmesh}, for any $ v_h\in V_0^h$ satisfying $\sum_{i=1}^N w_i v_i=1$ and $v_i\geq 0,\ \forall~i$, where $v_i = v_h({\bf x}_i)$, $\mathbf E_h(\sqrt{\mathbf v})$ is strongly convex w.r.t. the vector $\mathbf v =\begin{bmatrix} v_1 & \cdots & v_N \end{bmatrix}^\top $.
\end{theorem}
	
\begin{proof}
    Let $u_h\in V_0^h$ satisfy $u_h({\bf x}_i)=\sqrt{v_i}=\sqrt{v_h({\bf x}_i)}$. By \eqref{fem-energy}, we have
    \[\mathbf E_h(\sqrt{\mathbf v})=\frac{1}{2}\langle \nabla u_h, \nabla u_h\rangle+\frac{1}{2}\langle V u_h, u_h\rangle+\frac{\beta}{4}\langle u_h^2, u_h^2\rangle.\]

The quadrature yields that $\frac{\beta}{4}\langle u_h^2, u_h^2\rangle=\frac{\beta}{4}   \sum_i w_i v_i^2$ is quadratic and strongly convex in $\mathbf v$ and that $\langle V u_h, u_h\rangle=\sum_i w_i V_i v_i$ is linear in $\mathbf v$.
      By \eqref{discretegradient-P1FEM},   we have
      \[\langle \nabla u_h, \nabla u_h\rangle=\sum\limits_E\left(\sum_{T\supset E} \frac{1}{d(d-1)}|\kappa_E^T|\cot \theta_E^T \right) |\sqrt{v_i}-\sqrt{v_j}|^2, \quad \mbox{the edge $E$ connects $\bx_i$ and $\bx_j$.}\]
       With \eqref{simpicialmesh}, the convexity of the term $\langle \nabla u_h, \nabla u_h\rangle$ is induced by the convexity of the bivariate function $f(x,y)=|\sqrt{x}-\sqrt{y}|^2$, which is easy to verify.
\end{proof}

\begin{theorem}
	\label{theorem-discrete-energy-2}
	For the $P^1$ finite element method with quadrature on a simplicial mesh satisfying \eqref{simpicialmesh},  $\forall u_h\in V_0^h$, $\mathbf E_h(\mathbf u)\geq \mathbf E_h(|\mathbf u|)$.
\end{theorem}
	
\begin{proof}
	It suffices to verify that  $\langle \nabla  {u_h}, \nabla  {u_h}\rangle\geq \langle \nabla |u_h|, \nabla |u_h|\rangle$.
    By \eqref{discretegradient-P1FEM},   we have
      \[\langle \nabla u_h, \nabla u_h\rangle=\sum\limits_E\left(\sum_{T\supset E} \frac{1}{d(d-1)}|\kappa_E^T|\cot \theta_E^T \right) |u_i-u_j|^2, \quad \mbox{the edge $E$ connects $\bx_i$ and $\bx_j$.}\]
With \eqref{simpicialmesh}, it suffices to verify $(x-y)^2\geq (|x|-|y|)^2$, which is trivial.
\end{proof}

\begin{theorem}
	\label{theorem-discrete-energy-3}
    For the $P^1$ finite element method with quadrature on a simplicial mesh satisfying \eqref{simpicialmesh},
	the discrete energy $E_h(u_h)$ under the constraint $\langle u_h, u_h\rangle=1$   has a unique and positive minimizer $u_h^*$. Let $\mathbf u^*$ be the vector representing its point values $u_h^*(\bx_i)$, then $\mathbf u^*$ solves \eqref{fd2}, and $\mathbf u^*$ is the eigenvector associated to the smallest eigenvalue of the linear operator $A_{\mathbf u^*}=-\Delta_h+\mathbb V+\beta \diag({\mathbf u}^*)^2.$
\end{theorem} 
	
\begin{proof}
	Strong convexity over a convex constraint in Theorem \ref{theorem-discrete-energy} gives the existence and uniqueness of the minimizer $u_h^*$.  Theorem \ref{theorem-discrete-energy-2}  implies that $u^*_h(\bx_i)\geq 0$. For minimizing $E_h(u_h)$ with $\langle u_h, u_h\rangle=1$, or equivalently minimizing $\mathbf E_h(\mathbf u)$ with $ \mathbf u^T\mathbb M  \mathbf u=1$, the Lagrangian for the constrained minimization is given by $L(u_h, \lambda_h)=E_h(u_h)-\lambda_h(\langle u_h, u_h\rangle-1)$. The minimizer must satisfy the critical point equation  $\frac{\delta L}{\delta u_h}=0$, thus $u^*_h$ satisfies \eqref{fem2}, or equivalently, $\mathbf u$ satisfies \eqref{fd2}.  By Theorem \ref{thm:2nd-monotone},  the matrix $A_{\mathbf u}=-\Delta_h+\mathbb V+\beta \diag(\mathbf u^2)$ is monotone. By Perron Frobenius Theorem (Theorem \ref{Perron-Frobenius}), such a monotone matrix $A_{\mathbf u^*}$ has a unique positive unit eigenvector associated with its smallest eigenvalue. Since $\mathbf u^*\geq 0$, it is the unique unit eigenvector to the smallest eigenvalue of the linear operator $A_{\mathbf u^*}$.
\end{proof}

\begin{remark}
	Neither Theorem  \ref{theorem-discrete-energy} nor Theorem \ref{theorem-discrete-energy-2} can be extended to $Q^2$ finite element method. Nonetheless, the monotonicity of the fourth-order scheme may hold if the mesh size is very small, which ensures that the ground state of the nonlinear eigenvalue problem gives rise to a positive eigenvector by Perron-Frobenius Theorem.
\end{remark}

\section{fully discretized Sobolev gradient descent and modified $H^1$ scheme}
\label{sec:discretescheme}

We define the schemes for minimizing the discrete energy \eqref{fem_matrix-energy} or \eqref{fd-energy} associated with the $P^k$ or $Q^k$ finite element method with suitable quadrature, under the normalization constraint:
\begin{equation}\label{manifold}
	\calM = \left\{ u_h \in V_0^h : \langle u_h, u_h\rangle = 1\right\} = \left\{\mathbf u\in \mathbb R^N: \langle \mathbf u, \mathbf u\rangle_h=\mathbf u^\top \mathbb M\mathbf u = 1\right\},
\end{equation}
where $u_h=\Pi(\mathbf u)$ with $\Pi$ being the interpolation operator defined in \eqref{interpolation}. The tangent space of the manifold $\mathcal M$  at $u_h$ or $\mathbf u$ is
\begin{equation*}
	\calT_{u_h}\calM = \calT_{\mathbf u}\calM = \left\{v_h \in V_0^h: \langle u_h, v_h\rangle = 0\right\} = \left\{\mathbf v\in \mathbb R^N: \langle \mathbf u, \mathbf v\rangle_h = 0\right\}.
\end{equation*}
	
\subsection{The Sobolev gradient descent}
The gradient of the energy $E_h(u_h)$, say $\nabla E_h(u_h)\in V_0^h$, should be understood in the sense of the Fr\'echet derivative in the space $V_0^h$, and can  be computed by 
\begin{align*}
    \langle \nabla E_h(u_h), v_h\rangle & =  \lim_{t\to 0}\frac{E_h(u_h+tv_h)-E_h(u_h)}{t}\\
	& = \mathbf v^\top \mathbb S \mathbf u^\top+\mathbf v^\top\mathbb M\mathbb V\mathbf u+\beta \mathbf v^\top\mathbb M\mathbf u^3=\langle \nabla u_h, \nabla v_h\rangle+\langle   V u_h,   v_h\rangle+\beta\langle   \Pi(\mathbf u^3),   v_h\rangle,
\end{align*}
for all $v_h\in V_0^h$. With the discrete integration by parts \eqref{integrationbyparts}, we get 
\begin{equation*}
	\nabla E_h(u_h)=-\Pi[\Delta_h \mathbf u] +\Pi(\mathbb V\mathbf u)+\beta \Pi(\mathbf u^3).
\end{equation*}
Similarly, since $\langle \mathbf u, \mathbf v\rangle_h=\mathbf u^\top \mathbb M\mathbf v$ as in \eqref{discreteL2norm}, the gradient $\nabla \mathbf E_h(\mathbf u)\in \mathbb R^N$ is given by
\[\nabla \mathbf E_h(\mathbf u)=\mathbb M^{-1}( \mathbb S \mathbf u+\mathbb M\mathbb V\mathbf u+\beta \mathbb M\mathbf u^3)=(-\Delta_h+\mathbb V)\mathbf u+\beta \mathbf u^3 = A_{\mathbf u} \mathbf u,\]
where $A_{\mathbf u}$ is defined in \eqref{eq:A_u}. Thus the two Fr\'echet derivatives $\nabla E_h(u_h)$ and $\nabla \mathbf E_h(\mathbf u)$ are also identical in the sense that $\nabla E_h(u_h)=\Pi[\nabla \mathbf E_h(\mathbf u)]$.

Given any inner product $\langle \cdot,\cdot\rangle_X$ on $\mathbb R^N$, one can equip the manifold $\mathcal M\subset \mathbb R^N$ defined in \eqref{manifold} with an Riemannian metric $g(\mathbf u, \mathbf v) = \langle\mathbf u,\mathbf v\rangle_X$. Let $\mathbb G_X\in\mathbb R^{N\times N}$ be the positive definite matrix satisfying
\begin{equation*}
	\langle \mathbf u,\mathbb G_X \mathbf v\rangle_X = \langle \mathbf u,\mathbf v\rangle_h,\quad\forall~\mathbf u,\mathbf v\in\mathbb R^N.
\end{equation*}
The {\it Riemannian gradient} of $\mathbf E_h(\mathbf u)$ at $\mathbf u\in \mathcal M$ is defined as $\nabla_X^\calR \mathbf E_h(\mathbf u)\in \mathcal T_{\mathbf u}\mathcal M$ satisfying
\begin{equation*}
	g\big(\nabla_X^\calR \mathbf E_h(\mathbf u), \mathbf v\big)=\langle \nabla \mathbf E_h(\mathbf u), \mathbf v\rangle_h, \quad \forall~\mathbf v\in \mathcal T_{\mathbf u}\mathcal M. 
\end{equation*}
Following a similar derivation of Fr\'echet derivatives as before, the gradient of the discrete energy with respect to the inner product $\langle \cdot,\cdot\rangle_X$ can be computed as
\begin{equation*}
	\nabla_X \mathbf E_h(\mathbf u) = \mathbb G_X A_{\mathbf u} \mathbf u.
\end{equation*}
For any $\mathbf w\in\mathbb R^N$, the projection of $\mathbf w$ onto $\calT_{\mathbf u}\calM$ with respect to $\langle \cdot,\cdot\rangle_X$ is given by
\begin{equation*}
	\calP_{\calT_{\mathbf u}\calM, X}(\mathbf w) = \mathbf w-\frac{\langle \mathbf u,\mathbf w\rangle_h}{\langle \mathbf u, \mathbb G_X\mathbf u\rangle_h} \mathbb G_X\mathbf u.
\end{equation*}
Therefore, the Riemannian gradient is
\begin{equation}
	\nabla_X^\calR \mathbf E_h(\mathbf u) = \calP_{\calT_{\mathbf u}\calM, X} \left(\nabla_X \mathbf E_h(\mathbf u)\right) = \nabla_X \mathbf E_h(\mathbf u)-\frac{\langle \mathbf u,\nabla_X \mathbf E_h(\mathbf u)\rangle_h}{\langle\mathbb G_X\mathbf u, \mathbf u\rangle_h} \mathbb G_X\mathbf u,
\end{equation}
and the Riemannian gradient descent  of minimizing $\mathbf E_h(\mathbf u)$ over $\calM$ with step size $\tau$  is 
$$\mathbf u^{n+1} = R_h\left(\mathbf u^n - \tau \nabla_X^\calR \mathbf E_h(\mathbf u^n)\right),\quad R_h(\mathbf u) = \frac{\mathbf u} {\sqrt{\langle \mathbf u, \mathbf u \rangle}_h},$$
where $R_h(\mathbf u)$ is the retraction operator approximating the exponential map \cite{absil}.

\subsection{The modified $H^1$ scheme}
Different choices of the inner product $\langle \mathbf u,\mathbf v\rangle_X$ or the Riemannian metric $g$ would lead to different schemes. In this work, we mainly focus on the modified $H^1$-scheme. In particular, for some constant $\alpha>0$ the inner product $(\nabla u, \nabla v)+\alpha ( u,  v)$ for $H^1_0(\Omega)$ gives the following discrete inner product or metric
\begin{equation}
	\label{modifiednorm}
	g(\mathbf u, \mathbf v) =\langle \mathbf u,\mathbf v\rangle_X = \langle\nabla u_h, \nabla v_h\rangle+\alpha \langle u_h,  v_h\rangle  = \langle \mathbf u ,(-\Delta_h+\alpha \mathbb I)\mathbf v\rangle_h =\mathbf u^\top (\mathbb S+\alpha \mathbb M)\mathbf v.
\end{equation}
This induces $\mathbb G_X = (-\Delta_h+\alpha \mathbb I)^{-1}$ and
\begin{subequations}
	\label{modified-H1-scheme}
	\begin{equation}\label{eq:gradX}
		\nabla_X \mathbf E_h(\mathbf u) = \mathbb G_X A_{\mathbf u} \mathbf u =  (-\Delta_h+\alpha \mathbb  I)^{-1} (-\Delta_h+\mathbb V + \beta \text{diag}(\mathbf u)^2 )\mathbf u.
	\end{equation}
	The corresponding Riemannian gradient is 
    \begin{equation}\label{eq:Riemann_gradX}
		\nabla_X^\calR \mathbf E_h(\mathbf u)={\nabla_X E_h (\mathbf u)}-\frac{\langle \mathbf u, {\nabla_X E_h (\mathbf u)} \rangle_h}{\langle \mathbf u, (-\Delta_h+\alpha\mathbb  I)^{-1} \mathbf u \rangle_h} (-\Delta_h+\alpha \mathbb I)^{-1} \mathbf u,
	\end{equation}
    and the Riemannian gradient descent method or the Sobolev gradient flow under the modified $H^1$-norm is hence given by
    \begin{equation}
		\mathbf u^{n+1} = R_h\left(\mathbf u^n - \tau \nabla_X^\calR \mathbf E_h(\mathbf u^n)\right).
	\end{equation}	
\end{subequations}
If $\alpha=0$, then \eqref{modified-H1-scheme} is the $H^1$ gradient flow algorithm in \cite{henning2020sobolev}. There are algorithms induced by other more complicated Riemannian metrics such as $a_0$-scheme with $g_u(w,z) = (\nabla w, \nabla z)+( w, V z)$ and the $a_u$-scheme with $g_u(w,z) = (\nabla w, \nabla z)+( w, V z) + \beta(w,u^2 z)$. We refer interested readers to \cites{henning2020sobolev, zhang2019exponential, chen2023convergence}.

\section{Global and local convergence}
\label{sec:convergence}
This section proves the convergence of the modified $H^1$-scheme \eqref{modified-H1-scheme}. The theories are inspired by our prior work \cite{chen2023convergence} without spatial discretization and the main difficulty/novelty of this work is the analysis of the discrete schemes and discrete eigengap.
	
\subsection{Energy decay and global convergence}
\label{sec:global_converge}
	
We define the discrete $L^2$ norm and the $X$-norm with a fixed parameter $\alpha>0$ as follows:
\begin{align*}
    \|\mathbf u\|_{2}&=\sqrt{\langle \mathbf u, \mathbf u\rangle_h}=\sqrt{ \mathbf u^\top\mathbb M \mathbf u},\\
    \|\mathbf u\|_{X}&=\sqrt{\langle \mathbf u, \mathbf u\rangle_X}=\sqrt{\langle \mathbf u, \mathbb G_X^{-1}\mathbf u\rangle_h}=\sqrt{ \mathbf u^\top(\mathbb S+\alpha \mathbb M) \mathbf u}=\sqrt{\langle \mathbf u, (-\Delta_h+\alpha\mathbb I) \mathbf u\rangle_h}.
\end{align*}
Note that we have omitted the dependence of $X$-norm on $\alpha$ in the above. The main theorem in this subsection is stated as follows, which quantitatively characterizes the energy decay property of the Sobolev gradient flow under the modified $H^1$-norm \eqref{modified-H1-scheme}.

\begin{assumption}\label{asp:V}
    The potential energy $V$ satisfies that $V_{\min}\leq V\leq V_{\max}$ for some constants $V_{\min},V_{\max}\in(0,+\infty)$.
\end{assumption}

\begin{theorem}\label{thm:energy_decay}
    Suppose that Assumption~\ref{asp:V} holds. Let $\mathbf u^0\in \calM\subset \bR^N$ and let $\{\mathbf u^n\}_{n=0}^\infty$ be the sequence generated by \eqref{modified-H1-scheme}. There exist constants $C_u, C_g, C_\tau > 0$ depending only on $\|\mathbf u^0\|_X$ and $\Omega, V_{\min}, V_{\max}, \alpha, \beta, k$, such that as long as $0<\tau_{\min} \leq \tau_n\leq \tau_{\max}\leq \min\{1, C_\tau\},\ \forall~n\geq 0$, the followings holds for any $n\geq 0$:
    \begin{itemize}
	   \item[(i)] $\norm{\mathbf u^n}_X\leq C_u$.
	   \item[(ii)] $\norm{\nabla_X^\calR \mathbf E_h(\mathbf u^n)}_X\leq \norm{\nabla_X \mathbf E_h(\mathbf u^n)}_X \leq C_g $.
	   \item[(iii)] $\mathbf E_h(\mathbf u^n) - \mathbf E_h(\mathbf u^{n+1})\geq C_d \norm{\nabla_X^\calR \mathbf E_h(\mathbf u^n)}_X^2$, where $C_d = \frac{\tau_{\min}}{2}$.
    \end{itemize}
\end{theorem}

A direct corollary is the global convergence to a critical point.

\begin{corollary}\label{cor:global_converge}
    In the same setting as in Theorem~\ref{thm:energy_decay}, every limit point of $\{\mathbf u^n\}_{n=0}^\infty$ is a critical point.
\end{corollary}

The rest of this subsection is for proving Theorem~\ref{thm:energy_decay} and Corollary~\ref{cor:global_converge}. We need a sequence of lemmas with the proofs of the first two lemmas being deferred to Appendix~\ref{apx:pf_global}.

\begin{lemma}\label{lem:esti_retraction}
	For any $\mathbf u\in\calM$ and $\mathbf v\in \calT_{\mathbf u}\calM$, it holds that
	\begin{equation}\label{esti_retraction}
		\norm{R_h(\mathbf u+\mathbf v) - (\mathbf u+\mathbf v)}_X\leq \frac{1}{2}\norm{\mathbf v}^2_{2} \norm{\mathbf u+\mathbf v}_X.
	\end{equation}
\end{lemma}

\begin{lemma}
	\label{lemma-norm}
	There exist positive constants $C_1,C_2$ independent of the mesh size $h$ such that for any $\mathbf u\in \bR^N$: 
\begin{itemize}
    \item[(i)] $\norm{\mathbf u}_{2}\leq \frac{1}{\sqrt{\alpha}} \norm{\mathbf u}_X$;
    \item[(ii)] $\norm{(-\Delta_h+\alpha \mathbb I)^{-1}\mathbf u}_X \leq \frac{1}{\sqrt{\alpha}} \norm{\mathbf u}_2$;
    \item[(iii)] $\|\mathbf u^2\|_2\leq C_1\|\mathbf u\|_X^2$;
    \item[(iv)] $\|\mathbf u^3\|_2\leq C_2\|\mathbf u\|_X^3$.
    \end{itemize}
\end{lemma}

\begin{lemma}\label{lem:esti_gradEu}
    Let $C_2$ be the one in Lemma~\ref{lemma-norm} and $V_{\alpha,\max} = \|V-\alpha\|_{L^{\infty}(\Omega)}.$ Then it holds for any $\mathbf u\in\calM$ that
    \begin{align}
        \notag&\langle \nabla_X^\calR \mathbf E_h(\mathbf u) , \nabla_X \mathbf E_h(\mathbf u)\rangle_X=\|\nabla_X^\calR \mathbf E_h(\mathbf u)\|^2_X, \\
        \label{eq:bound_nableE}
        &\norm{\nabla_X^\calR \mathbf E_h(\mathbf u)}_X\leq \norm{\nabla_X \mathbf E_h(\mathbf u)}_X \leq \norm{\mathbf u}_X + \frac{V_{\alpha, \max}}{\alpha} \norm{\mathbf u}_X  + \frac{\beta C_2}{\sqrt{\alpha}} \norm{\mathbf u}_X^3.
    \end{align}
\end{lemma}

\begin{proof}[Proof of Lemma~\ref{lem:esti_gradEu}]
Since $	\nabla_X^\calR \mathbf E_h(\mathbf u)\in \calT_{\mathbf u}\calM$ implies $  \left\langle \nabla_X^\calR \mathbf E_h(\mathbf u) ,\mathbf u \right\rangle_h =0$, 
\begin{equation*}
	\left\langle \nabla_X^\calR \mathbf E_h(\mathbf u) , \frac{\langle \mathbf u,\nabla_X \mathbf E_h(\mathbf u)\rangle}{\langle\mathbb G_X\mathbf u, \mathbf u\rangle} \mathbb G_X\mathbf u \right\rangle_X = \frac{\langle \mathbf u,\nabla_X \mathbf E_h(\mathbf u)\rangle_h}{\langle\mathbb G_X\mathbf u, \mathbf u\rangle_h} \left\langle \nabla_X^\calR \mathbf E_h(\mathbf u) ,\mathbf u \right\rangle_h = 0.
\end{equation*} 
Due to $\nabla_X \mathbf E_h(\mathbf u) = \nabla_X^\calR \mathbf E_h(\mathbf u) + \frac{\langle \mathbf u,\nabla_X \mathbf E_h(\mathbf u)\rangle_h}{\langle\mathbb G_X\mathbf u, \mathbf u\rangle_h} \mathbb G_X\mathbf u$, the first equation holds, and
\begin{equation*}
	\norm{\nabla_X \mathbf E_h(\mathbf u)}_X^2 = \norm{\nabla_X^\calR \mathbf E_h(\mathbf u)}_X^2 + \norm{\frac{\langle \mathbf u,\nabla_X \mathbf E_h(\mathbf u)\rangle}{\langle\mathbb G_X\mathbf u, \mathbf u\rangle} \mathbb G_X\mathbf u}_X^2\Rightarrow\norm{\nabla_X^\calR \mathbf E_h(\mathbf u)}_X\leq \norm{\nabla_X \mathbf E_h(\mathbf u)}_X.
\end{equation*}		
Let $\mathbb V_\alpha = \mathbb V - \alpha I_N$, then $-\Delta_h+\mathbb V =(-\Delta_h+\alpha \mathbb I)+\mathbb V_\alpha$. By \eqref{eq:gradX} and Lemma~\ref{lemma-norm},  
\begin{multline*}
	\norm{\nabla_X \mathbf E_h(\mathbf u)}_X \leq \norm{\mathbf u}_X + \norm{(-\Delta_h+\alpha \mathbb I)^{-1} (\mathbb V_\alpha \mathbf u)}_X  + \beta \norm{(-\Delta_h+\alpha \mathbb I)^{-1}  (\mathbf u^3)}_X \\
    \leq \norm{\mathbf u}_X + \frac{1}{\sqrt{\alpha}}\norm{\mathbb V_\alpha \mathbf u}_2  + \frac{\beta C_2}{\sqrt{\alpha}} \norm{ \mathbf u}_X^3   \leq \norm{\mathbf u}_X + \frac{V_{\alpha, \max}}{\alpha} \norm{\mathbf u}_X  + \frac{\beta C_2}{\sqrt{\alpha}} \norm{\mathbf u}_X^3.
\end{multline*}
\end{proof}
	
\begin{lemma}\label{lem:linear_error}
    Let $C_1$ be the constant in Lemma~\ref{lemma-norm}. For any $\mathbf u,\mathbf v\in \bR^N$,
    \begin{align*}
	   &\left|\mathbf E_h(\mathbf u+\mathbf v) - \mathbf E_h(\mathbf u)-\langle \nabla_X \mathbf E_h(\mathbf u),\mathbf v\rangle_X\right|\\
	   & \quad \leq \frac{1}{2}  \norm{\mathbf v}_X^2  +  \frac{V_{\alpha, \max}}{2\alpha }\norm{\mathbf v}_X^2 +\frac{3\beta C_1^2}{2} \norm{\mathbf u}_X^2 \norm{\mathbf v}_X^2 + \beta C_1^2 \norm{\mathbf u}_X \norm{\mathbf v}_X^3 + \frac{\beta C_1^2}{4} \norm{\mathbf v}_X^4 .
    \end{align*}
\end{lemma}

\begin{proof}[Proof of Lemma~\ref{lem:linear_error}] 
It can be computed that
\begin{align*}
    & \mathbf E_h(\mathbf u+\mathbf v) - \mathbf E_h(\mathbf u)= \left(\frac{1}{2}\langle -\Delta_h(\mathbf u+\mathbf v),\mathbf u + \mathbf v\rangle_h - \frac{1}{2}\langle -\Delta_h \mathbf u,\mathbf u\rangle_h \right)  \\
    &\quad\qquad + \left(\frac{1}{2}\langle \mathbb V(\mathbf u + \mathbf v),\mathbf u + \mathbf v\rangle_h - \frac{1}{2} \langle \mathbb V\mathbf u,\mathbf u\rangle_h \right) + \left(\frac{\beta}{4} \langle (\mathbf u+\mathbf v)^2,(\mathbf u + \mathbf v)^2\rangle_h - \frac{\beta}{4} \langle \mathbf u^2, \mathbf u^2\rangle_h\right) \\
    &\quad = \langle - \Delta_h \mathbf u + \mathbb V\mathbf u + \beta\mathbf u^3,\mathbf v\rangle_h \\
    & \quad\qquad + \left( \frac{1}{2}\langle -\Delta_h \mathbf v,\mathbf v\rangle_h  +  \frac{1}{2}\langle \mathbb V \mathbf v, \mathbf v\rangle_h +\frac{3\beta}{2} \langle \mathbf  u^2,\mathbf v^2\rangle_h + \beta \langle\mathbf u,\mathbf v^3\rangle_h + \frac{\beta}{4} \langle \mathbf v^2, \mathbf v^2\rangle_h \right) ,
\end{align*}
which leads to
\begin{align*}
	&\left|\mathbf E_h(\mathbf u+\mathbf v) - \mathbf E_h(\mathbf u)-\langle \nabla_X \mathbf E_h(\mathbf u),\mathbf v\rangle_X\right|\\
    &\quad= \left|\mathbf E_h(\mathbf u+\mathbf v) - \mathbf E_h(\mathbf u)-\langle - \Delta_h \mathbf u + \mathbb V\mathbf u + \beta\mathbf u^3,\mathbf v\rangle_h\right| \\
	&\quad\leq \frac{1}{2}\langle (-\Delta_h+\alpha\mathbb I) \mathbf v,\mathbf v\rangle_h  +  \frac{1}{2}\langle (\mathbb V-\alpha\mathbb I) \mathbf v, \mathbf v\rangle_h +\frac{3\beta}{2} \langle \mathbf  u^2,\mathbf v^2\rangle_h + \beta |\langle\mathbf u,\mathbf v^3\rangle_h| + \frac{\beta}{4} \langle \mathbf v^2, \mathbf v^2\rangle_h  \\
	&\quad\leq \frac{1}{2}  \norm{\mathbf v}_X^2  +  \frac{V_{\alpha,\max}}{2}\norm{\mathbf v}_2^2 +\frac{3\beta}{2} \norm{\mathbf u^2}_2 \norm{\mathbf v^2}_2 + \beta \norm{\mathbf u^2}_2^{\frac{1}{2}} \norm{\mathbf v^2}_2^{\frac{3}{2}} + \frac{\beta}{4} \norm{\mathbf v^2}_2^2  \\
	&\quad \leq \frac{1}{2}  \norm{\mathbf v}_X^2  +  \frac{V_{\alpha,\max}}{2\alpha }\norm{\mathbf v}_X^2 +\frac{3\beta C_1^2}{2} \norm{\mathbf u}_X^2 \norm{\mathbf v}_X^2 + \beta C_1^2 \norm{\mathbf u}_X \norm{\mathbf v}_X^3 + \frac{\beta C_1^2 }{4} \norm{\mathbf v}_X^4 .
    \end{align*}
\end{proof}
		
We can now present the proof of Theorem~\ref{thm:energy_decay} and Corollary~\ref{cor:global_converge}.

\begin{proof}[Proof of Theorem~\ref{thm:energy_decay}]
Let $C_1$ and $C_2$ be constants in Lemma~\ref{lemma-norm}. Define $C_3\leq 1$ as a constant depending on $\alpha$ and $V_{\min}$ that satisfies
\begin{equation}\label{eq:smallest_eigenvalue}
    \langle -\Delta_h \mathbf u,\mathbf u\rangle_h + \langle \mathbb V \mathbf u,\mathbf u\rangle_h \geq C_3 \|\mathbf u\|_X^2,\quad \forall~\mathbf u\in\bR^N.
\end{equation}
Define
\begin{equation*}
    C_u = \left(\left(\frac{2}{C_3} + \frac{2 V_{\max}}{\alpha C_3}\right) \norm{\mathbf u^0}_X^2 + \frac{\beta C_1^2}{2 C_3}\norm{\mathbf u^0}_X^4\right)^{1/2},\quad
    C_g = C_u + \frac{V_{\alpha,\max}}{\alpha} C_u + \frac{\beta C_2}{\sqrt{\alpha}} C_u^3.
\end{equation*}

We prove the theorem by induction. It is clear that (i) holds for $n=0$ since $C_3\leq 1$ implies $C_u\geq \norm{\mathbf u^0}_X$. Suppose that (i) holds for $0,1,\dots,n$ and that (ii) and (iii) hold for $0,1,\dots,n-1$. We aim to show that (ii) and (iii) hold for $n$ and that (i) holds for $n+1$.
		
It follows directly from Lemma~\ref{lem:esti_gradEu} and (i) that (ii) holds for $n$. We focus on (iii) then. The iterative scheme is
\begin{align*}
    \mathbf u^{n+1} & = R_h\left(\mathbf u^n - \tau_n \nabla_X^\calR \mathbf E_h(\mathbf u^n)\right) = \mathbf u^n - \tau_n  \nabla_X^\calR \mathbf E_h(\mathbf u^n) +\mathbf R^n, \\
    \mathbf R^n & = R_h\left(\mathbf u^n - \tau_n \nabla_X^\calR \mathbf E_h(\mathbf u^n)\right) - \left( \mathbf u^n - \tau_n \nabla_X^\calR \mathbf E_h(\mathbf u^n)\right).
\end{align*}
According to Lemma~\ref{lem:esti_retraction} and Lemma~\ref{lemma-norm}, it holds that
\begin{align*}
    \norm{\mathbf R^n}_X  &\leq \frac{\tau_n^2}{2}\norm{\nabla_X^\calR \mathbf E_h(\mathbf u^n)}_2^2 \norm{\mathbf u^n - \tau_n \nabla_X^\calR \mathbf E_h(\mathbf u^n)}_X \\
    & \leq \frac{\tau_n^2 }{2\alpha} (C_u +  C_g) \norm{\nabla_X^\calR \mathbf E_h(\mathbf u^n)}_X^2 \leq \frac{(C_u + C_g) C_g^2}{2\alpha},
\end{align*}
where we used $\tau_n\leq 1$.
Denote $ \mathbf g^n = \nabla_X^\calR \mathbf E_h(\mathbf u^n)$ and $\Tilde{\mathbf u}^n = \mathbf u^n - \tau_n \nabla_X^\calR \mathbf E_h(\mathbf u^n).$ By Lemma~\ref{lem:linear_error},  we have
\begin{align*}
	&|E_h(\Tilde{\mathbf u}^n) - E_h(\Tilde{\mathbf u}^n + \mathbf R^n)| \leq  \left|\langle \nabla_X E_h(\Tilde{\mathbf u}^n),\mathbf R^n\rangle_X \right| + \left( \frac{1}{2}  +  \frac{V_{\alpha,\max}}{2\alpha }\right)\norm{\mathbf R^n}_X^2\\
	&\quad\qquad + \left( \frac{3\beta C_1^2}{2} \norm{\Tilde{\mathbf u}^n}_X^2 \norm{\mathbf R^n}_X^2 + \beta C_1^2 \norm{\Tilde{\mathbf u}^n}_X \norm{\mathbf R^n}_X^3 + \frac{\beta C_1^2}{4} \norm{\mathbf R^n}_X^4 \right)\\
	&\quad\leq  \norm{\mathbf R^n}_X \left(\norm{\nabla_X E_h(\Tilde{\mathbf u}^n)}_X + \left( \frac{1}{2}  +  \frac{V_{\alpha,\max}}{2\alpha }\right)\norm{\mathbf R^n}_X\right.\\
	&\quad\qquad  \left.+ \frac{3\beta C_1^2}{2} \norm{\Tilde{\mathbf u}^n}_X \norm{\mathbf R^n}_X^2 + \beta C_1^2 \norm{\Tilde{\mathbf u}^n}_X \norm{\mathbf R^n}_X^2 + \frac{\beta C_1^2}{4} \norm{\mathbf R^n}_X^3\right) \\
    &\quad\leq  \tau_n^2 C_R \norm{\nabla_X^\calR \mathbf E_h(\mathbf u^n)}_X^2,
\end{align*}
where $C_R$ depends only on $C_u$, $C_g$, $V_{\alpha,\max}$, $\alpha$, $\beta$, and $C_1$. Similarly, we have
\begin{align*}
	&\left|E_h\left(\mathbf u^n - \tau_n \nabla_X^\calR \mathbf E_h(\mathbf u^n)\right)-\mathbf E_h(\mathbf u^n)-\left\langle \nabla_X \mathbf E_h(\mathbf u^n),- \tau_n \nabla_X^\calR \mathbf E_h(\mathbf u^n)\right\rangle_X\right|\\
	&\quad \leq \frac{1}{2}  \norm{\tau_n \mathbf g^n}_X^2  +  \frac{V_{\alpha,\max}}{2\alpha }\norm{\tau_n \mathbf g^n}_X^2 +\frac{3\beta C_1^2}{2} \norm{\mathbf u^n}_X^2 \norm{\tau_n \mathbf g^n}_X^2 \\
    & \quad\qquad + \beta C_1^2 \norm{\mathbf u^n}_X \norm{\tau_n \mathbf g^n}_X^3 + \frac{\beta C_1^2}{4} \norm{\tau_n \mathbf g^n}_X^4   \\
    &\quad\leq  \tau_n^2 \norm{\nabla_X^\calR \mathbf E_h(\mathbf u^n)}_X^2\left(\frac{1}{2} + \frac{V_{\alpha,\max}}{2\alpha} +  \frac{3\beta C_1^2 C_u^2}{2} +\beta C_1^2 C_u C_g + \frac{\beta C_1^2 C_g^2}{4} \right).
\end{align*}
By Lemma \ref{lem:esti_gradEu}, if $\tau_{\max}\left(\frac{1}{2} + \frac{V_{\alpha,\max}}{2\alpha} +  \frac{3\beta C_1^2 C_u^2}{2} +\beta C_1^2 C_u C_g + \frac{\beta C_1^2 C_g^2}{4} +C_R\right)\leq \frac{1}{2},$ then
\begin{align*}
	& \mathbf E_h(\mathbf u^n) - \mathbf E_h(\mathbf u^{n+1}) = \mathbf E_h(\mathbf u^n) - E_h(\Tilde{\mathbf u}^n) +E_h(\Tilde{\mathbf u}^n) - E_h(\Tilde{\mathbf u}^n + \mathbf R^n) \\
	&\quad\geq\tau_n \left\langle \nabla_X \mathbf E_h(\mathbf u^n), \nabla_X^\calR \mathbf E_h(\mathbf u^n)\right\rangle_X- |E_h(\Tilde{\mathbf u}^n) - E_h(\Tilde{\mathbf u}^n + \mathbf R^n)| \\ 
	&\quad\qquad - \left|E_h\left(\mathbf u^n - \tau_n  \nabla_X^\calR \mathbf E_h(\mathbf u^n)\right)-\mathbf E_h(\mathbf u^n)-\left\langle \nabla_X \mathbf E_h(\mathbf u^n),- \tau_n \nabla_X^\calR \mathbf E_h(\mathbf u^n)\right\rangle_X\right| \\
	&\quad\geq \tau_n \norm{\nabla_X^\calR \mathbf E_h(\mathbf u^n)}_X^2 - \tau_n^2 C_R \norm{\nabla_X^\calR \mathbf E_h(\mathbf u^n)}_X^2\\
	&\quad\qquad -\tau_n^2 \norm{\nabla_X^\calR \mathbf E_h(\mathbf u^n)}_X^2 \left(\frac{1}{2} + \frac{V_{\alpha,\max}}{2\alpha} +  \frac{3\beta C_1^2 C_u^2}{2} +\beta C_1^2 C_u C_g + \frac{\beta C_1^2 C_g^2}{4} \right) \\ 
	&\quad\geq \frac{\tau_{\min}}{2}\norm{\nabla_X^\calR \mathbf E_h(\mathbf u^n)}_X^2,
\end{align*} 
which means that (iii) holds for $n$. Since
$$\mathbf E_h(\mathbf u^{n+1})\geq \frac12 \langle -\Delta_h \mathbf u^{n+1},\mathbf u^{n+1}\rangle_h + \frac12 \langle \mathbb V \mathbf u^{n+1},\mathbf u^{n+1}\rangle_h \geq \frac12 C_3 \|\mathbf u^{n+1}\|_X^2,$$ 
one can conclude that
\begin{align*}
	\norm{\mathbf u^{n+1}}_X^2 & \leq \frac{2}{C_3} \mathbf E_h(\mathbf u^{n+1})\leq \frac{2}{C_3} \mathbf E_h(\mathbf u^0) \leq \frac{2}{C_3} \norm{\mathbf u^0}_X^2 + \frac{2 V_{\max}}{C_3}\norm{\mathbf u^0}_2^2 + \frac{\beta }{2 C_3}\norm{(\mathbf u^0)^2}_2^2\\
    &\leq \left(\frac{2}{C_3} + \frac{2 V_{\max}}{\alpha C_3}\right) \norm{\mathbf u^0}_X^2 + \frac{\beta C_1^2}{2 C_3}\norm{\mathbf u^0}_X^4,
\end{align*}
which implies that $\norm{\mathbf u^{n+1}}_X \leq  C_u$, i.e., (i) holds for $n+1$. Furthermore, we also have (ii) hold for $n+1$ by applying \eqref{eq:bound_nableE} and $\norm{\mathbf u^{n+1}}_X \leq  C_u$.
\end{proof}

\begin{proof}[Proof of Corollary~\ref{cor:global_converge}]
    It follows from Theorem~\ref{thm:energy_decay} that $\lim\limits_{n\rightarrow\infty}  \norm{\nabla_X^\calR \mathbf E_h(\mathbf u^n)}_X = 0$. Let $\mathbf u^*$ be any limit point of $\{\mathbf u^n\}_{n=0}^\infty$, i.e., $\mathbf u^{n_l}\rightarrow \mathbf u^*$ for some subsequence. Then $\mathbb G_X \mathbf u^{n_l}\rightarrow \mathbb G_X \mathbf u^*$ and $\nabla_X \mathbf E_h(\mathbf u^{n_l}) \rightarrow \nabla_X \mathbf E_h(\mathbf u^*)$. Hence $\nabla_X^\calR \mathbf E_h(\mathbf u^{n_l}) \rightarrow \nabla_X^\calR \mathbf E_h(\mathbf u^*),$
    which leads to $\nabla_X^\calR \mathbf E_h(\mathbf u^*) = 0$.
\end{proof}

\subsection{Locally exponential convergence} 
\label{sec:local_converge}
In this subsection, we prove the local convergence rate of the modified $H^1$-scheme \eqref{modified-H1-scheme}. We need the following assumption.

\begin{assumption}\label{asp:ustar}
    Let $\mathbf u^*\in\calM$ be the ground state of the nonlinear eigenvalue problem \eqref{fd2}, i.e., the global minimizer to \eqref{fd-energy}. We assume that the multiplicity of the smallest eigenvalue of $A_{\mathbf u^*} = - \Delta_h +\mathbb V + \beta \text{diag} (\mathbf u^*)^2$ is one, i.e., $\lambda^0_h<\lambda^1_h$ where $\lambda^0_h$ and $\lambda^1_h$ be the smallest and the second smallest eigenvalue of $A_{\mathbf u^*}$.
\end{assumption}

\begin{remark}
    Perron-Frobenius Theorem ensures Assumption~\ref{asp:ustar} if a monotone scheme is used. For example, the matrix $A_{\mathbf u^*}$ is monotone for  the second order finite difference scheme with any  mesh size $h$. For $Q^2$ spectral element method with a priori assumption on infinity norm of $\mathbf u^*$, the matrix $A_{\mathbf u^*}$ is monotone for small enough mesh size. However, Perron-Frobenius Theorem  does not provide a quantification of the eigengap $|\lambda_h^0-\lambda_h^1|$.  In Appendix \ref{apx:eigengap}, we prove that  $|\lambda_h^0-\lambda_h^1|$ has a uniform positive lower bound as $h\to 0$ for the  $P^1$ finite element method with quadarture on an unstructured shape regular simplicial mesh.
\end{remark}

The main local convergence result is stated as follows.

\begin{theorem}\label{thm:local_converge}
    Suppose that Assumption~\ref{asp:V} and Assumption~\ref{asp:ustar} hold. Assume the step size bounds $\tau_{\min}$ and $\tau_{\max}$ satisfy
    \begin{equation}
        \sup_{\tau_{\min}\leq \tau\leq \tau_{\max}} \left\{ (1+ L_g^2 \tau^2 )  - \frac{\tau}{C_3}\min\left\{ \frac{\lambda^1_h - \lambda^0_h}{\lambda^0_h}, 1\right\} \right\} \leq C_\tau < 1,
        \label{stepsize}
    \end{equation}
    where $L_g$ depends on $\Omega, V_{\min}, V_{\max}, \alpha, \beta, k$ and  $\mathbf u^*$. Then $\{\mathbf u^n\}_{n=0}^\infty$ converges exponentially to the ground state $\mathbf u^*$ in $\|\cdot\|_X$ when $\|\mathbf u^0-\mathbf u^*\|_X$ is sufficiently small.
\end{theorem}

For given $L_g,C_3,\frac{\lambda_h^1-\lambda_h^0}{\lambda_h^0}>0$, the condition \eqref{stepsize} holds for some $C_\tau < 1$ as long as $\tau_{\min}>0$ and $\tau_{\max}$ is sufficiently small, since the coefficient of the linear term is negative.

\begin{lemma}\label{lem:local_convex}
Suppose that Assumption~\ref{asp:ustar} holds. Then
\begin{equation}\label{eq:local_convex}
    \mathbf E_h(\mathbf u) - \mathbf E_h(\mathbf u^*) \geq \frac{\lambda^1_h -\lambda^0_h}{2} \norm{\mathbf u - \mathbf u^*}_2^2 - \frac{\lambda^1_h -\lambda^0_h}{8} \norm{\mathbf u - \mathbf u^*}_2^4,\quad \forall \mathbf u\in\mathbb R^N.
\end{equation}
\end{lemma}
 
\begin{lemma} \label{lem:lipchitz-grad}
Suppose that Assumption~\ref{asp:ustar} holds. Then there exist constants $L_\gamma,L_g>0$ depending on $\Omega, V, \alpha, \beta, k$ such that 
\begin{equation}\label{gamma-estimate}
    \norm{\nabla_X^\calR \mathbf E_h(\mathbf u)}_X\leq L_g \norm{\mathbf u-\mathbf u^*}_X,\quad \left|\frac{\langle \mathbf u,\nabla_X \mathbf E_h(\mathbf u)\rangle_h}{\langle\mathbb G_X\mathbf u, \mathbf u\rangle_h} - \lambda^0_h\right|\leq L_\gamma \norm{\mathbf u-\mathbf u^*}_X,
\end{equation}
as long as $\mathbf u\in\calM$ and $\norm{\mathbf u-\mathbf u^*}_X$ is sufficiently small.
\end{lemma}
The proof of the two lemmas will be given in Appendix \ref{sec:appendix-D}.

\begin{proof}[Proof of Theorem~\ref{thm:local_converge}]
    If $\|\mathbf u^n-\mathbf u^*\|_X$ is sufficiently small, then by  Lemma~\ref{lem:lipchitz-grad},
    \begin{align*}
        & \norm{(\mathbf u^n -\mathbf u^*) - \tau_n \nabla_X^\calR \mathbf E_h(\mathbf u^n)}_X^2 \\
        &\quad= \norm{\mathbf u^n -\mathbf u^*}_X^2 - 2\tau_n \langle \mathbf u^n - \mathbf u^*, \nabla_X^\calR \mathbf E_h(\mathbf u^n)\rangle_X +\tau_n^2 \norm{\nabla_X^\calR \mathbf E_h(\mathbf u^n)}_X^2 \\
        &\quad\leq (1+ L_g^2 \tau_n^2 ) \norm{\mathbf u^n -\mathbf u^*}_X^2 - 2\tau_n \langle \mathbf u^n - \mathbf u^*, \nabla_X^\calR \mathbf E_h(\mathbf u^n)\rangle_X \\
        &\quad= (1+ L_g^2 \tau_n^2 ) \norm{\mathbf u^n -\mathbf u^*}_X^2 + 2\tau_n \langle \mathbf u^* - \mathbf u^n, \nabla_X \mathbf E_h(\mathbf u^n)\rangle_X - 2\tau_n \gamma^n\langle \mathbf u^* - \mathbf u^n,  \mathbb G_X\mathbf u^n\rangle_X,
    \end{align*}
    where we used \eqref{eq:Riemann_gradX} with $\gamma^n = \frac{\langle \mathbf u^n,\nabla_X \mathbf E_h(\mathbf u^n)\rangle_h}{\langle\mathbb G_X\mathbf u^n, \mathbf u^n\rangle_h}$. Set $\mathbf e^n = \mathbf u^n - \mathbf u^*$, then 
    \begin{align*}
        & \mathbf E_h(\mathbf u^*) - \mathbf E_h(\mathbf u^n)=\frac{1}{2}\langle -\Delta_h (\mathbf u^n - \mathbf e^n), \mathbf u^n - \mathbf e^n \rangle_h + \frac{1}{2}\langle  \mathbb V(\mathbf u^n - \mathbf e^n ) ,\mathbf u^n - \mathbf e^n \rangle_h    \\ 
        &\quad\qquad + \frac{\beta}{4} \langle  (\mathbf u^n - \mathbf e^n )^2 , (\mathbf u^n - \mathbf e^n)^2 \rangle_h -\frac{1}{2}\langle -\Delta_h \mathbf u^n , \mathbf u^n \rangle_h - \frac{1}{2}\langle  \mathbb V \mathbf u^n ,\mathbf u^n \rangle_h - \frac{\beta}{4} \langle  (\mathbf u^n)^2 , (\mathbf u^n)^2 \rangle_h\\
        &\quad = - \langle \nabla_X \mathbf E_h(\mathbf u^n), \mathbf e^n\rangle_X  + \frac{1}{2}\langle-\Delta_h \mathbf e^n, \mathbf e^n \rangle_h + \frac{1}{2}\langle\mathbb V\mathbf e^n,\mathbf e^n\rangle_h \\
        &\quad\qquad + \frac{3\beta}{2}\langle (\mathbf u^n)^2, (\mathbf e^n)^2\rangle_h - \beta\langle \mathbf u^n,(\mathbf e^n)^3\rangle_h +\frac{\beta}{4} \langle (\mathbf e^n)^2, (\mathbf e^n)^2\rangle_h \\
        &\quad = - \langle \nabla_X \mathbf E_h(\mathbf u^n), \mathbf e^n\rangle_X  + \frac{1}{2}\langle A_{\mathbf u^*} \mathbf e^n, \mathbf e^n \rangle_h  \\
        &\quad\qquad - \frac{\beta}{2} \langle (\mathbf u^*)^2, (\mathbf e^n)^2\rangle_h+ \frac{3\beta}{2}\langle (\mathbf u^n)^2, (\mathbf e^n)^2\rangle_h - \beta\langle \mathbf u^n,(\mathbf e^n)^3\rangle_h +\frac{\beta}{4} \langle (\mathbf e^n)^2, (\mathbf e^n)^2\rangle_h\\
        &\quad \geq - \langle \nabla_X \mathbf E_h(\mathbf u^n), \mathbf e^n\rangle_X  + \frac{1}{2}\langle A_{\mathbf u^*} \mathbf e^n, \mathbf e^n \rangle_h  \\
        &\quad\qquad + \frac{\beta}{2}\langle (\mathbf u^n)^2 - (\mathbf u^*)^2, (\mathbf e^n)^2\rangle_h - \beta\langle \mathbf u^n,(\mathbf e^n)^3\rangle_h +\frac{\beta}{4} \langle (\mathbf e^n)^2, (\mathbf e^n)^2\rangle_h \\
        &\quad = - \langle \nabla_X \mathbf E_h(\mathbf u^n), \mathbf e^n\rangle_X  + \frac{1}{2}\langle A_{\mathbf u^*} \mathbf e^n, \mathbf e^n \rangle_h - \frac{\beta}{4} \| (\mathbf e^n)^2\|_2^2 \\
        &\quad \geq - \langle \nabla_X \mathbf E_h(\mathbf u^n), \mathbf e^n\rangle_X  + \frac{1}{2}\langle A_{\mathbf u^*} \mathbf e^n, \mathbf e^n \rangle_h - \frac{\beta C_1^2}{4} \| \mathbf e^n\|_X^4,
    \end{align*}
    where Lemma \ref{lemma-norm} was used in the last step. As a consequence, we obtain that 
    \begin{align*}
        & \langle \nabla_X \mathbf E_h(\mathbf u^n), \mathbf u^* - \mathbf u^n\rangle_X  
        \leq   \mathbf E_h(\mathbf u^*) - \mathbf E_h(\mathbf u^n) - \frac{1}{2}\langle A_{\mathbf u^*} \mathbf e^n, \mathbf e^n \rangle_h + \frac{\beta C_1^2}{4} \| \mathbf e^n\|_X^4 \\
        \leq & - \frac{\lambda^1_h-\lambda^0_h}{2} \| \mathbf e^n\|_2^2 + \frac{\lambda^1_h - \lambda^0_h}{8} \| \mathbf e^n\|_2^4 - \frac{1}{2}\langle A_{\mathbf u^*} \mathbf e^n, \mathbf e^n \rangle_h + \frac{\beta C_1^2}{4} \| \mathbf e^n\|_X^4 \\
        \leq & - \frac{\lambda^1_h-\lambda^0_h}{2} \| \mathbf e^n\|_2^2 - \frac{1}{2}\langle A_{\mathbf u^*} \mathbf e^n, \mathbf e^n \rangle_h + \left(\frac{\lambda^1_h - \lambda^0_h}{8\alpha^2} + \frac{\beta C_1^2}{4}\right) \| \mathbf e^n\|_X^4,
    \end{align*}
    where we used \eqref{eq:local_convex} and Lemma \ref{lemma-norm}. Note   that $ \mathbf u^n\in \calM\Rightarrow \norm{\mathbf u^n}_2^2=1$, thus
    \[  \langle \mathbf u^* - \mathbf u^n,  \mathbb G_X\mathbf u^n\rangle_X =  \langle \mathbf u^* - \mathbf u^n,\mathbf u^n\rangle_h =\frac{1}{2}\left(\norm{\mathbf u^*}_2^2 -\norm{\mathbf u^n}_2^2 -\norm{\mathbf u^n - \mathbf u^*}_2^2 \right) = - \frac{1}{2} \norm{\mathbf e^n}_2^2.\]
    Combining all estimations above we obtain that
    \begin{align*}
        & \norm{(\mathbf u^n -\mathbf u^*) - \tau_n \nabla_X^\calR \mathbf E_h(\mathbf u^n)}_X^2 \\
        &\quad\leq (1+ L_g^2 \tau_n^2 ) \norm{\mathbf u^n -\mathbf u^*}_X^2 + 2\tau_n \langle \mathbf u^* - \mathbf u^n, \nabla_X \mathbf E_h(\mathbf u^n)\rangle_X - 2\tau_n \gamma^n\langle \mathbf u^* - \mathbf u^n,  \mathbb G_X\mathbf u^n\rangle_X \\
        &\quad\leq (1+ L_g^2 \tau_n^2 ) \norm{\mathbf e^n}_X^2  + \tau_n( \gamma^n -(\lambda^1_h - \lambda^0_h) ) \| \mathbf e^n\|_2^2 -\tau_n \langle A_{\mathbf u^*} \mathbf e^n, \mathbf e^n \rangle_h \\
        &\quad\qquad + \tau_n\left(\frac{\lambda^1_h - \lambda^0_h}{4\alpha^2} + \frac{\beta C_1^2}{2}\right) \| \mathbf e^n\|_X^4 \\
        &\quad\leq (1+ L_g^2 \tau_n^2 ) \norm{\mathbf e^n}_X^2  + \tau_n( \lambda^0_h -(\lambda^1_h - \lambda^0_h) ) \| \mathbf e^n\|_2^2 -\tau_n \langle A_{\mathbf u^*} \mathbf e^n, \mathbf e^n \rangle_h \\
        & \quad\qquad + \tau_n \frac{L_\gamma}{\alpha} \| \mathbf e^n\|_X^3 + \tau_n\left(\frac{\lambda^1_h - \lambda^0_h}{4\alpha^2} + \frac{\beta C_1^2}{2}\right) \| \mathbf e^n\|_X^4 \\
        &\quad\leq (1+ L_g^2 \tau_n^2 ) \norm{\mathbf e^n}_X^2  + \tau_n\max\left\{ \lambda^0_h -(\lambda^1_h - \lambda^0_h), 0\right\}\frac{1}{\lambda^0_h} \langle A_{\mathbf u^*} \mathbf e^n, \mathbf e^n \rangle_h -\tau_n \langle A_{\mathbf u^*} \mathbf e^n, \mathbf e^n \rangle_h \\
        & \quad\qquad + \tau_n \frac{L_\gamma}{\alpha} \| \mathbf e^n\|_X^3 + \tau_n\left(\frac{\lambda^1_h - \lambda^0_h}{4\alpha^2} + \frac{\beta C_1^2}{2}\right) \| \mathbf e^n\|_X^4 \\
        &\quad\leq (1+ L_g^2 \tau_n^2 ) \norm{\mathbf e^n}_X^2  - \tau_n\min\left\{ \frac{\lambda^1_h - \lambda^0_h}{\lambda^0_h}, 1\right\} \langle A_{\mathbf u^*} \mathbf e^n, \mathbf e^n \rangle_h   + \tau_n \frac{L_\gamma}{\alpha} \| \mathbf e^n\|_X^3 \\
        &\quad\qquad + \tau_n\left(\frac{\lambda^1_h - \lambda^0_h}{4\alpha^2} + \frac{\beta C_1^2}{2}\right) \| \mathbf e^n\|_X^4 \\
        &\quad\leq (1+ L_g^2 \tau_n^2 ) \norm{\mathbf e^n}_X^2  - \frac{\tau_n}{C_3}\min\left\{ \frac{\lambda^1_h - \lambda^0_h}{\lambda^0_h}, 1\right\} \|\mathbf e^n\|_X^2 \\
        &\quad\qquad + \tau_n \frac{L_\gamma}{\alpha} \| \mathbf e^n\|_X^3  + \tau_n\left(\frac{\lambda^1_h - \lambda^0_h}{4\alpha^2} + \frac{\beta C_1^2}{2}\right) \| \mathbf e^n\|_X^4,
    \end{align*}
    where we have used \eqref{gamma-estimate} and \eqref{eq:smallest_eigenvalue}. Since we assume that $\|\mathbf u^n - \mathbf u^*\|_X$ is sufficiently small, there are some constants $C_u^{\text{loc}},C_g^{\text{loc}}>0$ such that $\|\mathbf u^n\|_X\leq C_u^{\text{loc}}$ and $\|\nabla_X^\calR \mathbf E_h(\mathbf u^n)\|_X\leq C_g^{\text{loc}}$. Recall in the proof of Theorem~\ref{thm:energy_decay}, we define
    \begin{equation*}
        \mathbf R^n = R_h\left(\mathbf u^n - \tau_n \nabla_X^\calR \mathbf E_h(\mathbf u^n)\right) - \left( \mathbf u^n - \tau_n \nabla_X^\calR \mathbf E_h(\mathbf u^n)\right),
    \end{equation*}
    and have the following bound
    \begin{align*}
        \norm{\mathbf R^n}_X  \leq \frac{\tau_n^2 }{2\alpha} (C_u^{\text{loc}} +  C_g^{\text{loc}}) \norm{\nabla_X^\calR \mathbf E_h(\mathbf u^n)}_X^2 \leq \frac{\tau_n^2 }{2\alpha} (C_u^{\text{loc}} +  C_g^{\text{loc}}) L_g^2 \norm{\mathbf e^n}_X^2.
    \end{align*}
    With \eqref{stepsize}, we have 
    \begin{align*}
        &\norm{\mathbf u^{n+1} - \mathbf u^*}_X \leq \norm{(\mathbf u^n -\mathbf u^*) - \tau_n \nabla_X^\calR \mathbf E_h(\mathbf u^n)}_X + \norm{ \mathbf R^n}_X \\
        & \quad\leq  \Bigg( (1+ L_g^2 \tau_n^2 ) \norm{\mathbf e^n}_X^2  - \frac{\tau_n}{C_3}\min\left\{ \frac{\lambda^1_h - \lambda^0_h}{\lambda^0_h}, 1\right\} \|\mathbf e^n\|_X^2  + \tau_n \frac{L_\gamma}{\alpha} \| \mathbf e^n\|_X^3 \\
        & \quad\qquad + \tau_n\left(\frac{\lambda^1_h - \lambda^0_h}{4\alpha^2} + \frac{\beta C_1^2}{2}\right) \| \mathbf e^n\|_X^4\Bigg)^{1/2} + \frac{\tau_n^2 }{2\alpha} (C_u^{\text{loc}} +  C_g^{\text{loc}}) L_g^2 \norm{\mathbf e^n}_X^2 \\
        &\quad\leq  \left( C_\tau+\tau_n \frac{L_\gamma}{\alpha} \| \mathbf e^n\|_X +\tau_n\left(\frac{\lambda^1_h - \lambda^0_h}{4\alpha^2} + \frac{\beta C_1^2}{2}\right) \| \mathbf e^n\|_X^2\right)^{1/2}\norm{\mathbf e^n}_X \\
        &\quad\qquad+ \frac{\tau_n^2 }{2\alpha} (C_u^{\text{loc}} +  C_g^{\text{loc}}) L_g^2 \norm{\mathbf e^n}_X^2. 
    \end{align*} 
    Since $C_\tau<1$, if $\norm{\mathbf e^n}_X$ is sufficiently small, then we get $ \norm{\mathbf u^{n+1} - \mathbf u^*}_X\leq  C_\delta \norm{\mathbf e^n}_X = C_\delta \norm{\mathbf u^n - \mathbf u^*}_X$, where $C_\delta\in(0,1)$ is a constant. 
\end{proof}

At the end of this section, let us remark that our results for both global and local convergence are for $\alpha>0$ and they do not directly apply to $\alpha=0$ although the algorithm remains valid when setting $\alpha=0$. The constants in Theorem~\ref{thm:energy_decay} and Theorem~\ref{thm:local_converge} depend on the choice of $\alpha>0$.

\section{Numerical tests}
\label{sec-tests}
For numerical implementation, we will only consider the $Q^k$ spectral element method, which can be easily implemented with significant acceleration on modern GPUs \cite{liu2023simple}. 
In this section, we implement the $H^1$ gradient flow \eqref{modified-H1-scheme} with the following spatial discretization for the Laplacian operator and energy function:
\begin{enumerate}
	\item The $Q^k$ spectral element method for any $k\geq 1$, see \cites{liu2023simple, li2020superconvergence}. 
    \item For $k=1$, it is exactly the same as the classical second-order finite difference. 
	\item The fourth-order compact finite difference scheme, see \cites{li2018high, li2023} for details. The definition of the discrete energy is the same as the one for the second-order finite difference scheme, i.e., the trapezoidal rule is used for approximating the integral. Only the discrete Laplacian is replaced by the compact finite difference. 
\end{enumerate}

\subsection{Accuracy test of discrete Laplacian schemes}
We consider an exact solution to the nonlinear eigenvalue problem \eqref{continuum} on $\Omega=[-1,1]^3$ with a potential $V(\bx)=\beta(1-|u^*(\bx)|^2)$ where the ground state is
\begin{equation}
	u^*(\bx)= \beta \sin\left(\pi\frac{x+1}{2 }\right)\sin\left(\pi\frac{y+1}{2 }\right)\sin\left(\pi\frac{z+1}{2 }\right)
	\label{ex-smoothexactsol}
\end{equation}
and the eigenvalue and energy are $ \lambda^*=d \frac{\pi^2}{4}+\beta,\quad E(u^*)= \frac12 \lambda^*-\frac{\beta}{4}\left(\frac{3}{4}\right)^d$. We test the accuracy of various discrete Laplacian and discrete energy schemes, shown in Table \ref{table-accuracy}. The $H^1$-scheme \eqref{modified-H1-scheme} converges within $20$ iterations with $\alpha=0.2$ and step size $1$. 
	
\begin{table}[htbp]
	\centering
	\caption{Accuracy test of various schemes for a 3D problem \eqref{ex-smoothexactsol}. For both finite difference schemes, the discrete ground state coincides with the exact ground state at grid points for \eqref{ex-smoothexactsol}, so ground state errors are meaningless for \eqref{ex-smoothexactsol} thus not listed. }
	\begin{tabular}{|c |c |c c|c c|}
		\hline
		\multicolumn{6}{|c|} {The second order finite difference (FD) } \\
		\hline  FD grid &  $\Delta x$ & $|\lambda_h^*-\lambda^*|$  &  order & $|E_h^*-E^*|$ &order  \\
		% \hline
		%  $9^3$ & $h=0.2$ & 6.07E-2 & - & 3.03E-2 & - \\
		%  \hline
		%  $19^3$ & $h=0.1$ & 1.52E-2 & 1.996 & 7.60E-3 & 1.996\\
		\hline
		$39^3$ & $h=0.05$ & 3.80E-3 & 1.999 &  1.90E-3 &-\\ % & 1.999\\
		\hline
		$79^3$ & $h=0.025$ & 9.51E-4 & 2.000 & 4.76E-4 & 2.000\\
		\hline 
		\multicolumn{6}{|c|} {The fourth-order compact finite difference} \\
		% \hline  FD grid &  Mesh size & $|\lambda_h^*-\lambda^*|$  &  order & $|E_h^*-E^*|$  & order  \\
		% \hline
		% $9^3$ & $h=0.2$ & 3.01E-4 & - & 1.51E-4 & -\\
		%  \hline
		%  $19^3$ & $h=0.1$ & 1.88E-5 & 4.004 & 9.40E-6 & 4.004\\
		\hline
		$39^3$ & $h=0.05$ & 1.17E-6 & 4.001 & 5.87E-7 &-\\ %& 4.001\\
		\hline
		$79^3$ & $h=0.025$ & 7.33E-8 & 4.000 & 3.66E-8 & 4.000\\
		\hline
	\end{tabular}
	\begin{tabular}{|c |c |c c|c c|c c|} 
		\hline
		\multicolumn{8}{|c|} {High order finite element methods} \\
		\hline   DoFs &  Mesh & $|\lambda_h^*-\lambda^*|$  &  order & $|E_h^*-E^*|$  & order& $\|\mathbf u^*_h-\mathbf u^*\|_\infty$ & order \\
		\hline
		\multicolumn{8}{|c|} {$Q^2$ spectral element method} \\
		\hline
		$9^3$ & $5^3$ & 8.13E-4 & - & 4.05E-4 & -& 3.67E-4 &-\\
		\hline
		$19^3$ & $10^3$ & 5.02E-5 & 4.016 & 2.51E-5 & 4.012&2.39E-5&3.94\\
		\hline
		% $39^3$ & $20^3$ & 3.13E-6 & 4.003 & 1.56E-6 & 4.003&1.56E-6 &3.94\\
		%  \hline
		% $79^3$ & $40^3$ & 1.95E-7 & 4.001 & 9.78E-8 & 4.001&9.87E-8 & 3.98\\
		% \hline
		\multicolumn{8}{|c|} {$Q^3$ spectral element method} \\
		\hline
		$11^3$ & $4^3$ & 5.20E-5 & - & 2.87E-6 & -&1.48E-5 &-\\
		\hline
		$23^3$ & $8^3$ & 8.98E-8 & 5.855 & 4.49E-8 & 6.000&5.21E-7& 4.83\\
		\hline
		%  $47^3$ & $16^3$ & 1.40E-9 & 6.001 & 7.01E-10 & 6.000&  1.67E-8& 4.96\\
		% \hline 
		\multicolumn{8}{|c|} {$Q^4$ spectral element method} \\
		\hline
		$11^3$ & $3^3$ & 9.00E-7 & - & 5.22E-8 & -& 4.99E-6 &-\\
		\hline
		$23^3$ & $6^3$ & 4.11E-10 & 11.09 & 2.05E-10 & 7.990&8.30E-8 &5.91\\
		\hline
		$47^3$ & $12^3$ & 2.26E-12 & 7.506 & 1.23E-12 & 7.384&1.44E-9 & 5.85\\
		\hline 
		\multicolumn{8}{|c|} {$Q^5$ spectral element method} \\
		\hline
		$9^3$ & $2^3$ & 1.49E-6 & - & 3.87E-5 & -&1.14E-6 &-\\
		\hline
		$19^3$ & $4^3$ & 6.84E-11 & 14.41 & 6.50E-12 & 22.51&7.87E-9 & 7.17\\
		\hline 
		\multicolumn{8}{|c|} {$Q^6$ spectral element method} \\
		\hline
		$9^3$ & $2^3$ & 8.08E-9 & - & 7.31E-7 & -& 4.06E-8  &-\\
		\hline
		\multicolumn{8}{|c|} {$Q^7$ spectral element method}  \\
		\hline
		$13^3$ & $2^3$ & 1.89E-10 & - & 9.99E-9 & -&1.61E-9 &-\\
		\hline
		\multicolumn{8}{|c|} {$Q^8$ spectral element method}  \\
		\hline
		$15^3$ & $2^3$ & 3.76E-13 & - & 1.03E-10 & -&5.76E-11 &-\\
		\hline
	\end{tabular}
	\label{table-accuracy}
\end{table}

\subsection{Comparison of various gradient flow algorithms in 2D}
	
We consider the 2D problem with $V(\bx)=\sin\left(\frac{\pi}{4}x\right)^2\sin\left(\frac{\pi}{4}y\right)^2$ on the region $\Omega=[-16, 16]^2$. The performance of different gradient flow algorithms with fixed step size $1$ is shown in Figure \ref{2D-example-figure} (a) and (b). See \cite{henning2020sobolev} and references therein for the definition of these schemes. We emphasize that these algorithms could be faster with different step sizes, e.g., the $L^2$ flow will be faster with a larger step size, and $A_u$ algorithm can be faster with adaptive step size. Here we just use the same step size to compare them. Notice that  only $(-\Delta_h+\alpha I)^{-1}$ needs to be applied twice in the modified $H^1$ gradient flow \eqref{modified-H1-scheme}. In each iteration of $L^2$, $A_u$ and $A_0$ schemes, one needs to invert matrices like $-\Delta_h+V(\bx)$ or $-\Delta_h+V(\bx)+\beta |\bf u|^2$, which is much more expensive than computing $(-\Delta_h+\alpha I)^{-1}$  \cite{liu2023simple}. As shown in Figure \ref{2D-example-figure} (a) and (b),   \eqref{modified-H1-scheme} with a proper parameter $\alpha>0$ can allow a much larger step size for convergence, compared to $\alpha=0$. Since $H^1$, $A_0$ and $A_u$ schemes can all be written as  Riemannian gradient descent methods, in each iteration, one can also numerically compute the best step size by  minimizing the energy function w.r.t. the step size. For simplicity, we use the {\it fminbnd} function in MATLAB to solve such a one-dimensional minimization problem, which involves evaluating the energy function quite a few times.
In  Figure \ref{2D-example-figure} (c) and (d), we show the performance of $H^1$, $A_0$ and $A_u$ schemes using such an optimal step size. We observe that iteration numbers in all the Riemannian gradient descent  schemes with the optimal step size to reach convergence are almost the same. On the other hand,  for CPU time of solving this particular problem, $H^1$ scheme with $\alpha=0.15$ and the optimal step size is obviously slower than $H^1$ scheme with $\alpha=0.15$ and a constant step size $1$, due to the extra computational cost of computing the optimal step size, which would be more expensive for 3D problems.
	
\begin{figure}[htbp]
	\begin{center} 
    \subfigure[Fixed step size $1$.]{\includegraphics[scale=0.33]{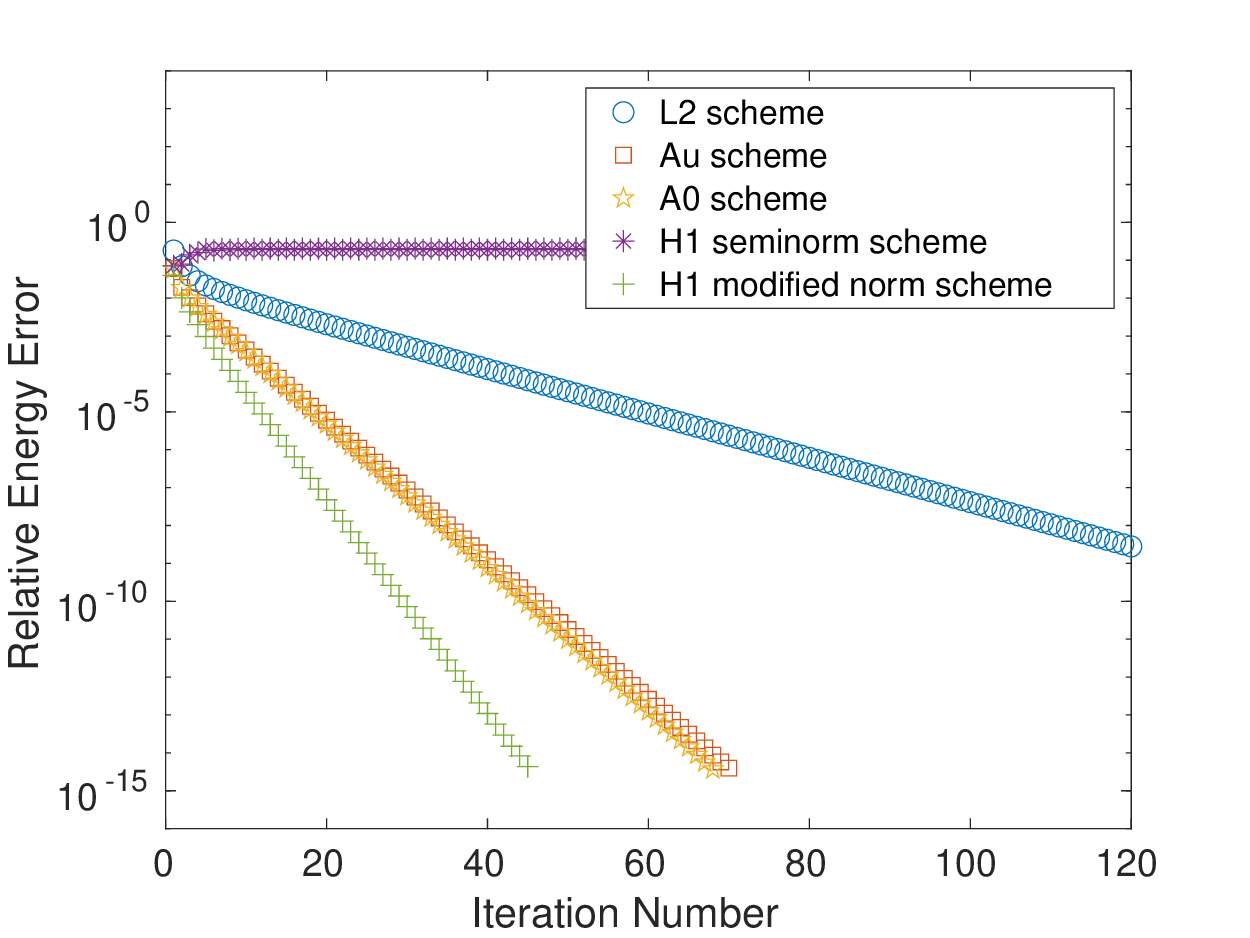}}  
     \subfigure[Fixed step size $1$. ]{\includegraphics[scale=0.33]{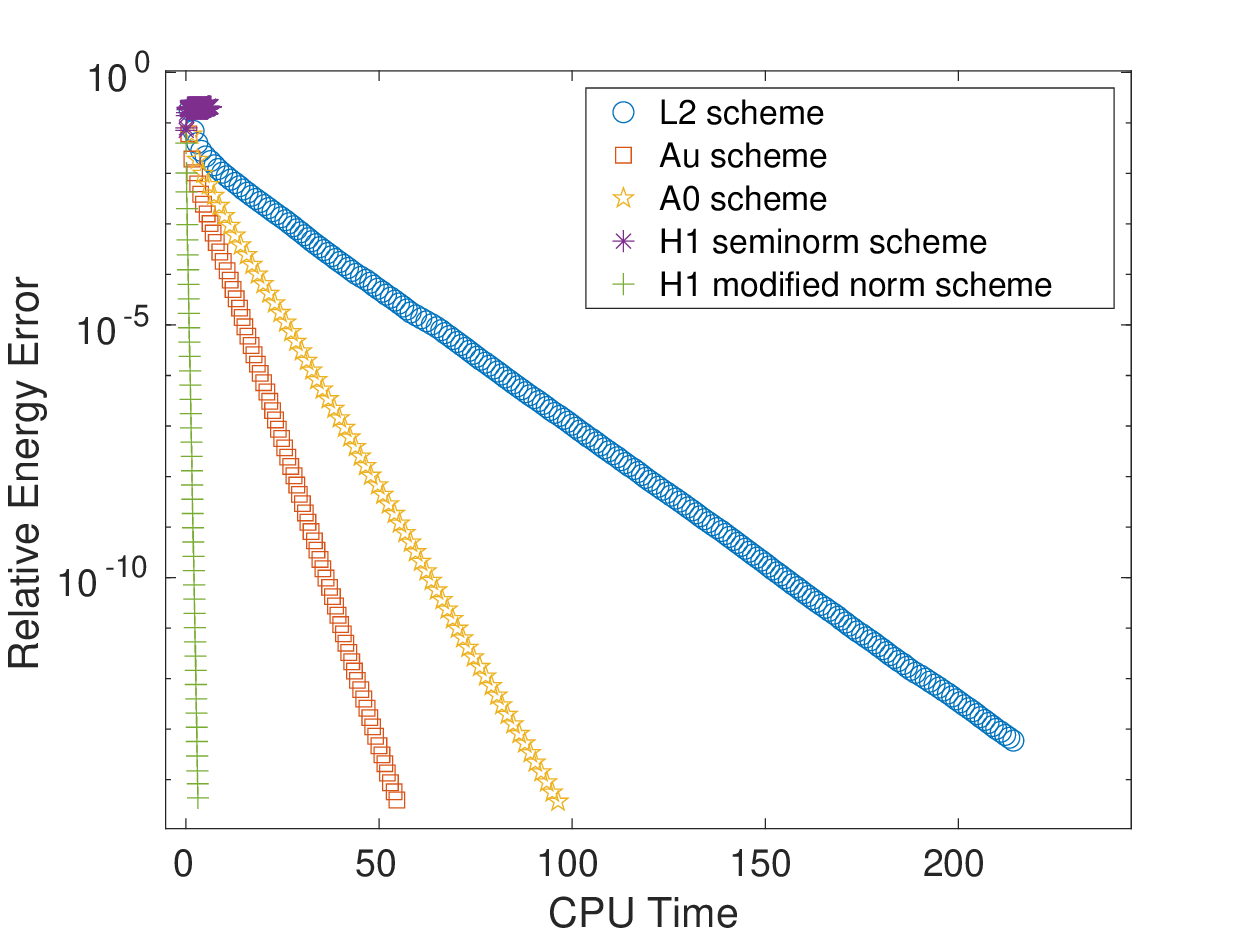}} \\  
      \subfigure[Optimal step size is used for gradient flows.]{\includegraphics[scale=0.33]{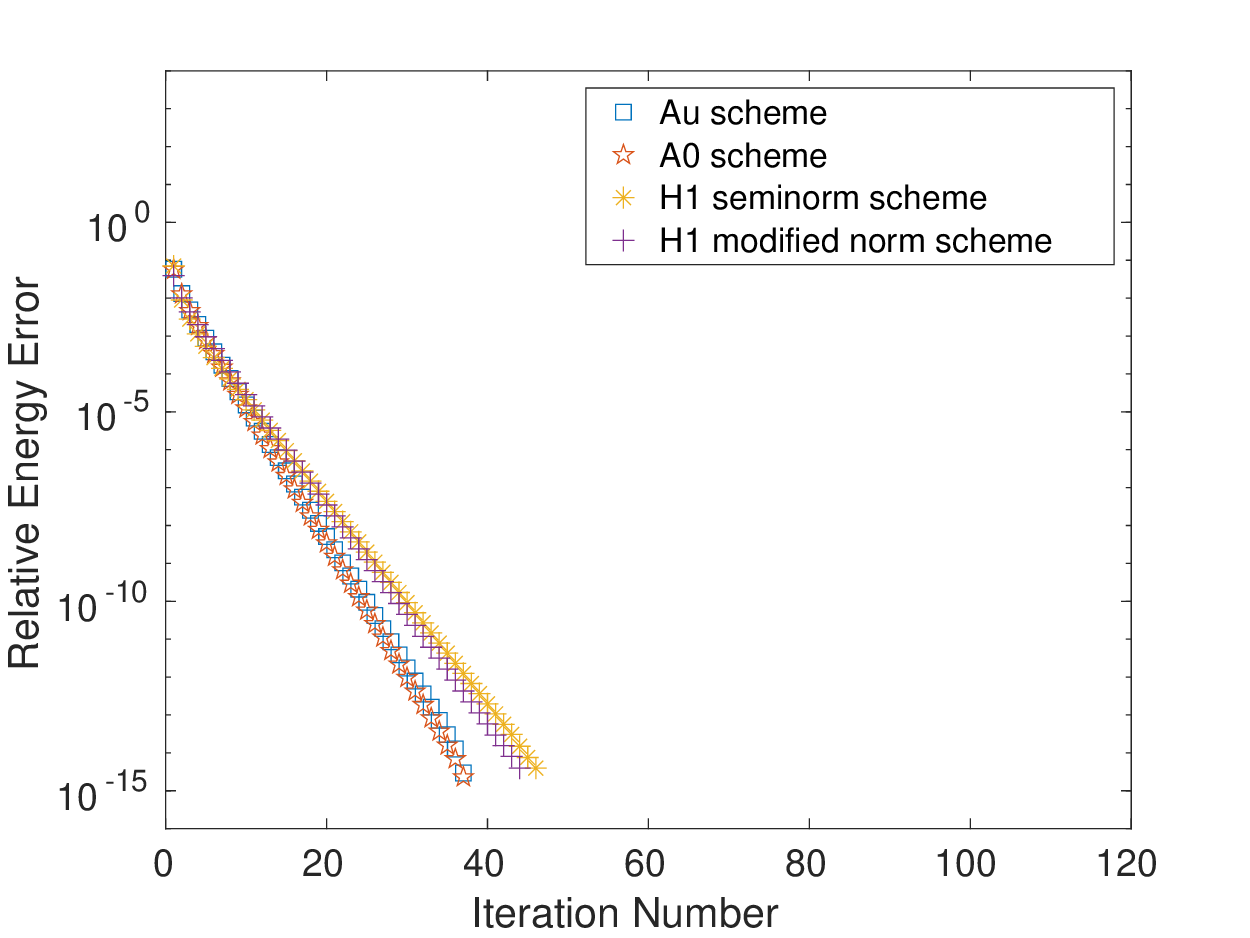}}  
     \subfigure[Optimal step size is used for gradient flows.]{\includegraphics[scale=0.33]{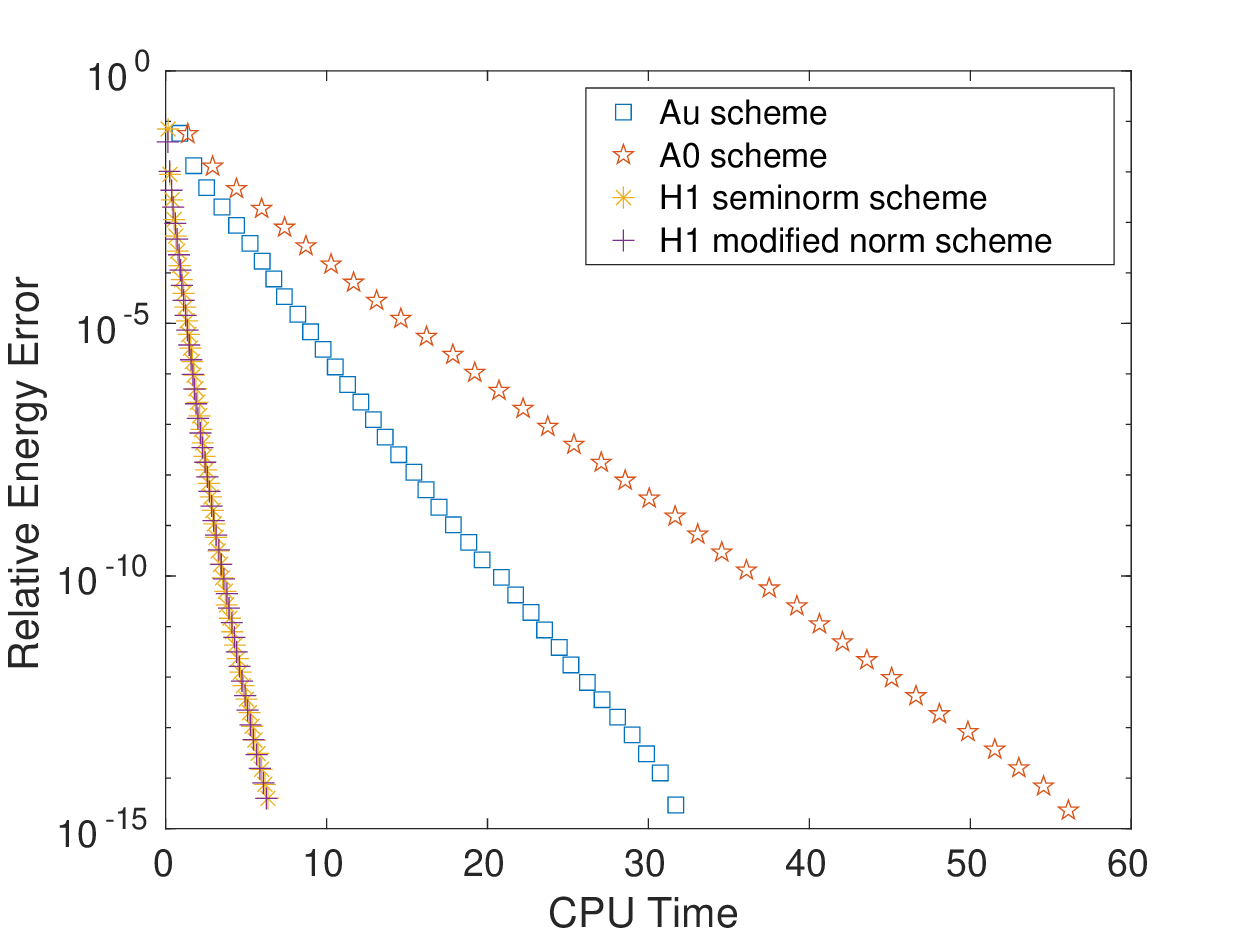}}   
    \end{center}
    \caption{A 2D example with $\beta=10$ of second order finite difference on a  $800\times 800$ grid. The modified $H^1$ norm has parameter $\alpha=0.15$  in \eqref{modified-H1-scheme}. The $H^1$ seminorm scheme is  \eqref{modified-H1-scheme} with $\alpha=0$. 
    The CPU time for the $H^1$ scheme with $\alpha=0.15$ and the fixed step size $1$ to converge is 3 seconds, and the CPU time for the $H^1$ scheme with $\alpha=0.15$ or $\alpha=0$ with optimal step size to converge is more than 6 seconds.
    For $L^2$, $A_u$, and $A_0$ schemes (see \cite{henning2020sobolev} for definition), preconditioned conjugate gradient (PCG) is used for inverting a matrix like $-\Delta_h+V(\bx)$ and $(-\Delta_h)^{-1}$ is used as a preconditioner. The PCG converges within 30 iterations for all linear systems in this test. 
    }	\label{2D-example-figure}
\end{figure}
	
\subsection{Comparison with the Backward Forward Euler method}
	
The modified $H^1$ flow has the advantage of inverting only constant coefficient Laplacian, which can be easily accelerated on modern GPUs as shown recently in \cite{liu2023simple}. On the other hand, in the literature, there are similar schemes, e.g., the Backward Forward Euler method with a stabilization parameter in \cite{bao2006efficient} is given by
\begin{equation}
	\label{BFSP} \tilde{\mathbf u}=\left(-\Delta_h+\alpha\mathbb I+\frac{1}{\Delta t} \mathbb I\right)^{-1}\left(\alpha+\frac{1}{\Delta t}-\mathbb V-\beta\diag(\mathbf u^n)^2\right)\mathbf u^n,\quad \mathbf u^{n+1}=\frac{\tilde{\mathbf u}}{\|\tilde{\mathbf u}\|_2}.
\end{equation}
The modified $H^1$ flow is very different from \eqref{BFSP}. For instance, \eqref{modified-H1-scheme} is a Riemannian gradient descent method. In particular, only one inversion of Laplacian is needed in each step of  \eqref{BFSP}, but there are two inversions of Laplacian in \eqref{modified-H1-scheme}. The optimal parameter $\alpha$ for \eqref{BFSP} was given in \cite{bao2006efficient}, yet it is unclear what the optimal $\alpha$ should be for the modified $H^1$ flow. In numerical tests, the modified $H^1$ flow performs better when using tuned $\alpha$, see Figure  \ref{figure-bfsp}. 
\begin{figure}[htbp]
	\begin{center}
		\subfigure{\includegraphics[scale=0.33]{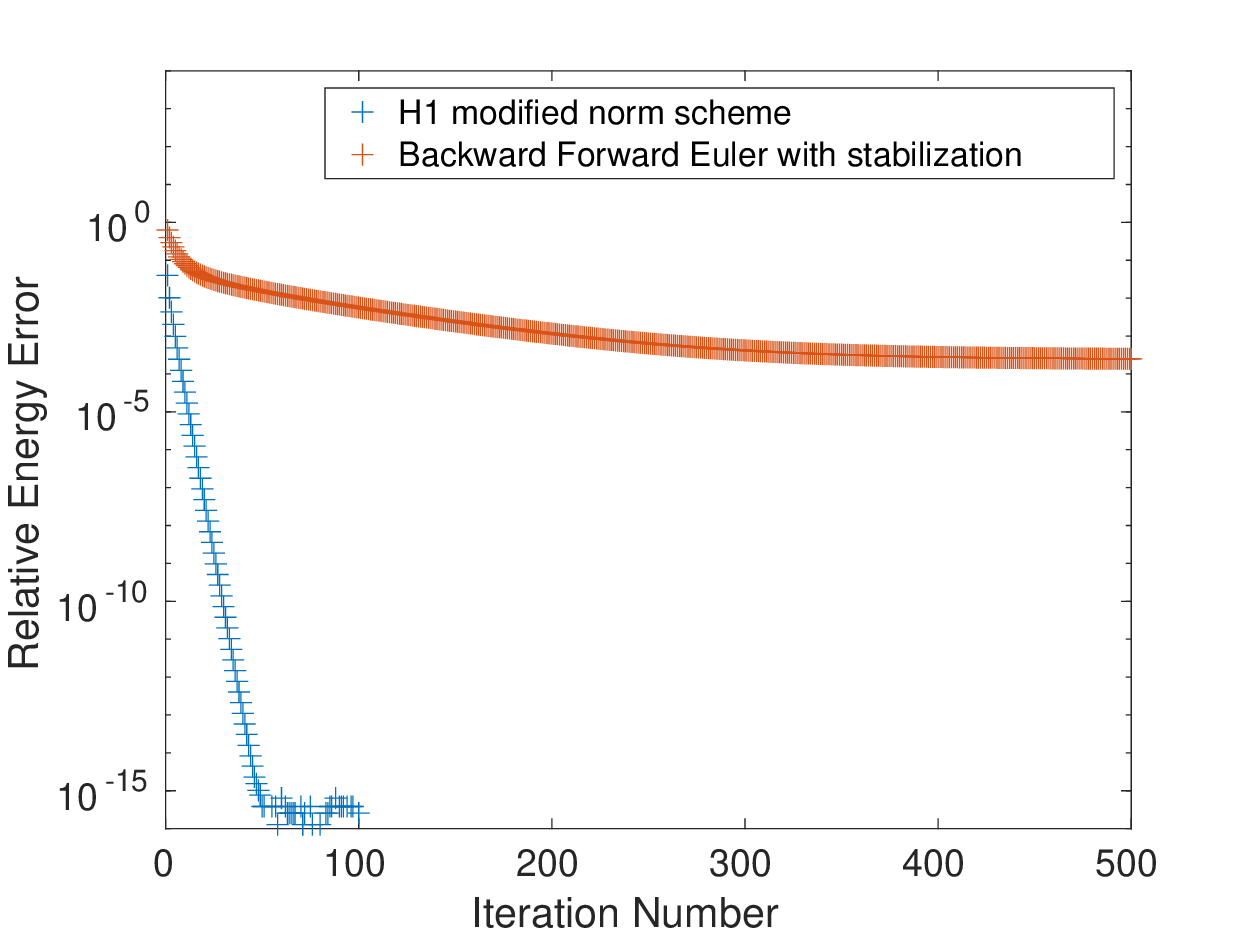} } 
		\hspace{-0.2cm}
		% \subfigure{\includegraphics[scale=0.23]{discretized problem/Figure/BFSP3.eps} } 
		% \hspace{-0.2cm}
		\subfigure{\includegraphics[scale=0.33]{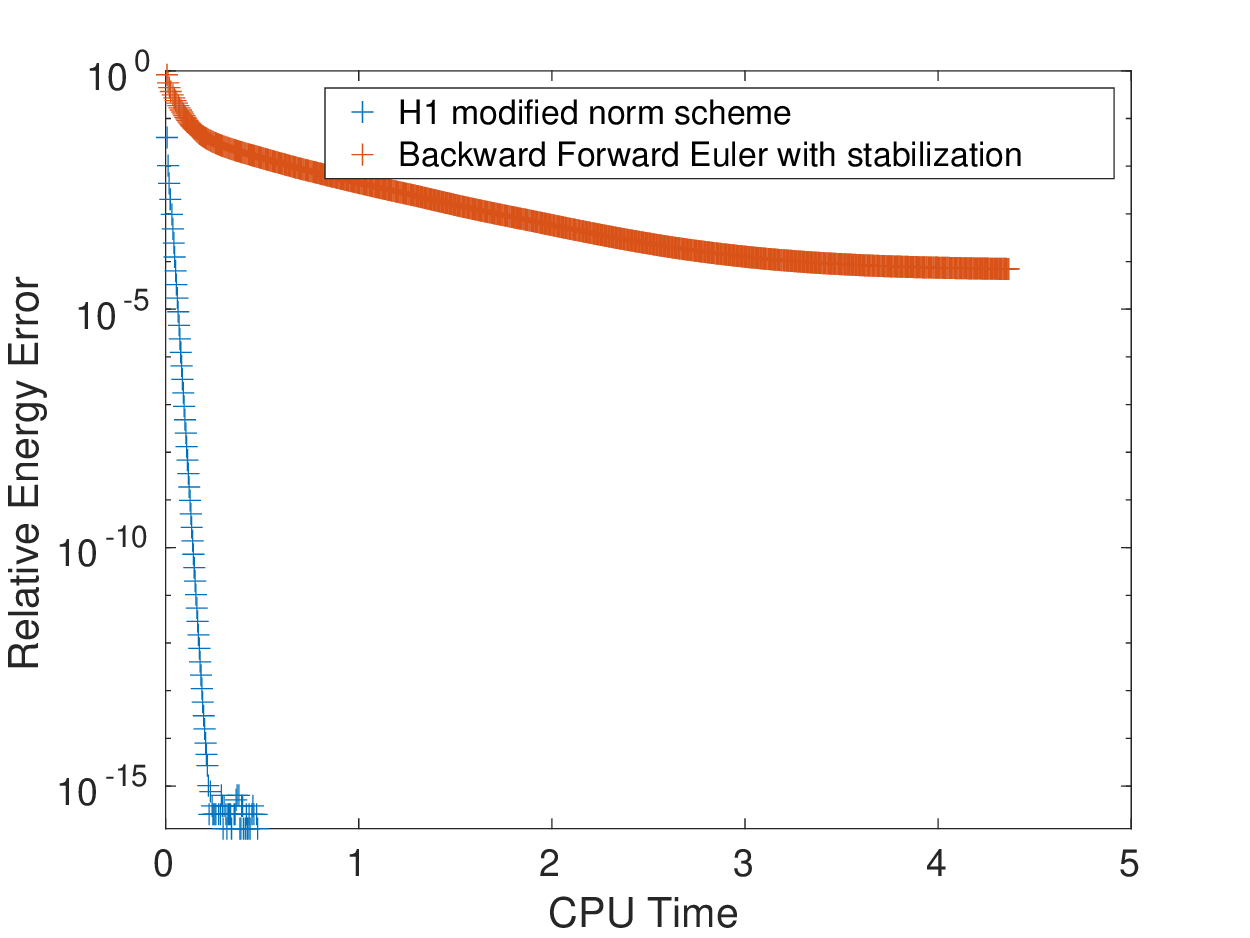} } \\
		%   \subfigure{\includegraphics[scale=0.23]{discretized problem/Figure/BFSP4.eps} } 
		%  \hspace{-0.2cm}
		%  \subfigure{\includegraphics[scale=0.23]{discretized problem/Figure/BFSP5.eps} } 
		%  \hspace{-0.2cm}
		%  \subfigure]{\includegraphics[scale=0.23]{discretized problem/Figure/BFSP6.eps} } 
	\end{center}
	\caption{A 2D example with $\beta=5$ of second order discrete Laplacian on a $300\times 300$ grid. The modified $H^1$ norm has parameter $\alpha=0.15$ and step size $1$   in \eqref{modifiednorm}.  The Backward Forward Euler with stabilization \eqref{BFSP} uses the optimal parameters $\alpha$ in \cite{bao2006efficient} and the largest stable step size $0.1$, which is also the most efficient step size. The initial condition is the ground state for $\beta=0$. }
	\label{figure-bfsp}
\end{figure}

\subsection{Tuning parameters}
We consider  $\Omega=[-16, 16]^3$ with a potential:
\begin{equation}
	V(\bx)=\sin\left(\frac{\pi}{4}x\right)^2\sin\left(\frac{\pi}{4}y\right)^2\sin\left(\frac{\pi}{4}z\right)^2.
	\label{3d-potential-gpu}
\end{equation}
For \eqref{3d-potential-gpu} and $\beta=10$, the modified  $H^1$ scheme \eqref{modified-H1-scheme} with  $\alpha=0.15$ and $\tau=1$ has the same convergence performance for any grid size or any discrete Laplacian, unless it is an extremely coarse grid, as shown in Figure \ref{3D-example-differentgrid} (a). Thus we can easily find the best step size for a given $\beta=4000$ by tuning it on a $100^3$ grid as shown in Figure \ref{3D-example-differentgrid} (b). 
	
\begin{figure}[htbp]
	\begin{center}
		\subfigure[ Step size is $1$, $\alpha=0.15, \beta=10$.]{\includegraphics[scale=0.24]{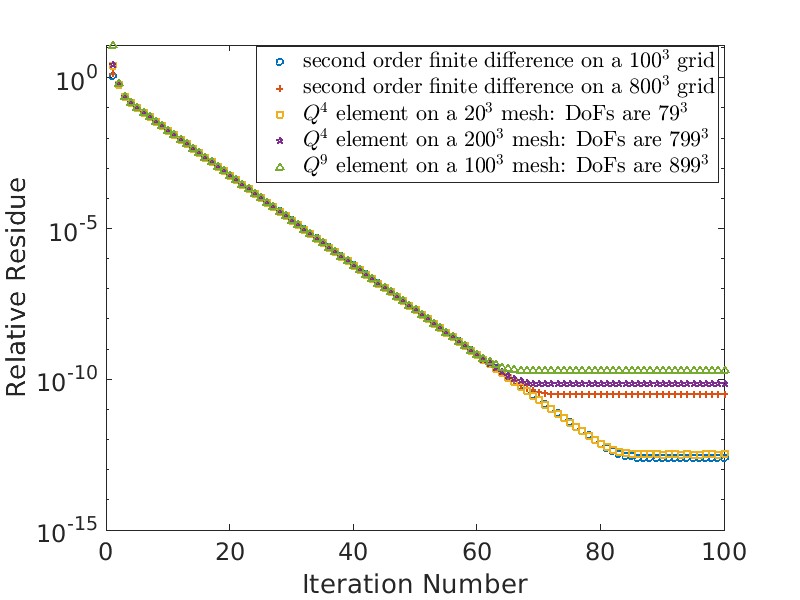} } 
		\hspace{-0.1cm}
		%  \subfigure[ Step size is $1$, $\alpha=0.15, \beta=10$.]{\includegraphics[scale=0.25]{discretized problem/Figure/Compare_grid_coarse}} \\
		%   \subfigure[Fixed step size $1$, $\alpha=0.15$, second order finite difference on a $100^3$ grid.]{\includegraphics[scale=0.25]{discretized problem/Figure/compare_beta} } 
		%  \hspace{-0.1cm}
		\subfigure[Fixed $\beta=4000$, $\alpha=0.15$, second order finite difference on a $100^3$ grid.]{\includegraphics[scale=0.24]{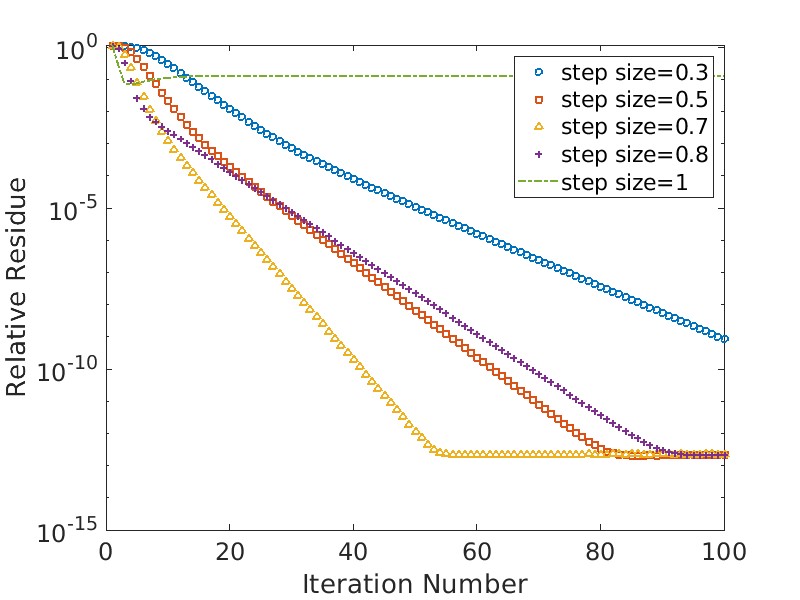}} 
	\end{center}
	\caption{The performance of  the modified  $H^1$ scheme \eqref{modified-H1-scheme} solving 3D Gross-Pitaevskii nonlinear eigenvalue problem with potential \eqref{3d-potential-gpu}. The left shows that the performance is independent of discretization and grid size. Thus parameters can be tuned on a coarse grid as shown in the right.  }
	\label{3D-example-differentgrid}
\end{figure}

\subsection{3D implementation on GPUs}
	
As shown recently in \cite{liu2023simple}, any discrete Laplacian on a Cartesian grid can be easily accelerated on modern GPUs with a simple implementation in MATLAB 2023. In particular, to invert a discrete Laplacian on a grid size $1000^3$, it only takes 0.8 seconds on one Nividia A100 GPU card. And such a result holds for $Q^k$ spectral element method, see \cite{liu2023simple} for details.   
	
We consider solving the 3D problem with potential \eqref{3d-potential-gpu}. See Figure \ref{3D-example-potential} for visualization of the potential and its ground state for $\beta=10$ and $\beta=4000$. We define the online computation time as the computational time without counting the offline computational time such as preparing discrete Laplacian matrices and loading data to the GPU. For the potential \eqref{3D-example-potential} with $\beta=10$,  the modified $H^1$ flow scheme \eqref{modified-H1-scheme} with $\alpha=0.15$ and  $\tau=1$, we stop the iteration when the relative residue stops decreasing, where the relative residue is defined as
\[\text{residue}=\left\|\mathbf u-\frac{-\Delta_h \mathbf u+\mathbb V\mathbf u +\beta |\mathbf u|^3}{\|-\Delta_h \mathbf u+\mathbb V\mathbf u +\beta |\mathbf u|^3\|}\right\|/\|\mathbf u\|.\]
The results of computing energy and the eigenvalue are listed in Table \ref{table-accuracy-GPU}, in which GPU time is the online computational time. In particular, the online computation time is 214 seconds on the  Nvidia A100 for 100 iterations of \eqref{modified-H1-scheme} on a $1000^3$ grid. For $\beta=4000$, we use step size $0.7$ and $\alpha=0.15$, we stop the iteration when the relative residue stops decreasing. The performance is listed in Table \ref{table-accuracy-GPU2}. In both Table \ref{table-accuracy-GPU} and Table \ref{table-accuracy-GPU2}, the reference solutions are generated by $Q^{10}$ spectral element method on a $100^3$ mesh, and the ground state errors $\frac{\|{\bf u^*}-{\bf u}_{ref}\|_{\ell^\infty}}{\|{\bf u}_{ref}\|_\infty}$ are measured only at the nodes of matching with nodes of $Q^{10}$ elements a $100^3$ mesh. For instance, for $Q^5$ element on a $20^3$ mesh,  $\frac{\|{\bf u^*}-{\bf u}_{ref}\|_{\ell^\infty}}{\|{\bf u}_{ref}\|_\infty}$ is measured only at the cell vertices.

\begin{figure}[htbp]
	\begin{center}
		\subfigure[The potential function.]{\includegraphics[scale=0.23]{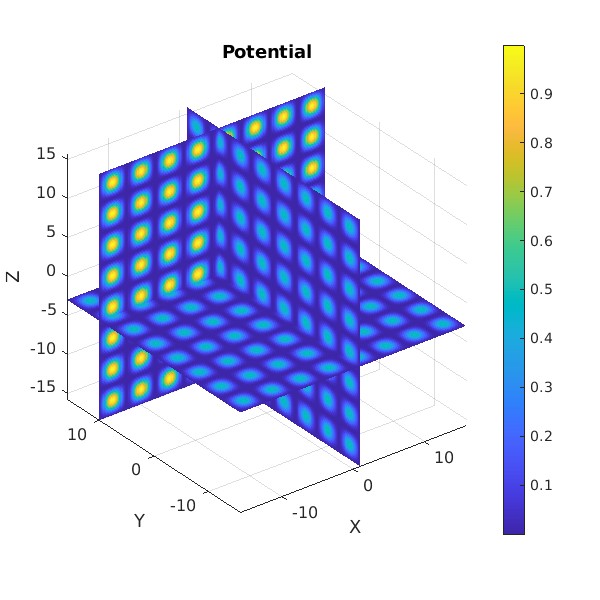} }
		\hspace{-0.5cm}
		\subfigure[$\beta=10$]{\includegraphics[scale=0.23]{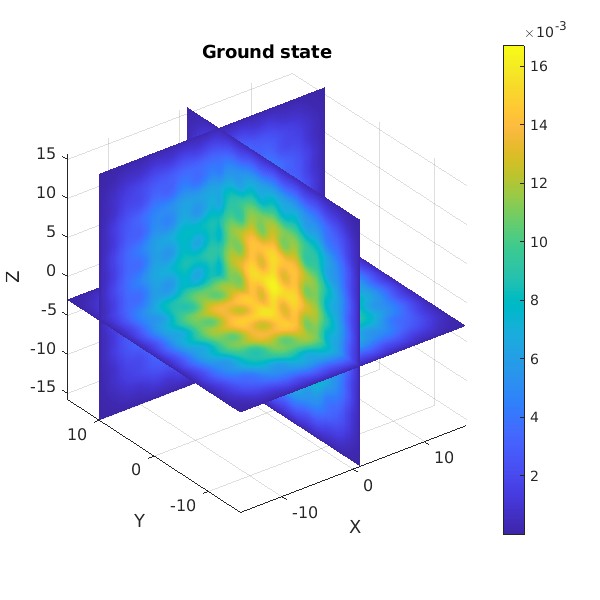}}
		\hspace{-0.5cm}
		\subfigure[$\beta=4000$]{\includegraphics[scale=0.23]{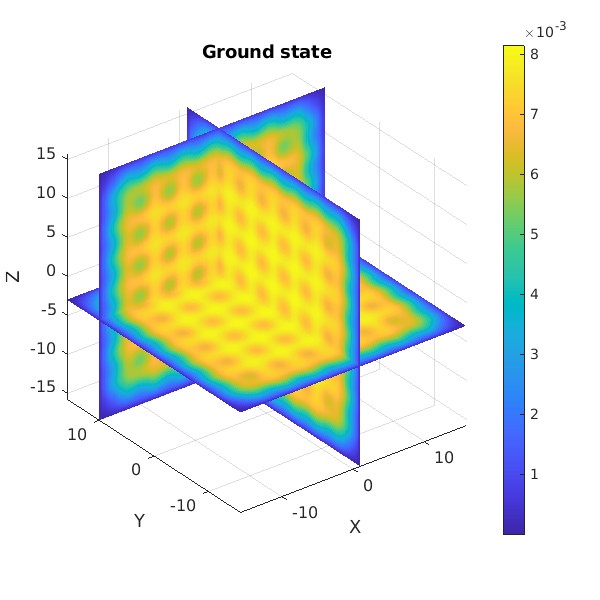}}
	\end{center}
	\caption{The potential function \eqref{3d-potential-gpu} and its ground state.}
	\label{3D-example-potential}
\end{figure}

\begin{table}[htbp]
	\centering
	\caption{The 3D problem as shown in Figure \ref{3D-example-potential} with $\beta=10$. The GPU (Nvidia A100 80G) online computation time of  the scheme \eqref{modified-H1-scheme} with $\alpha=0.15$ and  $\tau=1$. The iteration stops when the relative residue stops decreasing. The reference eigenvalue $\lambda^*_{ref}=0.143834048046$, energy $E_{ref}=0.071660785256$ and ground state ${\bf u}_{ref}$ are generated by $Q^{10}$ SEM on a $100^3$ mesh. }
	\resizebox{\textwidth}{!}{
		\begin{tabular}{|c |c |c |c | c| c|c|}
			\hline
			\multicolumn{7}{|c|} {The second order finite difference} \\
			\hline 
			DoFs & Mesh size & $\frac{|\lambda_h^*-\lambda_{ref}^*|}{\lambda_{ref}^*}$  & $\frac{|E_h({\bf u^*})-E_{ref}|}{E_{ref}}$ & $\frac{\|{\bf u^*}-{\bf u}_{ref}\|_{\ell^\infty}}{\|{\bf u}_{ref}\|_\infty}$ & Iteration \# &  GPU time    \\
			\hline 
			$199^3$ & $h=\frac{32}{200}$ & 3.70E-4 & 3.72E-4 & 8.04E-4 &86 & 0.88 second \\
			\hline
			$999^3$ & $h=\frac{32}{1000}$ & 1.48E-5 & 1.48E-5 & 3.21E-5 &76 & 165 seconds \\
			\hline
			\multicolumn{7}{|c|} {The fourth-order compact finite difference} \\
			\hline 
			$199^3$ & $h=\frac{32}{200}$ & 1.17E-6 & 1.18E-6 & 2.00E-6 &84 & 0.92 second \\
			\hline
			$999^3$ & $h=\frac{32}{1000}$ & 1.86E-9 & 1.87E-9 & 3.21E-9 &73 & 161 seconds \\
			\hline
			\multicolumn{7}{|c|} {$Q^4$ spectral element method} \\
			\hline 
			DoFs & FEM Mesh   & $\frac{|\lambda_h^*-\lambda_{ref}^*|}{\lambda_{ref}^*}$  &  $\frac{|E_h({\bf u^*})-E_{ref}|}{E_{ref}}$ & $\frac{\|{\bf u^*}-{\bf u}_{ref}\|_{\ell^\infty}}{\|{\bf u}_{ref}\|_\infty}$ & Iteration \# &  GPU time    \\
			\hline
			$199^3$ & $50^3$ & 8.57E-10 & 8.62E-10 & 2.73E-7 &83 & 0.88 second \\
			\hline
			$399^3$ & $100^3$ & 3.60E-12 & 3.62E-12 & 4.86E-9 &78 & 6.10 seconds \\
			\hline 
			\multicolumn{7}{|c|} {$Q^{5}$ spectral element method} \\
			\hline  
			$99^3$ & $20^3$ & 3.53E-9 & 3.49E-9 & 6.06E-7 &88 & 0.54 second \\
			\hline
			\multicolumn{7}{|c|} {$Q^{20}$ spectral element method} \\
			\hline  
			$99^3$ & $5^3$ & 5.31E-12 & 5.23E-12 & 1.29E-8 &83 & 0.54 second \\
			\hline
			$199^3$ & $10^3$ & 8.79E-12 & 8.85E-12 & 6.42E-12 &75 & 0.82 second \\
			\hline 
		\end{tabular}}
	\label{table-accuracy-GPU}
\end{table}
	
\begin{table}[htbp]
	\centering
	\caption{The 3D problem as shown in Figure \ref{3D-example-potential} with $\beta=4000$. The GPU (Nvidia A100 80G) online computation time of the scheme \eqref{modified-H1-scheme} with $\alpha=0.15$ and  $\tau=0.7$. The iteration stops when the relative residue stops decreasing. The reference eigenvalue $\lambda^*_{ref}=0.34919956116$, energy $E_{ref}=0.127936543199$ and ground state ${\bf u}_{ref}$ are generated by $Q^{10}$ SEM on a $100^3$ mesh.  }
	\resizebox{\textwidth}{!}{
		\begin{tabular}{|c |c |c |c | c| c|c|}
			\hline
			\multicolumn{7}{|c|} {The second order finite difference} \\
			\hline 
			DoFs & Mesh size & $\frac{|\lambda_h^*-\lambda_{ref}^*|}{\lambda_{ref}^*}$  & $\frac{|E_h({\bf u^*})-E_{ref}|}{E_{ref}}$ & $\frac{\|{\bf u^*}-{\bf u}_{ref}\|_{\ell^\infty}}{\|{\bf u}_{ref}\|_\infty}$ & Iteration \# &  GPU time    \\ 
			\hline
			$999^3$ & $h=\frac{32}{1000}$ & 5.67E-6 & 7.82E-6 & 2.96E-5 &48 & 105 seconds \\
			\hline
			\multicolumn{7}{|c|} {The fourth-order compact finite difference} \\
			\hline
			$999^3$ & $h=\frac{32}{1000}$ & 5.67E-10 & 8.72E-10 & 2.77E-9 &48 & 106 seconds \\
			\hline
			\multicolumn{7}{|c|} {$Q^4$ spectral element method} \\
			\hline 
			DoFs & FEM Mesh   & $\frac{|\lambda_h^*-\lambda_{ref}^*|}{\lambda_{ref}^*}$  & $\frac{|E_h({\bf u^*})-E_{ref}|}{E_{ref}}$ & $\frac{\|{\bf u^*}-{\bf u}_{ref}\|_{\ell^\infty}}{\|{\bf u}_{ref}\|_\infty}$ & Iteration \# &  GPU time    \\
			\hline
			$199^3$ & $50^3$ & 2.64E-10 & 4.06E-10 & 2.81E-7 &54 & 0.63 second \\
			\hline
			$399^3$ & $100^3$ & 1.79E-12 & 2.33E-12 & 4.45E-9 &50 & 3.88 seconds \\
			\hline 
			\multicolumn{7}{|c|} {$Q^{5}$ spectral element method} \\
		  \hline  
			$99^3$ & $20^3$ & 7.49E-8 & 6.44E-8 & 2.18E-7 &57 & 0.42 second \\
			\hline
			\multicolumn{7}{|c|} {$Q^{20}$ spectral element method} \\
			\hline  
			$99^3$ & $5^3$ & 3.43E-12 & 4.07E-12 & 8.32E-9 &54 & 0.37 second \\
			\hline
			$199^3$ & $10^3$ & 3.58E-12 & 4.96E-12 & 3.22E-12 &50 & 0.57 second \\
			\hline 
		\end{tabular}}
	\label{table-accuracy-GPU2}
\end{table}

\subsection{A 3D example with a combined harmonic and optical lattice potential}
	
We consider the 3D example in \cite{bao2006efficient} with the following combined harmonic and optical lattice potential on the domain $[-8,8]^3$:
\[ V(x,y,z)=x^2+y^2+z^2+100\left(\sin^2\frac{\pi x}{4}+\sin^2\frac{\pi y}{4}+\sin^2\frac{\pi z}{4}\right).\]
	
For $\beta=1600$, we find that $\alpha=10$ and $\Delta t=0.1$ are efficient parameters. The modified $H^1$ flow with $Q^{40}$ spectral element method on a $5\times 5\times 5$ finite element mesh converges with residue reaching $6.3\times 10^{-13}$ after 665 iterations using a simple and crude initial guess $\mathbf u^0\equiv 1$. The online computation time is 6 seconds on Nvidia A100. The numerical ground state energy and eigenvalue are $E(\mathbf u^*_h)=33.80227900547$ and $\lambda_h^*=80.89511440602$. The results are consistent with findings in  \cite{bao2006efficient}. Due to the different definitions of energy functions  in this paper and \cite{bao2006efficient},	$\beta=1600$, $E(\mathbf u^*_h)=33.80227900547$ and $\frac12 \lambda_h^*= 40.44755720301$ in this paper, correspond to the case for $\beta=800$, $E_g=33.8023,$ and $\mu_g=40.4476$ in \cite{bao2006efficient}. See Figure \ref{figure-3d-example3}.

\begin{figure}[htbp]
	\begin{center} 
		\subfigure{\includegraphics[scale=0.32]{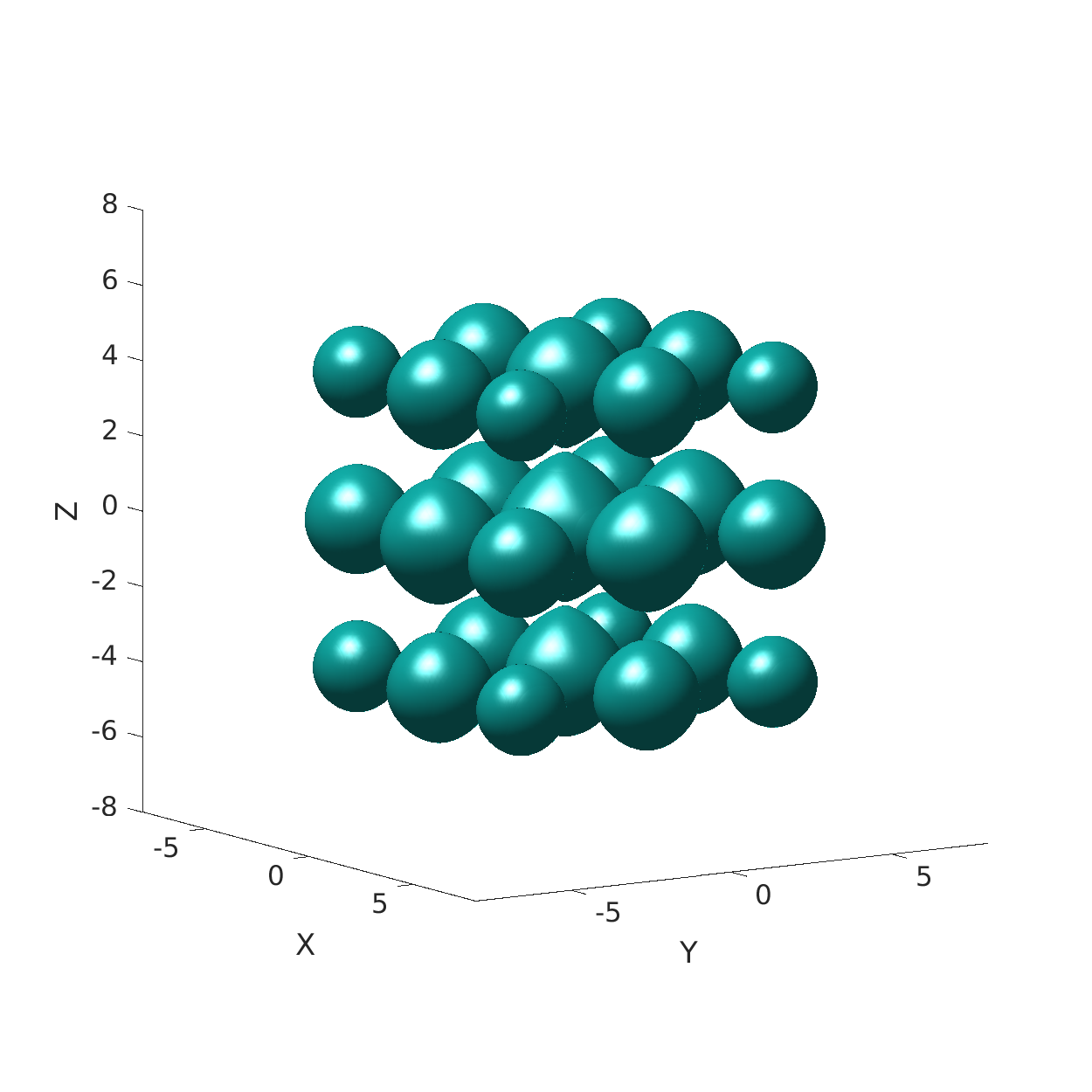} } 
		\hspace{-0.5cm}
		\subfigure{\includegraphics[scale=0.32]{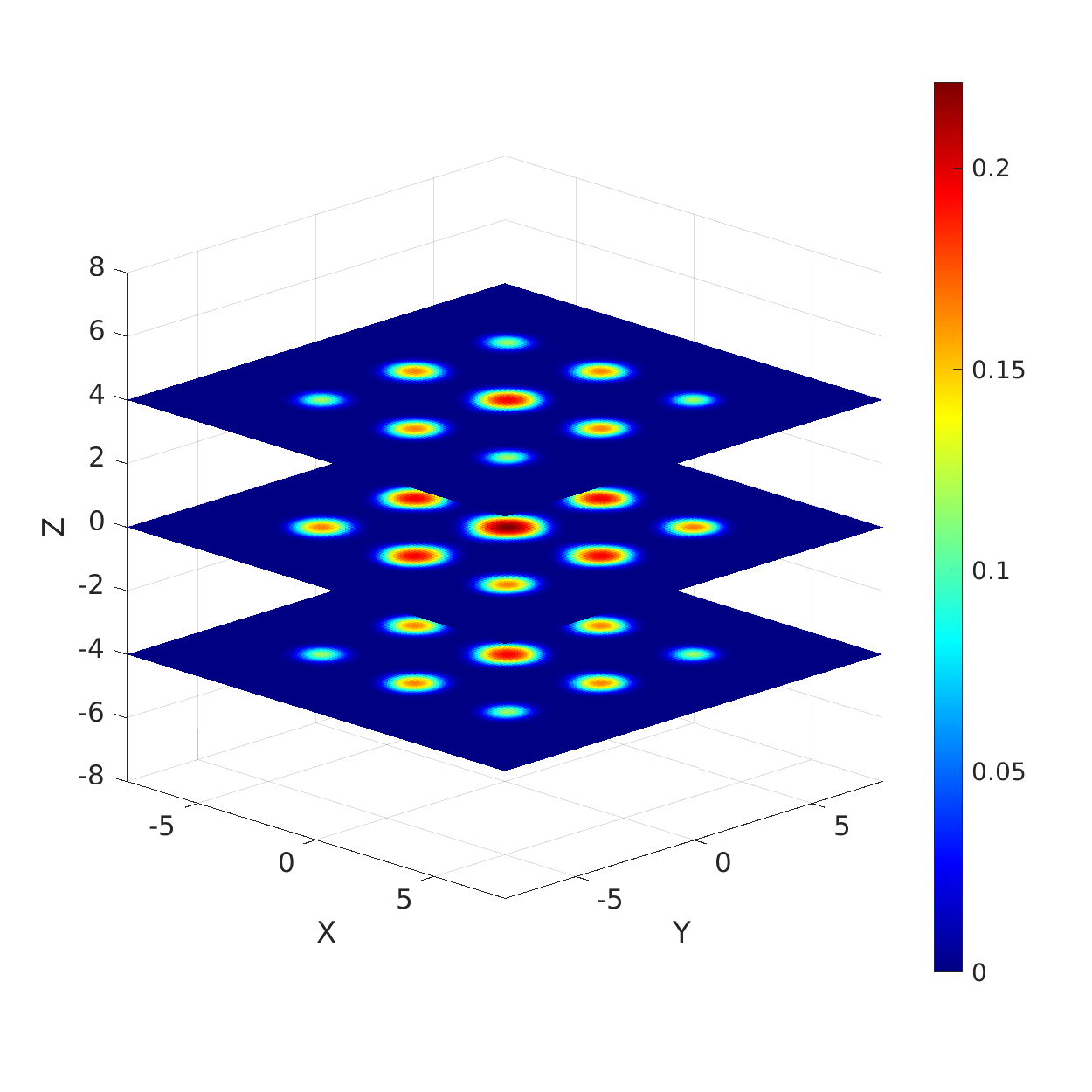} } 
    \end{center}
	\caption{A 3D example for a combined harmonic and optical lattice potential. Left is the isosurface of the ground state for isovalue $0.002$, and right is the slice view of the ground state.  For $\beta=1600$, using $Q^{40}$ spectral element method on a $5^3$ mesh, \eqref{modified-H1-scheme} with $\alpha=10$ and $\tau=0.1$ and $\mathbf u^0\equiv 1$ converges after 665 iterations. The online computation time is 6 seconds on Nvidia A100. $E(\mathbf u^*_h)=33.80227900547$ and $\lambda_h^*=80.89511440602$, consistent with the results in \cite{bao2006efficient}.}
	\label{figure-3d-example3}
\end{figure}

\section{Concluding remarks}
\label{sec:remarks}
We have considered the $H^1$ Sobolev gradient flow for finding the ground state of the Gross-Pitaevskii eigenvalue problem, under a modified $H^1$-norm. Global convergence to a critical point and the local exponential convergence rate have been established. 
Numerical experiments suggest that  the scheme with the spectral element method can be very efficient when using tuned parameters, which can be easily and efficiently implemented on modern GPUs.

\bibliographystyle{amsplain}
\bibliography{references}

\appendix

\section{Explicit finite difference formulation and discrete energy}
\label{apx:2nd_finite_diff_explicit}
	
We give explicit equivalent finite difference formulation of the $Q^k$ spectral element scheme \eqref{fd2}, especially the $Q^1$ case, which is equivalent to the second order finite difference.
We also give the explicit expressions of the $P^1$ finite element method on an unstructured mesh.
	
\subsection{The second order finite difference scheme}
	\label{appendix-FD}
For a one-dimensional uniform grid $-L=x_0<x_1<\cdots<x_n<x_{n+1}=L$ with grid spacing $h=\frac{2L}{n+1}$, for any vector $u_h\in V_0^h$ with $P^1$ polynomial basis, it can be represented by $\mathbf u=\begin{bmatrix} u_1 & u_2 & \cdots & u_n\end{bmatrix}^\top$ with $u_0=u_{n+1}=0$. The discrete  inner product is given by
\[\langle u_h, v_h\rangle=\frac12h u_0 v_0+\sum_{i=1}^n u_i v_i h+\frac12h u_n v_n=h \mathbf u^\top \mathbf v. \]
Define $ M=h \mathbb I_{n\times n}$  and $K=M^{-1}S$ where $S=\frac{1}{h}\begin{pmatrix} 2 & -1 & & & \\ -1 & 2 & -1 & &  \\ & \ddots & \ddots & \ddots & \\ & & -1 & 2 &-1 \\ & & & -1 & 2\end{pmatrix}_{n\times n}$. Then the matrices in \eqref{fd3} and \eqref{fd2} for 1D are $\mathbb M=M$, $\mathbb S=S$, and $-\Delta_h=K.$
	
For a two-dimensional problem on a uniform grid for the domain $[-L, L]^2$, assume there are $n\times n$ interior grid points. Let $U$, $V$ and $F$ denote 2D arrays of size $n\times n$ consisting of point values of $u_h(x_1, x_2), V(x_1, x_2), f(x_1, x_2)$ at grid points. Let $\mathrm{vec}(U)$ be the vector generated by arranging $U$ column by column. The scheme \eqref{fd3} becomes
\[ [S\otimes M + M\otimes S+\diag(\mathrm{vec}(V))] \mathrm{vec}(U)+\beta M\otimes M \mathrm{vec}(U^3)= \lambda_h M\otimes M \mathrm{vec}(U),\]
where $U^3$ denotes the entrywise cubic power. In 2D, \eqref{fd2} can be written as 
\[ [K\otimes I + I\otimes K+\diag(\mathrm{vec}(V))] \mathrm{vec}(U)+\beta \mathrm{vec}(U^3)=\lambda_h \mathrm{vec}(U).\]
With the property $(B^\top \otimes A)\mathrm{vec}(X) = \mathrm{vec}(AXB)$, it can be equivalently expressed as 
\[ KU+UK^\top+V\circ U+\beta U^3=\lambda_h U,\]
where $\circ$ represents Hadamard product, i.e., entrywise product. With similar notation as in \cite{liu2023simple}, the three-dimensional case of \eqref{fd2} can be expressed as
\[ [K\otimes I \otimes I+ I\otimes K\otimes I+I\otimes I\otimes K+\diag(\mathrm{vec}(V))] \mathrm{vec}(U)+\beta \mathrm{vec}(U^3)=\lambda_h \mathrm{vec}(U).\]
Thus, the matrix $-\Delta_h$ is given explicitly as follows:
\begin{equation}
    \label{2nd-Delta}   
	-\Delta_h=\begin{cases}
		K, & d=1, \\
		K\otimes I+I\otimes K, & d=2, \\
		K\otimes I\otimes I+I\otimes K\otimes I+I\otimes I\otimes K, & d=3, \\
	\end{cases}    
\end{equation}  
\begin{remark}
    The $Q^1$ scheme with quadrature gives exactly the same second-order centered difference for the interior grid. For Neumann boundary condition,  the $Q^1$ scheme with quadrature gives a slightly different scheme from a conventional finite difference scheme, see Remark 3.3 in \cite{10.1093/imanum/drac014}. When deriving finite difference from the finite element method, convergence is trivially implied by finite element error estimates.   
\end{remark}

\subsection{The $P^1$ finite element scheme on a simplicial mesh}
 \label{appendix-P1FEM}
Let $\Omega_h\subset\mathbb R^d$ be a simplicial mesh satisfying \eqref{simpicialmesh}. Let $\bx_i, i=1,\cdots, N$ denote all interior quadrature points. 
For any $u_h, v_h\in V_0^h$, since $\nabla u_h$ and $\nabla v_h$ are piecewise constants, the quadrature is exact. 
For an edge $E$ in a simplex $T$, let $\bx_i$ and $\bx_j$ be its ends and $\delta_E u_h$ denote $u_h(\bx_i)-u_h(\bx_j)=u_i-u_j$. 
By formulae in \cite{xu1999monotone}*{Section 2},   
$$\langle \nabla u_h, \nabla v_h \rangle=\int_{\Omega_h} \nabla u_h\cdot \nabla v_h \mathrm{d}\bx=\sum_{T\in \Omega_h}\sum_{E\subset T}\frac{1}{d(d-1)}|\kappa_E^T| \cot \theta_E^T \delta_E u_h \delta_E v_h,$$
which can also be written as
\begin{equation}
\label{discretegradient-P1FEM}
    \langle \nabla u_h, \nabla v_h \rangle= \sum_{E}\sum_{T\supset E} \frac{1}{d(d-1)}|\kappa_E^T| \cot \theta_E^T (u_i-u_j)(v_i-v_j),\quad \mbox{$\bx_i,\bx_j$ are two ends of the edge $E$}.
\end{equation}
By notation in Section \ref{sec-fem}, we have $\mathbf u^T \mathbb S\mathbf v=\langle \nabla u_h, \nabla v_h \rangle$, which implies that the stiffness matrix can be explicitly written as follows. The off-diagonal entries of $\mathbb S$ are given as
\[\mathbb S_{ij}= 
-\sum\limits_{T\supset E} \frac{1}{d(d-1)}|\kappa_E^T| \cot \theta_E^T, \quad i\neq j\quad \mbox{and $E$ is the edge  connecting $\bx_i$ and $\bx_j$.}
\]
And the diagonal entries of $\mathbb S$ can be obtained by the fact that each row sum of $\mathbb S$ should be zero.
\begin{wrapfigure}{r}{0pt}
\begin{tikzpicture}[scale=0.5]
    % draw the background
    \draw [line width=1.5pt, fill=gray!2] (0,0) -- (60:4) -- (4,0) -- cycle;

    \draw [line width=1.5pt, fill=gray!2]  (60:4) -- (4,0) --(7,2) -- cycle;

    \coordinate[label=left:]  (A) at (0,0);
    \coordinate[label=below:$\bx_j$] (B) at (4,0);
    \coordinate[label=above:$\bx_i$] (C) at (2,3.464);
    \coordinate[label=above: ] (D) at (7,2);

    % angle alpha
    \draw[fill=green!30] (0,0) -- (0:0.75cm) arc (0:60:.75cm);
    \draw (1.2cm,0.5cm) node {$\theta^1_{ij}$};
    
    \begin{scope}[shift={(7cm,2cm)}]
        \draw[fill=green!30] (0,0) -- (165:0.75cm) arc (165:210:0.75cm);
        \draw (-1.2cm,-0.2cm) node {$\theta^2_{ij}$};
    \end{scope}
    
    % the triangle
    \draw [line width=1.5pt] (A) -- (B) -- (C) -- cycle;
  \end{tikzpicture}
\end{wrapfigure}

In  a triangular mesh in two dimensions, for the edge connecting two interior vertices $\bx_i, \bx_j$, there are two angles $\theta^1_{ij}$ and $\theta^2_{ij}$ as shown in the figure, thus the stiffness matrix can be written as
\[\mathbb S_{ij}=\begin{cases}
-\frac{\cot \theta^1_{ij}+\cot \theta^2_{ij}}{2}, & \quad j\neq i,\\
-\sum_{j\neq i} \mathbb S_{ij}&\quad j=i.
\end{cases}\]
The necessary and sufficient condition for $\cot \theta^1_{ij}+\cot \theta^2_{ij}\geq 0$ is $\theta^1_{ij}+\theta^2_{ij}\leq \pi$.
For $\mathbb S$ to be an M-matrix, it suffices to have 
$\theta^1_{ij}+\theta^2_{ij}\leq \pi$, which can be achieved in a Delaunay triangulation.

\subsection{The discrete Laplacian from $Q^k$ scheme}

The full detail can be found in \cite{liu2023simple}. Let $S$ be the stiffness matrix and $M$ be the mass matrix $Q^k$ spectral element method in one dimension. In two dimensions \eqref{fd3} can be written as 
\[ [S\otimes M + M\otimes S+(M\otimes M) \diag(\mathrm{vec}(V))] \mathrm{vec}(U)=\lambda_h (M\otimes M)\mathrm{vec}(U).\]
Define $H=M^{-1}S$, the scheme \eqref{fd2} in two dimensions can  be written as 
\[ [H\otimes I + I\otimes H+\diag(\mathrm{vec}(V))] \mathrm{vec}(U)+\beta \mathrm{vec}(U^3)=\lambda_h \mathrm{vec}(U),\]
or equivalently $HU+UH^\top+V\circ U+\beta U^3=\lambda_h U,$
and in three dimensions it is  
\[ [H\otimes I \otimes I+ I\otimes H\otimes I+I\otimes I\otimes H+\diag(\mathrm{vec}(V))] \mathrm{vec}(U)+\beta \mathrm{vec}(U^3)=\lambda_h \mathrm{vec}(U).\]
It is possible to derive explicit entries of matrices $S, M, H$, see \cites{li2020superconvergence,shen2022discrete} for more details.
\begin{remark}
	For $k\geq 2$, the discrete Laplacian above give a $(k+2)$-th order finite difference scheme in discrete $l^2$-norm for solving elliptic equations \cite{li2020superconvergence} and for parabolic, wave and Schr\"odinger equations \cite{li2022accuracy}.   For solving a linear eigenvalue problem, e.g., $\beta=0$ in \eqref{fd2}, standard {\it a priori} error estimates for eigenvalues is that $Q^k$ spectral element method gives $2k$-th order of accuracy if assuming sufficient regularity. 
\end{remark}       
	
\subsection{Discrete energy of the second order finite difference}     
	The formula \eqref{discretegradient-P1FEM} can be  written out more explicitly when the vertices in the simplicial mesh form a Cartesian grid, i.e., when the $P^1$ finite element method becomes the the second order finite difference method. 
For the one-dimensional case, recall that $u_0=u_{n+1}=0$ for $u_h\in V_0^h$, we have
\[\forall~u_h, v_h\in V_0^h,\quad \langle \nabla u_h, \nabla v_h\rangle=h\sum_{i=1}^{n+1} \frac{u_{i}-u_{i-1}}{h}\cdot\frac{v_{i}-v_{i-1}}{h}.\]
Define the matrix $D=\frac{1}{h}\begin{pmatrix} 1 &  & &  \\ -1 & 1 &   &   \\ & \ddots & \ddots &  \\ & &  -1 & 1 \\ & & &   -1 \end{pmatrix}_{(n+1)\times n}.$ Then it satisfies $D^\top D=K$. In one dimension $\langle u_h, u_h \rangle=\mathbf u^\top \mathbb M\mathbf v=h \mathbf u^\top  \mathbf v$. Thus we have
\[  \langle \nabla u_h, \nabla v_h\rangle=\langle D\mathbf u, D\mathbf v\rangle_h=\langle  D^\top D \mathbf u, \mathbf v\rangle_h=\langle K\mathbf u, \mathbf v\rangle_h=\langle-\Delta_h \mathbf u,  \mathbf v\rangle_h. \]
In two dimensions, by plugging in the quadrature, for any $u_h, v_h\in V_0^h$, we have 
\begin{equation}
	\langle \nabla u_h, \nabla v_h\rangle=h^2\sum\limits_{i,j=1}^{n+1} \frac{u_{i,j}-u_{i-1,j}}{h}\cdot\frac{v_{i,j}-v_{i-1,j}}{h}+\frac{u_{i,j}-u_{i,j-1}}{h}\cdot\frac{v_{i,j}-v_{i,j-1}}{h}.\label{fem-grad} 
\end{equation}
With our notation for the two-dimensional problem, let $\mathbf u=\mathrm{vec}(U)$, then we have
\begin{align*}
    \langle \nabla u_h, \nabla v_h\rangle &= \langle DU, DV\rangle_h+ \langle UD, VD\rangle_h= \langle D^\top DU, V\rangle_h+ \langle UDD^\top, V\rangle_h,\\
    \langle \nabla u_h, \nabla v_h\rangle & =  \langle KU+UK, V\rangle_h=  \langle (K\otimes I+I\otimes K) \mathrm{vec}(U) , \mathrm{vec}(V)\rangle_h=\langle-\Delta_h \mathbf u,  \mathbf v\rangle_h.
\end{align*}
The three-dimensional case of the discrete gradient can be similarly written out.

\section{M-matrix and Perron-Frobenius theorem}
\label{sec-Mmatrix}
\subsection{M-matrix}
	
Nonsingular M-matrices are monotone. There are many equivalent definitions or characterizations of M-matrices, see \cite{plemmons1977m}. The following is a convenient sufficient but not necessary characterization of nonsingular M-matrices \cite{li2019monotonicity}:
\begin{theorem}
	\label{rowsumcondition-thm}
	For a real square matrix $A$  with positive diagonal entries and non-positive off-diagonal entries, $A$ is a nonsingular M-matrix if  all the row sums of $A$ are non-negative and at least one row sum is positive. 
\end{theorem}
	
\subsection{Irreducible nonnegative matrices}
\label{sec-PF}
A matrix $A\in \mathbb C^{n\times n}$ is called {\it reducible} if there exists a permutation matrix $P$ such that $PAP^\top$ is block upper triangular. A matrix is irreducible if and only the graph it represents is strongly connected. 
	
\begin{lemma}
	For a nonsingluar irreducible matrix $A$, $A^{-1}$ is also irreducible.
\end{lemma}
	
The following results can be found in \cite{varga1999matrix}:

\begin{theorem}[Perron-Frobenius]
	\label{Perron-Frobenius}
	If $A\geq 0$ is irreducible, then:
	\begin{enumerate}
		\item The spectral radius $\rho(A)$ is a simple eigenvalue of $A$ with an eigenvector $x>0$. 
		\item $\rho(A)$ increases when any entry of $A$ increases. 
	\end{enumerate}
\end{theorem}
	
\begin{theorem}
	\label{Perron-Frobenius-2}
	The positive eigenvector (Perron-Frobenius eigenvector)  for an irreducible nonnegative matrix is unique up to scalar multiplication.  
\end{theorem}
\begin{proof}
	Let $x>0$ be the left Perron-Frobenius eigenvector then $x^\top A=\rho(A) x^\top$. If there exists another eigenvector $y>0$ for an eigenvalue $\lambda$, then $Ay=\lambda y\implies x^\top A y=\lambda x^\top y$. Since $x^\top A=\rho(A) x^\top \implies x^\top A y=\rho(A) x^\top y$, we get $(\rho(A)-\lambda) x^\top y=0$ and $x^\top y>0\implies \rho(A)=\lambda$. Thus there is only one eigenvalue, i.e., $\rho(A)$, with positive eigenvectors, and $\rho(A)$ is a simple eigenvalue by Theorem~\ref{Perron-Frobenius}.
\end{proof}
	
\section{Deferred proofs for Section~\ref{sec:global_converge}}
\label{apx:pf_global}

\subsection{Norm equivalence and standard regularity results}
 
\begin{lemma}
    There are positive constants $D_1, D_2, D_3$ independent of mesh size $h$, such that the followings hold for any $v_h\in V_0^h$:
    \begin{equation}
        \frac{1}{D_p} |\langle v_h^p, v_h^p\rangle |^{\frac{1}{2p}} \leq \|v_h\|_{L^{2p}(\Omega)}   \leq D_p |\langle v_h^p, v_h^p\rangle |^{\frac{1}{2p}},\quad p=1,2,3.\label{discretpnorm-equivalence}
    \end{equation}
\end{lemma}

\begin{proof}
    We first prove it for $Q^k$ spectral element method.
    Consider a cubic cell $e=[x_1^e-\frac{h}{2}, x_1^e+\frac{h}{2}]\times \cdots \times [x_d^e-\frac{h}{2}, x_d^e+\frac{h}{2}]\in\Omega_h$ and a reference cell $\hat K=[-1,1]^3$. For $v_h({\bf x})$ defined on $e$, consider $\hat v_h({\bf t})=v_h\left(t_1\frac{h}{2}+x_1^e,\cdots,t_d\frac{h}{2}+x_d^e\right)$, which is defined on $\hat K$. Let $\langle \hat v_h, \hat v_h \rangle_{\hat K}$ denote the approximation to the integral $\int_{\hat K} |\hat v_h({\bf t})|^2\mathrm{d}{\bf t}$ by $(k+1)$-point Gauss Lobatto quadrature for each variable. Since both $\sqrt{\langle \hat v_h, \hat v_h \rangle_{\hat K}}$ and  $\sqrt{\int_{\hat K} |\hat v_h({\bf t})|^2\mathrm{d}{\bf t}}$ are norms of  $Q^k(\hat K)$, by the equivalence of any two norms on the finite-dimensional space $Q^k(\hat K)$, we have 
    \[ \forall~\hat v_h\in Q^k(\hat K),\quad \frac{1}{D_1} \sqrt{\langle \hat v_h, \hat v_h \rangle_{\hat K}}\leq \sqrt{\int_{\hat K}  |\hat v_h({\bf t})|^2\mathrm{d}{\bf t}}\leq D_1\sqrt{\langle \hat v_h, \hat v_h \rangle_{\hat K}}.\]
    By mapping back to $e$, and summing over $e$, we get \eqref{discretpnorm-equivalence} for $p=1$.

    With the same notation and arguments above, for a $Q^k$ polynomial $v_h$ with $(k+1)$-point Gauss-Lobatto quadrature, let $\hat w_i$ and $\hat v_i$ be quadrature weights and quadrature node values on the reference cell $\hat K$, then we have $|\langle \hat v^2_h, \hat v^2_h \rangle_{\hat K}|^{\frac14}=\left|\sum_{i=1}^{(k+1)^d} \hat w_i \hat{  v}^4_i\right|^{\frac14}$, which can be easily verified to be a norm of $\mathbb R^{(k+1)^d}$ thus a norm of $Q^k(\hat K)$. By the equivalence of any two norms on $Q^k(\hat K)$, we have
    \[ \forall~\hat v_h\in Q^k(\hat K),\quad \frac{1}{D_2} \left|\langle \hat v^2_h, \hat v^2_h \rangle_{\hat K}\right|^{\frac{1}{4}}\leq \|\hat v_h \|_{L^4(\hat K)} \leq D_2 \left|\langle \hat v^2_h, \hat v^2_h \rangle_{\hat K}\right|^{\frac{1}{4}}.\]
    By mapping back to $e$, and summing over $e$, we get \eqref{discretpnorm-equivalence} for $p=2$. The proof of $p=3$ is almost identical to the case $p=3$.
    Finally, the same proof also applies to  $P^k$ finite element method on a simplicial mesh with quadrature.
\end{proof}

Second, we want to show the $X$-norm is equivalent to $H^1$-norm for piecewise polynomials in $V_0^h$. By Lemma 5.1 in \cite{li2020superconvergence}, we have the following standard $V^h$-ellipticity result for a domain  on which elliptic regularity holds, e.g., $\Omega=[-L, L]^d$.
\begin{lemma}
    Let $\|v_h\|_X = \sqrt{\langle \nabla v_h , \nabla v_h \rangle+\alpha \langle  v_h ,  v_h \rangle} = \|\mathbf v\|_X$. 
    Let $\Omega$ be a domain with elliptic regularity. There is a constant $D_4>0$ independent of mesh size $h$ s.t.
    \begin{equation}
        \label{Xnorm-equivalency}
        \forall~v_h\in V_0^h, \quad \frac{1}{D_4} \|v_h\|_X \leq \|v_h\|_{H^1(\Omega)} \leq D_4 \|v_h\|_X.
    \end{equation} 
\end{lemma}

\subsection{Proofs of Lemma~\ref{lem:esti_retraction} and Lemma~\ref{lemma-norm}}

\begin{proof}[Proof of Lemma~\ref{lem:esti_retraction}]
	Since $\langle \mathbf u,\mathbf v\rangle_h =0$, we have that $\| \mathbf u+\mathbf v\|_2^2 = \| \mathbf u\|_2^2 + \| \mathbf v\|_2^2  = 1 + \| \mathbf v\|_2^2$. This implies that
	\begin{equation*}
		R_h(\mathbf u+\mathbf v) - (\mathbf u+\mathbf v) = \left(\frac{1}{\| \mathbf u+\mathbf v\|_2}-1\right)(\mathbf u+\mathbf v) = \left(\left( 1 +  \| \mathbf v\|_2^2\right)^{-\frac{1}{2}} - 1\right)(\mathbf u+\mathbf v).
	\end{equation*}
	Then \eqref{esti_retraction} follows from the elementary inequality $1-\frac{x}{2}\leq (1+x)^{-1/2}\leq 1,\ x\geq 0$.
\end{proof}

\begin{proof}[Proof of Lemma~\ref{lemma-norm}] 
    (i) $\norm{\mathbf u}_X^2 = \mathbf u^\top (\mathbb S+\alpha \mathbb M) \mathbf u \geq \alpha \mathbf u^\top  \mathbb M \mathbf u =\alpha \norm{\mathbf u}_2^2$.
		
	(ii)  By Cauchy-Schwartz inequality for the inner product $\langle\cdot, \cdot\rangle_h$ and (i), we get
	\[ \resizebox{\hsize}{!}{$\norm{(-\Delta_h+\alpha \mathbb I)^{-1}\mathbf u}_X^2 = \left\langle (-\Delta_h+\alpha \mathbb I)^{-1}\mathbf u, \mathbf u\right\rangle_h  \leq   \norm{(-\Delta_h+\alpha \mathbb I)^{-1}\mathbf u}_2 \norm{\mathbf u}_2 \leq \frac{1}{\sqrt{\alpha}}\norm{(-\Delta_h+\alpha \mathbb I)^{-1}\mathbf u}_X \norm{\mathbf u}_2.$}\]
  
    (iii) Since $\Pi(\mathbf u^2)$ coincides with $u_h^2$ at the quadrature nodes, by \eqref{discretpnorm-equivalence} for $p=2$, 
	$$\|\mathbf u^2\|_2=\sqrt{\langle \mathbf u^2,\mathbf u^2\rangle_h}=\sqrt{\langle \Pi(\mathbf u^2), \Pi(\mathbf u^2)\rangle }=\sqrt{\langle u_h^2, u_h^2\rangle }\leq D_2^2 \|u_h\|^2_{L^4(\Omega)}.$$
	With the Sobolev embedding $H^1(\Omega)\subset L^4(\Omega)$ for dimension $d\leq 3$, and \eqref{Xnorm-equivalency}, we get
    \[ \|u_h\|_{L^4(\Omega)}^2\leq D^2 \|u_h\|_{H^1(\Omega)}^2 \leq D^2 D_4^2\|\mathbf u\|_{X}^2, \]
    where $D>0$ is the constant associated with the embedding $H^1(\Omega)\subset L^4(\Omega)$.

    (iv) The proof is almost identical to that of (iii), which uses \eqref{discretpnorm-equivalence} for $p=3$, \eqref{Xnorm-equivalency}, and the Sobolev embedding $H^1(\Omega)\subset L^6(\Omega)$ for dimension $d=1,2,3$.
\end{proof}

\section{Deferred proofs for Section~\ref{sec:local_converge}}
\label{sec:appendix-D}
\begin{proof}[Proof of Lemma~\ref{lem:local_convex}]
    Since $A_{\mathbf u^*}$ is self-adjoint w.r.t. $\langle\cdot,\cdot\rangle_h$, it has orthonormal eigenvectors. Let $\mathbf u_\parallel = \langle \mathbf u, \mathbf u^*\rangle_h\mathbf u^*$ be the orthogonal projection of $\mathbf u$ onto the subspace spanned by $\mathbf u^*$. Let $\mathbf u_\perp = \mathbf u - \mathbf u_\parallel$, then $\langle \mathbf u_\parallel,\mathbf u_\perp\rangle_h = 0$ and $\|\mathbf u_\parallel\|_2^2 + \|\mathbf u_\perp\|_2^2 = 1$. Thus, 
    \begin{align*}
        &\mathbf E_h(\mathbf u) - \mathbf E_h(\mathbf u^*) =\frac12 \langle-\Delta_h \mathbf u,  \mathbf u\rangle_h +\frac12  \langle\mathbb V \mathbf u,  \mathbf u\rangle_h+\frac{\beta}{4}  \langle \mathbf u^2,  \mathbf u^2\rangle_h \\
        &\quad\qquad - \frac12 \langle-\Delta_h \mathbf u^*,  \mathbf u^*\rangle_h - \frac12  \langle\mathbb V \mathbf u^*,  \mathbf u^*\rangle_h - \frac{\beta}{4}  \langle (\mathbf u^*)^2,  (\mathbf u^*)^2\rangle_h \\
        &\quad\geq   \frac12 \langle-\Delta_h \mathbf u,  \mathbf u\rangle_h +\frac12  \langle\mathbb V \mathbf u,  \mathbf u\rangle_h + \frac{\beta}{2}\langle (\mathbf u^*)^2 ,\mathbf u^2\rangle_h \\
        &\quad\qquad - \frac12 \langle-\Delta_h \mathbf u^*,  \mathbf u^*\rangle_h - \frac12  \langle\mathbb V \mathbf u^*,  \mathbf u^*\rangle_h - \frac{\beta}{2}  \langle (\mathbf u^*)^2,  (\mathbf u^*)^2\rangle_h \\
        &\quad= \frac{1}{2} \langle A_{\mathbf u^*} \mathbf u,\mathbf u\rangle_h - \frac{1}{2} \langle A_{\mathbf u^*} \mathbf u^*,\mathbf u^*\rangle_h=\frac{1}{2} \langle A_{\mathbf u^*} \mathbf u_\parallel,\mathbf u_\parallel\rangle_h +\frac{1}{2} \langle A_{\mathbf u^*} \mathbf u_\perp,\mathbf u_\perp\rangle_h - \frac{\lambda^0_h}{2}  \\
        &\quad\geq  \frac{\lambda^0_h}{2} \norm{\mathbf u_\parallel}_2^2 + \frac{\lambda^1_h}{2}\norm{\mathbf u_\perp}_2^2 - \frac{\lambda_0}{2}= \frac{\lambda^1_h - \lambda^0_h}{2}\norm{\mathbf u_\perp}_2^2.
    \end{align*}
    With the following fact, \eqref{eq:local_convex} follows from the estimate above:
    \begin{align*}
        \norm{\mathbf u_\perp}_2^2 = 1 - \norm{\mathbf u_\parallel}_2^2 = 1 - \left|\langle \mathbf u,\mathbf  u^*\rangle_h\right|^2 = 1 - \frac{1}{4}\left(2 - \norm{\mathbf u - \mathbf u^*}_2^2\right)^2 = \norm{\mathbf u - \mathbf u^*}_2^2  - \frac{1}{4}\norm{\mathbf u - \mathbf u^*}_2^4.
    \end{align*}
\end{proof}

\begin{lemma}\label{lem:lipchitz-u3}
    There is some constant $L_1>0$ independent of the mesh size $h$ s.t.
    \begin{equation*}
        \forall~\mathbf u,\mathbf v\in\bR^N,\quad  \norm{(-\Delta_h + \alpha \mathbb I)^{-1} (\mathbf u^3 - \mathbf v^3)}_X \leq L_1 \norm{\mathbf u-\mathbf v}_X \left(\norm{\mathbf u}_X^4+\norm{\mathbf v}_X^4\right).
    \end{equation*}
\end{lemma}

\begin{proof}
    Using Lemma~\ref{lemma-norm} (ii) and (iv), we can compute that
    \begin{align*}
        & \norm{(-\Delta_h + \alpha \mathbb I)^{-1} (\mathbf u^3 - \mathbf v^3)}_X^2 
        \leq  \frac{1}{\alpha} \norm{\mathbf u^3 - \mathbf v^3}_2^2 = \frac{1}{\alpha}\sum_i w_i(u_i^3 - v_i^3 )^2 \\
        &\quad= \frac{1}{\alpha} \sum_i w_i(u_i - v_i )^2 (u_i^2 + u_i v_i + v_i^2)^2 
        \leq   \frac{8}{\alpha} \sum_i w_i(u_i - v_i )^2 u_i^4 + \frac{8}{\alpha} \sum_i w_i(u_i - v_i )^2 v_i^4 \\
        &\quad\leq \frac{8}{\alpha} \left(\sum_i w_i(u_i - v_i )^6\right)^{1/3} \left(\sum_i w_i u_i^6\right)^{2/3} + \frac{8}{\alpha} \left(\sum_i w_i(u_i - v_i )^6\right)^{1/3} \left(\sum_i w_i v_i^6\right)^{2/3} \\
        &\quad= \frac{8}{\alpha} \norm{(\mathbf u - \mathbf v)^3}_2^{2/3}\left(\norm{\mathbf u^3}_2^{4/3} + \norm{\mathbf v^3}_2^{4/3}\right)   
        \leq   \frac{8 C_2^2}{\alpha} \norm{\mathbf u-\mathbf v}_X^2 \left(\norm{\mathbf u}_X^4+\norm{\mathbf v}_X^4\right).
    \end{align*}
\end{proof}

\begin{lemma}\label{lem:lipchitz_nablaX}
    For any $\mathbf u,\mathbf v\in\bR^N$ with $\|\mathbf u\|_X\leq 2 \|\mathbf u^*\|_X$ and $\|\mathbf v\|_X\leq 2 \|\mathbf u^*\|_X$, it holds for some constant $L_2>0$ depending on $\|\mathbf u^*\|_X$ and independent of the mesh size $h$ that $ \norm{\nabla_X \mathbf E_h(\mathbf u) - \nabla_X \mathbf E_h(\mathbf v)}_X \leq L_2 \norm{\mathbf u - \mathbf v}_X.$
\end{lemma}

\begin{proof}
Recall that $\nabla_X \mathbf E_h(\mathbf u) = \mathbf u + (-\Delta_h+\alpha \mathbb  I)^{-1} (\mathbb V-\alpha\mathbb I + \beta \text{diag}(\mathbf u)^2 )\mathbf u$. Using Lemma~\ref{lem:lipchitz-u3} and Lemma~\ref{lemma-norm}, we can conclude that 
\begin{align*}
    & \norm{\nabla_X \mathbf E_h(\mathbf u) - \nabla_X \mathbf E_h(\mathbf v)}_X \\
    &\quad\leq \norm{\mathbf u-\mathbf v}_X + \norm{(-\Delta_h + \alpha\mathbb I)^{-1} (\mathbb V-\alpha \mathbb I)(\mathbf u - \mathbf v)}_X + \beta \norm{(-\Delta_h + \alpha\mathbb I)^{-1}(\mathbf u^3 - \mathbf v^3)}_X \\
    &\quad\leq \norm{\mathbf u-\mathbf v}_X + \frac{V_{\alpha,\max}}{\sqrt{\alpha}} \norm{\mathbf u-\mathbf v}_2 + \beta L_1 \norm{\mathbf u-\mathbf v}_X \left(\norm{\mathbf u}_X^4+\norm{\mathbf v}_X^4\right) \\
    &\quad\leq\left(1 + \frac{V_{\alpha,\max}}{\alpha} + 32 \beta L_1 \|\mathbf u^*\|_X^4\right)\norm{\mathbf u-\mathbf v}_X.
\end{align*}
\end{proof}

\begin{proof}[Proof of Lemma~\ref{lem:lipchitz-grad}]
Recall $\nabla_X^\calR \mathbf E_h(\mathbf u) = \nabla_X \mathbf E_h(\mathbf u)-\frac{\langle \mathbf u,\nabla_X \mathbf E_h(\mathbf u)\rangle_h}{\langle\mathbb G_X\mathbf u, \mathbf u\rangle_h} \mathbb G_X\mathbf u$. Set
\begin{equation*}
    \gamma = \frac{\langle \mathbf u,\nabla_X \mathbf E_h(\mathbf u)\rangle_h}{\langle\mathbb G_X\mathbf u, \mathbf u\rangle_h}\quad\text{and}\quad \gamma^* = \frac{\langle \mathbf u^*,\nabla_X \mathbf E_h(\mathbf u^*)\rangle_h}{\langle\mathbb G_X\mathbf u^*, \mathbf u^*\rangle_h}
    = \frac{\langle \mathbf u^*, \mathbb G_X A_{\mathbf u^*}\mathbf u^*\rangle_h}{\langle\mathbb G_X\mathbf u^*, \mathbf u^*\rangle_h}=\lambda^0_h.
\end{equation*}
Suppose that $\norm{\mathbf u - \mathbf u^*}_X\leq \norm{\mathbf u^*}_X$, then $\norm{\mathbf u}_X\leq 2 \norm{\mathbf u^*}_X$. With Lemma~\ref{lemma-norm} and Lemma~\ref{lem:lipchitz_nablaX}, it holds that
\begin{align*}
    \norm{\nabla_X^\calR \mathbf E_h(\mathbf u)}_X = & \norm{\nabla_X^\calR \mathbf E_h(\mathbf u) - \nabla_X^\calR \mathbf E_h(\mathbf u^*)}_X\\
    = & \norm{\nabla_X \mathbf E_h(\mathbf u) - \gamma \mathbb G_X\mathbf u - \nabla_X \mathbf E_h(\mathbf u^*) + \gamma^* \mathbb G_X\mathbf u^*}_X\\
    \leq & \norm{\nabla_X \mathbf E_h(\mathbf u) - \nabla_X \mathbf E_h(\mathbf u^*)}_X + |\gamma - \gamma^*|\cdot\norm{\mathbb G_X\mathbf u}_X + |\gamma^*| \cdot\norm{\mathbb G_X\mathbf u - \mathbb G_X\mathbf u^*}_X\\
    \leq & L_2 \norm{\mathbf u-\mathbf u^*}_X + \frac{\|\mathbf u^*\|_X}{\alpha} |\gamma - \gamma^*| + \frac{|\gamma^*|}{\alpha} \norm{\mathbf u-\mathbf u^*}_X.
\end{align*}
Thus it suffices to estimate $|\gamma - \gamma^*|$. Set
\begin{equation*}
    A = \langle \mathbf u,\nabla_X \mathbf E_h(\mathbf u)\rangle_h,~~  A^* = \langle \mathbf u^*,\nabla_X \mathbf E_h(\mathbf u^*)\rangle_h,~~ B = \langle\mathbb G_X\mathbf u, \mathbf u\rangle_h,~~ B^* = \langle\mathbb G_X\mathbf u^*, \mathbf u^*\rangle_h.
\end{equation*}
By Lemma~\ref{lemma-norm}, Lemma~\ref{lem:lipchitz_nablaX}, and $\|\mathbf u\|_2=\|\mathbf u^*\|_2=1$, we have
\begin{align*}
    |A - A^*| \leq & \norm{\mathbf u-\mathbf u^*}_2 \norm{\nabla_X \mathbf E_h(\mathbf u^*)}_2 + \norm{\mathbf u}_2 \norm{\nabla_X \mathbf E_h(\mathbf u)- \nabla_X \mathbf E_h(\mathbf u^*)}_2 \\
    \leq &\frac{\norm{\nabla_X \mathbf E_h(\mathbf u^*)}_X}{\alpha} \norm{\mathbf u-\mathbf u^*}_X + \frac{L_2}{\sqrt{\alpha}}\norm{\mathbf u}_X \norm{\mathbf u-\mathbf u^*}_X,\\
     |B-B^*| \leq & \norm{\mathbf u-\mathbf u^*}_2 \norm{\mathbb G_X \mathbf u}_2 + \norm{\mathbf u^*}_2 \norm{\mathbb G_X \mathbf u - \mathbb G_X \mathbf u^*}_2 \leq \frac{2}{\alpha^{3/2}} \norm{\mathbf u-\mathbf u^*}_X,\\
    |\gamma - \gamma^*| = & \left|\frac{A}{B} - \frac{A^*}{B^*}\right| \leq \frac{|A-A^*|\cdot|B^*| + |A^*|\cdot |B-B^*|}{|B B^*|} \leq L_\gamma \norm{\mathbf u-\mathbf u^*}_X,
\end{align*}
for small $\norm{\mathbf u-\mathbf u^*}_X$ and some constant $L_\gamma>0$, which completes the proof.
\end{proof}

\section{Positive eigengap independent of the mesh size}
\label{apx:eigengap}

In this section, we prove that Assumption~\ref{asp:ustar} holds with a positive eigengap independent of the mesh size for the $P^1$ finite element method with quadrature on an unstructured simplicial mesh with shape regularity.

\begin{definition}[Shape regular meshes]
    A simplicial mesh with $h_K$ denoting the size of each simplex $K$ is shape regular if there exists some constant $\sigma>0$ independent of the mesh size $h = \sup_K h_K$, such that $\frac{h_K}{\rho_K}\leq \sigma$ holds for all $K$, where $\rho_K$ is the diameter of the largest sphere contained in $K$.
\end{definition}

\subsection{Technical lemmas}

\begin{lemma}\label{lem:error_interpolation}
    Suppose that $V\in C^1(\Bar{\Omega})$. There exist constants $E_1,E_2,E_3>0$ independent of the mesh size, so that for any $h>0$ and $\mathbf u\in\bR^N$, the followings hold for $P^1$ finite element method with quadrature on an unstructured simplicial mesh with shape regularity:
    \begin{itemize}
        \item[(i)] $0\leq \|\mathbf u\|_2^2 - \|\Pi\mathbf u\|_{L^2(\Omega)}^2 \leq E_1 h \langle -\Delta_h \mathbf u,\mathbf u\rangle_h$.
        \item[(ii)] $\left|\langle \mathbf u,\mathbb V\mathbf u\rangle_h - \int_\Omega V |\Pi\mathbf u|^2 \right| \leq E_2 h \langle -\Delta_h\mathbf u,\mathbf u\rangle_h +E_3 h \|\mathbf u\|_2^2$.
        \item[(iii)] $\langle-\Delta_h \mathbf u,\mathbf u\rangle_h = |\Pi\mathbf u|^2_{H_0^1(\Omega)}$.
        \item[(iv)] $\langle \mathbf u^2, \mathbf u^2\rangle_h \geq \|\Pi\mathbf u\|_{L^4(\Omega)}^4$. 
    \end{itemize}
\end{lemma}

\begin{proof}
    It suffices to prove the results on each simplex $K$ of the mesh since all of them are additive. Consider a single simplex $K$. For convenience, by abusing notation, we let all vertices of $K$ be indexed as $\{\bx_i, i=1,\cdots, d+1\}$ and the $P^1$ interpolation polynomial $(\Pi\mathbf u)(\mathbf x) = a+\mathbf b^\top \mathbf x$ on $K$ satisfying $(\Pi\mathbf u)(\mathbf x_i)=u_i$. 
   Define the barycentric coordinates as $\sum\limits_{i=1}^{d+1}t_i=1, t_i\geq 0$ for the simplex $K$, i.e., $\bx =\sum\limits_{i=1}^{d+1} t_i\bx_i,\forall \bx\in K$.  It is straightforward to verify that the polynomial $[\Pi\mathbf u(\bx)]^2$ is a convex function of $\bx$, thus  we have Jensen's inequality
    \[ (\Pi\mathbf u)^2\left(\sum\limits_{i=1}^{d+1}  t_i \mathbf x_i\right)  \leq \sum\limits_{i=1}^{d+1}  t_i (\Pi\mathbf u)^2(\mathbf x_i).\]

 Conider a barycentric  simplex $\Delta_I = \{ (t_1,\cdots, t_{d+1})\in[0,1]^{d+1}: \sum_{i\in I} t_i=1\}\subset \mathbb R^{d+1}$. By regarding $\Delta_I$ as an embedded manifold in $\mathbb R^{d+1}$, its volume $|\Delta_I|$ can be written as an integral of constant $1$: $$|\Delta_I|=\int_{\Delta_I} 1 \mathrm{d}V=\int_{\Delta_I} \left(\sum\limits_{i=1}^{d+1}t_i \right) \mathrm{d}V=(d+1)\int_{\Delta_I} t_j  \mathrm{d}V,\quad \mbox{for any fixed } j,$$  
where $\mathrm{d}V$ is the volume form of the manifold and in the last step we have used the symmetry of integrating each $t_j$ on the simplex  $\Delta_I$. Since $\frac{1}{d+1}=\frac{1}{|\Delta_I|}\int_{\Delta_I} \left(t_j \right)  \mathrm{d}V, \quad \forall j$, we have
\begin{equation}\label{eq:Jenson_simplex_K}
        \begin{split}
            \frac{|K|}{d+1}\sum\limits_{i=1}^{d+1} u_i^2 & = \frac{|K|}{d+1}\sum\limits_{i=1}^{d+1} (\Pi\mathbf u)^2(\mathbf x_i) = \frac{|K|}{|\Delta_I|} \int_{\Delta_I}\sum\limits_{i=1}^{d+1} t_i (\Pi\mathbf u)^2(\mathbf x_i) \mathrm{d}V\\
            &  \geq \frac{|K|}{|\Delta_I|} \int_{\Delta_I} (\Pi\mathbf u)^2\left(\sum\limits_{i=1}^{d+1} t_i \mathbf x_i\right) \mathrm{d}V = \int_K (\Pi\mathbf u)^2(\mathbf x) \mathrm{d}\mathbf x.
        \end{split}
    \end{equation}
    Summing the inequality above over all simplexes in the mesh, 
the first inequality in (i) is proven. Note that the quadrature $\frac{|K|}{d+1} \sum_{i=1}^{d+1} f(\bx_i)=\frac{|K|}{d+1} \sum_{\mathbf x_i\in K} f(\bx_i)$  is equal to $\int_K f(\bx) \mathrm{d}\bx$ for any linear function $f(\bx)$. 
Let    $E_1 = 2 \sup_{\mathbf x\in \Omega}\|\mathbf x\|$, then
we have 
\begin{equation}\label{appendixc-estimate1}
    \begin{split}
        & \frac{|K|}{d+1} \sum_{\mathbf x_i\in K} (a+\mathbf b^\top\mathbf x_i)^2 - \int_K (a+\mathbf b^\top\mathbf x)^2 \mathrm{d}\mathbf x  = \frac{|K|}{d+1} \sum_{\mathbf x_i\in K}  (\mathbf b^\top\mathbf x_i)^2 - \int_K (\mathbf b^\top\mathbf x)^2 \mathrm{d}\mathbf x \\
        \leq & |K| \sup_{\mathbf x,\mathbf x'\in K} \left((\mathbf b^\top\mathbf x)^2- (\mathbf b^\top\mathbf x')^2\right)=|K| \sup_{\mathbf x,\mathbf x'\in K} \left(\mathbf b^\top\mathbf x -  \mathbf b^\top\mathbf x' \right)\left( \mathbf b^\top\mathbf x+\mathbf b^\top\mathbf x'\right)   \leq E_1 h |K| \|\mathbf b\|^2.
    \end{split}
\end{equation}  
    With notation in \eqref{discreteL2norm}, the discrete integration by parts \eqref{integrationbyparts} can be also be written as 
    \begin{equation}
        \label{appendixc-esitmate3}
        \langle -\Delta_h \mathbf u,\mathbf u\rangle_h=\langle \nabla u_h, \nabla u_h  \rangle=\langle \nabla \Pi\mathbf u (\bx), \nabla \Pi\mathbf u (\bx)  \rangle.
    \end{equation}
    Notice that $\nabla\Pi\mathbf u(\mathbf x) = \mathbf b$ is a constant over $K$, thus after summing over $K$, the term $|K|\|\mathbf b\|^2$ becomes $\langle \nabla \Pi\mathbf u (\bx), \nabla \Pi\mathbf u (\bx)  \rangle=\langle -\Delta_h \mathbf u,\mathbf u\rangle_h.$
By the same notation, we also have 
\begin{equation}
    \label{P1FEM-norms}
    \sum_K \frac{|K|}{d+1} \sum_{\mathbf x_i\in K} (a+\mathbf b^\top\mathbf x_i)^2=
    \langle  \Pi\mathbf u (\bx),   \Pi\mathbf u (\bx)  \rangle=\langle   \mathbf u,\mathbf u\rangle_h, \quad\sum_K  \int_K (a+\mathbf b^\top\mathbf x)^2 \mathrm{d}\mathbf x=\|\Pi\mathbf u\|_{L^2(\Omega)}^2,
\end{equation}
which proves the second inequality of (i).

Notice that \eqref{eq:Jenson_simplex_K} can also be written 
as \begin{equation}
    \label{appendixc-estimate2}
     \int_K (a+\mathbf b^\top\mathbf x)^2 \mathrm{d}\mathbf x \leq \frac{|K|}{d+1} \sum_{\mathbf x_i\in K} (a+\mathbf b^\top\mathbf x_i)^2.
\end{equation}
    
    For proving  (ii), with some fixed $\tilde{\mathbf x}\in K$, we have
    \begin{align*}
        & \left|\frac{|K|}{d+1} \sum_{\mathbf x_i\in K} V(\mathbf x_i)(a+\mathbf b^\top\mathbf x_i)^2 - \int_K V(\mathbf x)(a+\mathbf b^\top\mathbf x)^2 \mathrm{d}\mathbf x\right| \\
        \leq & |V(\tilde{\mathbf x})| \cdot \left|\frac{|K|}{d+1} \sum_{\mathbf x_i\in K} (a+\mathbf b^\top\mathbf x_i)^2 - \int_K (a+\mathbf b^\top\mathbf x)^2 \mathrm{d}\mathbf x\right| \\
        & \quad\qquad + \frac{|K|}{d+1} \sum_{\mathbf x_i\in K} |V(\mathbf x_i)-V(\tilde{\mathbf x})|(a+\mathbf b^\top\mathbf x_i)^2 + \int_K |V(\mathbf x) - V(\tilde{\mathbf x})|(a+\mathbf b^\top\mathbf x)^2 \mathrm{d}\mathbf x\\
        \leq & \|V\|_{L^\infty(\Bar{\Omega})} \left|\frac{|K|}{d+1} \sum_{\mathbf x_i\in K} (a+\mathbf b^\top\mathbf x_i)^2 - \int_K (a+\mathbf b^\top\mathbf x)^2 \mathrm{d}\mathbf x\right| \\
        &\quad\qquad + h_K \|\nabla V\|_{L^\infty(\Bar{\Omega})} \left( \frac{|K|}{d+1} \sum_{\mathbf x_i\in K} (a+\mathbf b^\top\mathbf x_i)^2 + \int_K (a+\mathbf b^\top\mathbf x)^2 \mathrm{d}\mathbf x \right)\\ 
        \leq &  \|V\|_{L^\infty(\Bar{\Omega})}  E_1 h |K| \|\mathbf b\|^2 + 2  h_K \|\nabla V\|_{L^\infty(\Bar{\Omega})} \frac{|K|}{d+1} \sum_{\mathbf x_i\in K} (a+\mathbf b^\top\mathbf x_i)^2,
    \end{align*} 
    where we have used  \eqref{appendixc-estimate1} and \eqref{appendixc-estimate2} in the last step. By summing over $K$, we have proven (ii). Since $\nabla \Pi\mathbf u$ is a constant over $K$, the quadrature for $(\nabla \Pi\mathbf u)^2$ is exact, i.e., $\langle \nabla \Pi\mathbf u (\bx), \nabla \Pi\mathbf u (\bx)  \rangle=||\nabla \Pi\mathbf u ||^2_{L^2(\Omega)}$, thus by \eqref{appendixc-esitmate3} the result (iii) is true.
    Finally, with a similar argument as in \eqref{eq:Jenson_simplex_K}, the result (iv) is also a consequence of Jensen's inequality since $[\Pi\mathbf u(\bx)]^4$ is a convex function on~$K$.
\end{proof}

\begin{lemma}\label{lem:h1PU}
    Consider $P^1$ finite element method with quadrature on an unstructured simplicial mesh with shape regularity in dimension $d=1,2,3$. For any $h>0$ and $u\in H^2(\Omega) \cap H_0^1(\Omega)$, it holds for some constant $E_4>0$ independent of the mesh size that $\left| |u|^2_{H_0^1(\Omega)} - \langle -\Delta_h\mathbf u,\mathbf u\rangle_h \right| \leq E_4 h \|u\|_{H^2(\Omega)}^2$, where $\mathbf u\in\bR^N$ is the vector of  the values of $u$ at quadrature points. 
\end{lemma}
\begin{remark}
 Lemma \ref{lem:error_interpolation} holds for any dimension $d$, but Lemma \eqref{lem:h1PU} holds only for dimension $d\leq 3$ since we need the point values $u(\bx_i)$ being well defined, ensured by the Sobolev embedding $u\in H^2(\Omega) \cap H_0^1(\Omega)\subset C(\Omega)$, which holds only in dimension $d=1,2,3$.
\end{remark}

\begin{proof}
  By Lemma~\ref{lem:error_interpolation} (iii), we have
    \begin{align*}
        & \left||u|^2_{H_0^1(\Omega)} - \langle -\Delta_h\mathbf u,\mathbf u\rangle_h \right| = \left||u|^2_{H_0^1(\Omega)} - |\Pi\mathbf u|_{H_0^1(\Omega)}^2\right| \\
        \leq & |u-\Pi\mathbf u|_{H_0^1(\Omega)}|u+\Pi\mathbf u|_{H_0^1(\Omega)} \leq |u-\Pi\mathbf u|_{H_0^1(\Omega)}\left(2|u|_{H_0^1(\Omega)} + |u-\Pi\mathbf u|_{H_0^1(\Omega)}\right).
    \end{align*}
    We can conclude the proof after combining the estimates above and the standard interpolation estimate $|u-\Pi\mathbf u|_{H_0^1(\Omega)}=O(h) \cdot \| u\|_{H^2(\Omega)}$, see \cite{ciarlet2002finite}).
\end{proof}

\begin{lemma}\label{lem:L2error2}
    Under the same assumptions as in Lemma~\ref{lem:error_interpolation}, let $\Bar{\Pi}\mathbf u$ be the piecewise constant interpolation of $\mathbf u$, i.e., $(\Bar{\Pi}\mathbf u)(\mathbf x) = \frac{1}{|K|}\int_K (\Pi\mathbf u)(\mathbf x') \mathrm{d}\mathbf x'$ on any simplex $K$. There is some constant $E_5>0$ independent of the mesh size s.t. $\|\Pi\mathbf u - \Bar{\Pi}\mathbf u\|_{L^2(\Omega)}^2 \leq  E_5 h^2 \langle -\Delta_h \mathbf u, \mathbf u\rangle_h$. 
\end{lemma}

\begin{proof}
    Consider a simplex $K$ on which we assume $(\Pi\mathbf u)(\mathbf x) = a+\mathbf b^\top \mathbf x$ and $(\Bar{\Pi}\mathbf u)(\mathbf x) = a+ \frac{\mathbf b^\top}{3}\sum_{\mathbf x_i\in K} \mathbf x_i$. Then we have
    \begin{align*}
        & \int_K ((\Pi\mathbf u)(\mathbf x) - (\Bar{\Pi}\mathbf u)(\mathbf x))^2 = \int_K \left(\mathbf b^\top\mathbf x - \frac{\mathbf b^\top}{3}\sum_{\mathbf x_i\in K}\mathbf x_i\right)^2 \mathrm{d}\mathbf x \leq h_K^2 |K| \|\mathbf b\|^2 \leq h^2 |K| \|\mathbf b\|^2.
    \end{align*}  
By the same arguments as in the proof for  Lemma \ref{lem:error_interpolation},  the result holds with $E_5=1$.
\end{proof}

\subsection{Proof of consistency and eigengap}

\begin{theorem}\label{thm:consist}
    Consider $P^1$ finite element method with quadrature on an unstructured simplicial mesh with shape regularity in dimension $d=1,2,3$. Let $u^*$ and $\lambda^*$ be the ground state and the corresponding eigenvalue of the continuous energy $E$. Let $u_h^*$ and $\lambda_h^*$ be the ground state and the corresponding eigenvalue of the discrete energy $E_h$. Suppose that $u^*\in H^2(\Omega)\cap H_0^1(\Omega)$ and that $V\in C^1(\Bar{\Omega})$. Then $u_h^*\to u^*$ in $H_0^1(\Omega)$ and $\lambda_h^*\to\lambda^*$ as $h\to 0$.
\end{theorem}  

\begin{proof}
    By our previous notation, $\mathbf u^*\in\bR^N$ is the vector collecting the values of $u_h^*$ at quadrature points and $\Pi(\mathbf u^*)=u_h^*$. For convenience, define $ {\mathbf u}^{exact}\in\bR^N$ as the vector representing the values of $u^*$ at quadrature points. Due to the embedding $H^2(\Omega)\hookrightarrow C^0(\Omega)$ for $d=1,2,3$ and the convergence of Riemannian integral for continuous functions, one has $\langle {\mathbf u}^{exact},{\mathbf u}^{exact} \rangle_h = \|u^*\|_{L^2(\Omega)}^2 + o(1) = 1+o(1)$, $\langle {\mathbf u}^{exact},\mathbb V {\mathbf u}^{exact}\rangle_h = \int V|u^*|^2 + o(1)$, and $\langle (({\mathbf u}^{exact})^2,({\mathbf u}^{exact})^2)\rangle_h = \int |u^*|^4 + o (1)$. Furthermore, thanks to Lemma~\ref{lem:h1PU}, we have $\mathbf E_h({\mathbf u}^{exact}) = E(u^*) + o(1)$. Then it holds that
    \begin{equation}\label{eq:error_energy1}
        \mathbf E_h(\mathbf u^*)\leq \mathbf E_h\left(\frac{{\mathbf u}^{exact}}{\langle {\mathbf u}^{exact},{\mathbf u}^{exact} \rangle_h}\right) = \mathbf E_h({\mathbf u}^{exact}) + o(1) = E(u^*) + o(1),
    \end{equation}
    which implies that $\langle -\Delta_h\mathbf u^*,\mathbf u^*\rangle_h$ is bounded as $h\to 0$.

    By Lemma~\ref{lem:error_interpolation} (i) and boundedness of $\langle -\Delta_h\mathbf u^*,\mathbf u^*\rangle_h$, one has $\norm{u_h^*}_{L^2(\Omega)} = 1 + o(1)$.  Therefore, it follows from Lemma~\ref{lem:error_interpolation} (ii)-(iv) that
    \begin{equation}\label{eq:error_energy2}
        E_h(u^*_h)=\mathbf E_h(\mathbf u^*) \geq E(u_h^*) + o(1) = E\left(\frac{u_h^*}{\norm{u_h^*}_{L^2(\Omega)}}\right) + o(1) \geq E(u^*) + o(1).
    \end{equation}
    Then we can conclude from \eqref{eq:error_energy1} and \eqref{eq:error_energy2} that $E(u^*) = E\left(\frac{u_h^*}{\norm{u_h^*}_{L^2(\Omega)}}\right) + o(1). $

    By compact embeddings $H_0^1(\Omega)\subset\subset L^2(\Omega)$ and $H_0^1(\Omega)\subset\subset L^4(\Omega)$, there exists a subsequence of the bounded sequence $\left\{\frac{u_h^*}{\norm{u_h^*}_{L^2(\Omega)}}\right\}$  converging weakly in $H_0^1(\Omega)$ and strongly in $L^2(\Omega)$ and $L^4(\Omega)$ to some function in $\in H_0^1(\Omega)$. Without loss of generality, we assume that the whole sequence converges:
    \begin{equation*}
        \frac{u_h^*}{\norm{u_h^*}_{L^2(\Omega)}}\rightarrow u_0,\quad\text{weakly in } H_0^1(\Omega)\text{ and strongly in }L^2(\Omega)\text{ and }L^4(\Omega).
    \end{equation*}
    Since the week convergence guarantees that $\|u^0\|_{H_0^1(\Omega)}\leq \liminf_{h\to 0} \norm{\frac{u_h^*}{\norm{u_h^*}_{L^2(\Omega)}}}_{H_0^1(\Omega)}$,  
    \begin{equation*}
        E(u^*)\leq E(u^0)\leq \lim_{h\to 0} E\left(\frac{u_h^*}{\norm{u_h^*}_{L^2(\Omega)}}\right) = E(u^*),
    \end{equation*}
    which implies that $E(u^0) = E(u^*)$ and hence that $u^0 = u^*$ by the uniqueness of positive ground state. Then we know that $\frac{u_h^*}{\norm{u_h^*}_{L^2(\Omega)}}\to u^*$ weakly in $H_0^1(\Omega)$ and $\norm{\frac{u_h^*}{\norm{u_h^*}_{L^2(\Omega)}}}_{H_0^1(\Omega)} \to \|u^*\|_{H_0^1(\Omega)}$. Note that the weak convergence and convergence of norm imply the strong convergence. Thus   $\frac{u_h^*}{\norm{u_h^*}_{L^2(\Omega)}}\to u^*$ and $u_h^*\to u^*$ strongly in $H_0^1(\Omega)$.

    In the discussion above, we have shown that every subsequence of $\{u_h^*\}$ has a subsequence that converges strongly in $H_0^1(\Omega)$ to $u^*$. Therefore, the whole sequence $\{u_h^*\}$ converges strongly in $H_0^1(\Omega)$ to $u^*$. This immediately implies that $\lambda_h^*\to \lambda^*$.
\end{proof}

\begin{assumption}\label{eq:continuous_eigengap}
    We assume that the multiplicity of the eigenvalue $\lambda^*$ to the following problem is one: $ -\Delta u + Vu +\beta |u^*|^2 u = \lambda u, \quad u=0\ \text{on}\ \partial\Omega$.
\end{assumption}

We remark that the smoothness of $u^*$ depends on the smoothness of $V$ \cite{lieb2001bosons}, and the assumption above can be proven with suitable smoothness assumptions on $V$ and $u^*$, e.g.,  the smallest eigenvalue of an elliptic operator $-\Delta  + V +\beta |u^*|^2 $ is simple due to the positivity and the maximum principle of the operator $(-\Delta  + V +\beta |u^*|^2)^{-1}$, implied by the Krein-Rutman Theorem (the extension of Perron-Frobenius Theorem from positive matrices to positive compact linear operators), see \cite{du2006order}*{Chapter 1}. The spectrum, including simplicity of the smallest eigenvalue, of  Schr\"odinger operators like $-\Delta  + V +\beta |u^*|^2 $, was also thoroughly studied in \cite{reed1972methods}*{Section XIII.12}.

\begin{theorem}\label{thm:positive_eigengap}
    Suppose that Assumption~\ref{eq:continuous_eigengap} and assumptions made in Theorem~\ref{thm:consist} hold. Let $(\lambda_h^1, \mathbf u^1)$ with $\langle\mathbf u^1,\mathbf u^1\rangle_h = 1$ be the second eigenpair of the linearized problem $-\Delta_h \mathbf u +  \mathbb V \mathbf u +\beta (\mathbf u^*)^2 \mathbf u = \lambda_h \mathbf u$. Then $ \liminf_{h\to 0} \lambda_h^1 - \lambda_h^* > 0$.
\end{theorem}

We remark that $\lambda_h^*$ in Theorem~\ref{thm:positive_eigengap} is exactly $\lambda_h^0$ in Assumption~\ref{asp:ustar}. In this sense Theorem~\ref{thm:positive_eigengap} validates Assumption~\ref{asp:ustar} with a positive lower bound for the eigengap that is independent of the mesh size $h$. % 

\begin{proof}[Proof of Theorem~\ref{thm:positive_eigengap}]
    Suppose that there is no positive eigengap. Then $\lambda_h^1\to \lambda^*$ when passing to a subsequence. Notice that
    \begin{equation}\label{eq:energy_uh1}
        \langle -\Delta_h \mathbf u^1,\mathbf u^1\rangle_h +  \langle \mathbb V \mathbf u^1,\mathbf u^1\rangle_h +\beta \langle (\mathbf u^*)^2 \mathbf u^1,\mathbf u^1\rangle_h = \lambda_h^1,
    \end{equation}
    which implies the boundedness of $|u_h^1|_{H_0^1(\Omega)} \leq \sqrt{\langle -\Delta_h \mathbf u^1,\mathbf u^1\rangle_h}$, where $u_h^1 = \Pi\mathbf u^1$ and we used Lemma~\ref{lem:error_interpolation} (iii). There exists $u^1\in H_0^1(\Omega)$ such that $u_h^1$ converges to $u^1$, a.e., weakly in $H_0^1(\Omega)$, and strongly in $L^2(\Omega)$ and $L^4(\Omega)$, when passing to another subsequence.

    Let us analyze the limiting behavior of each term in \eqref{eq:energy_uh1}. It follows from the weak convergence $u_h^1\to u^1$ in $H_0^1(\Omega)$ that
    \begin{equation*}
        \liminf_{h\to 0} \langle -\Delta_h \mathbf u^1,\mathbf u^1\rangle_h \geq \liminf_h \|u_h^1\|_{H_0^1(\Omega)}^2\geq \|u^1\|_{H_0^1(\Omega)}^2.
    \end{equation*}
    By Fatou's lemma and Lemma~\ref{lem:error_interpolation} (ii), one has
    \begin{equation*}
        \liminf_{h\to 0} \langle \mathbb V \mathbf u^1,\mathbf u^1\rangle_h = \liminf_{h\to 0} \int V |u_h^1|^2 \geq \int V|u^1|^2.
    \end{equation*}
    Let $\Bar{u}_h^* = \Bar{\Pi} \mathbf u^*$ and $\Bar{u}_h^1 = \Bar{\Pi} \mathbf u^1$ be the piecewise constant approximation of $\mathbf u^*$ and $\mathbf u^1$ in the sense of the one in Lemma~\ref{lem:L2error2}, then 
    \begin{equation*}
        \lim_{h\to 0} \|\Bar{u}_h^* - u_h^*\|_{L^2(\Omega)} = 0\quad\text{and}\quad \lim_{h\to 0} \|\Bar{u}_h^1 - u_h^1\|_{L^2(\Omega)} = 0,
    \end{equation*}
    which implies $\Bar{u}_h^* - u_h^* \to 0$ and $\Bar{u}_h^1 - u_h^1 \to 0$ a.e. when passing to a subsequence. By Fatou's lemma, it holds that
    \begin{equation*}
        \liminf_{h\to 0}  \langle (\mathbf u^*)^2 \mathbf u^1,\mathbf u^1\rangle_h = \liminf_{h\to 0} \int |\Bar{u}_h^*|^2 |\Bar{u}_h^1|^2 \geq \int |u^*|^2 |u^1|^2.
    \end{equation*}
    Then taking limit for \eqref{eq:energy_uh1} as $h\to 0$, one obtains that
    \begin{equation*}
        \lambda^* \geq \int |\nabla u^1|^2 +  V |u^1|^2 +  \beta|u^*|^2 |u^1|^2 \geq \lambda^* \int |u^1|^2 = \lambda^*,
    \end{equation*}
    where the last equality follows from $\norm{u_h^1}_{L^2(\Omega)} = 1 + o(1)$ by Lemma~\ref{lem:error_interpolation} (i) and $u_h^1\to u^1$ strongly in $L^2(\Omega)$. This implies that $u^1$ is also a ground state of the linearized problem at $u^*$. However, one has from Lemma~\ref{lem:error_interpolation} (i) that
    \begin{equation*}
        0 = \langle \mathbf u^*,\mathbf u^1\rangle_{h} = 1-\frac{1}{2}\|\mathbf u^* - \mathbf  u_h^1\|_2^2 =  1-\frac{1}{2}\|u_h^* - u_h^1\|_{L^2(\Omega)}^2 + O(h^2),
    \end{equation*}
    which implies that $0 = 1-\frac{1}{2}\|u^* - u^1\|_{L^2(\Omega)}^2 =  (u^*,u^1)_{L^2(\Omega)}$, leading to a contradiction.
\end{proof}

\end{document}